\documentclass[11pt]{amsart}
\usepackage[margin=1in]{geometry}
\usepackage{amsmath,amssymb,amsthm,mathtools}
\usepackage{booktabs,array,longtable}
\usepackage{enumerate}
\usepackage{mathrsfs}
\usepackage[colorlinks=true,linkcolor=blue,citecolor=blue,urlcolor=blue]{hyperref}
\allowdisplaybreaks

\newtheorem{theorem}{Theorem}[section]

\newtheorem{corollary}[theorem]{Corollary}

\newtheorem{lemma}[theorem]{Lemma}
\newtheorem{proposition}[theorem]{Proposition}

\numberwithin{equation}{section}

\newcommand{\eps}{\varepsilon}

\newcommand{\HH}{\mathcal H}
\newcommand{\la}{\lambda}

\newcommand{\E}{\mathcal E}

\newcommand{\ls}{\lesssim}
\newcommand{\gs}{\gtrsim}
\newcommand{\LL}{\mathcal L}
\newcommand{\II}{\mathrm{II}}

\newcommand{\R}{\mathbb R}

\newcommand{\supp}{\operatorname{supp}}

\newcommand{\sig}{\operatorname{sig}}
\newcommand{\rank}{\operatorname{rank}}
\newcommand{\diam}{\operatorname{diam}}
\newcommand{\lth}{\operatorname{Length}}

\newcommand{\Sym}{\operatorname{Sym}}

\newcommand{\norm}[1]{\left\lVert #1\right\rVert}
\newcommand{\abs}[1]{\left\lvert #1\right\rvert}

\begin{document}
\title[sharp endpoint multilinear estimates]{Sharp Endpoint Multilinear Estimates for Oscillatory Integrals and Spectral Clusters}
\author{Shengwen Gan}
\address{Department of Mathematics, Sun Yat-sen University, Guangzhou, 510275, P.R. China}
\email{shengwengan2018@gmail.com}

\author{Cheng Zhang}
\address{Yau Mathematical Sciences Center, Tsinghua University, Beijing, 100084, P.R. China}
\email{czhang98@tsinghua.edu.cn}

\author{Zhifei Zhu}
\address{Yau Mathematical Sciences Center, Tsinghua University, Beijing, 100084, P.R. China}
\email{zhifeizhu@tsinghua.edu.cn} 
	\date{}
	\begin{abstract}
		We prove sharp $k$-linear $L^p$ estimates for Carleson--Sj\"olin oscillatory integral operators with arbitrary separated frequency scales for all $k\ge 2$ and $1\le p\le \infty$.  The estimates are sharp, including the endpoint logarithmic behavior for general Carleson--Sj\"olin phases.  Moreover,  we obtain log-free endpoint bilinear spectral cluster estimates on every closed three-dimensional Riemannian manifold, resolving a problem of Burq--G\'erard--Tzvetkov \cite{bgt2004,bgt2005}.  As a consequence, we establish sharp $k$-linear $L^p$ spectral cluster estimates  for all $k\ge 2$ and $1\le p\le \infty$.
		
	\end{abstract}
	\maketitle

    \section{Introduction}\label{sec:main-estimates}
	Oscillatory integral operators with curvature form a central bridge between Fourier restriction and spectral cluster estimates. In this paper, we prove sharp multilinear estimates for such operators in the full range of Lebesgue exponents ($1\le p\le \infty$) and frequency configurations relevant to products of spectral clusters.  One of our results shows that the endpoint behavior in three dimensions depends on a genuine geometric distinction: spectral clusters admit  log-free endpoint estimates, while general  Carleson--Sj\"olin oscillatory integral operators may exhibit an unavoidable logarithmic loss.
	
	To motivate the question, we first introduce the model case which is formulated in terms of the Fourier extension operator. Let $n\ge 2$. For $\la>1$, define
	\[ E_\la f(y):=\la^{\frac{n-1}{2}}\int_{[0,1]^{n-1}}e^{i\la  (x\cdot\xi+t|\xi|^2)}f(\xi)d\xi, \]
	where $y=(t,x)\in\mathbb{R}\times \mathbb{R}^{n-1},\xi\in\mathbb{R}^{n-1}$. Given frequency scales $1\le \la_1\le \la_2\le \dots\le \la_k$, we study estimates of the form
	\begin{equation}\label{modelklinear}
		\left\|\prod_{j=1}^kE_{\la_j}f_j\right\|_{L^p(B^n(0,1))}\lesssim C(n,k,p;\la_1,\dots,\la_k)\prod_{j=1}^k\|f_j\|_{L^2([0,1]^{n-1})}.
	\end{equation}
	We refer to such an estimate as the \textit{$k$-linear restriction estimate with separated frequency scales}. This type of estimate is fundamentally different from the $k$-linear estimate developed by Bennett--Carbery--Tao \cite{bct2006} and others; see Section~\ref{subsec:remarks-exponents}.  We shall study \eqref{modelklinear} in a more general setting, with the Fourier extension operators replaced by Carleson--Sj\"olin oscillatory integral operators.
	
	Let $n\ge2$.  A Carleson--Sj\"olin operator has the form
	\[
	T_{\la,\phi}^a f(y):=\la^{\frac{n-1}2}\int_{\mathbb{R}^{n-1}}e^{i\la\phi(y,\xi)}a(y,\xi)f(\xi)d\xi,
	\]
	where $y=(t,x)\in \mathbb{R}\times \mathbb{R}^{n-1}$, $\xi\in \mathbb{R}^{n-1}$, $a\in C_0^\infty(\mathbb{R}\times \mathbb{R}^{n-1}\times \mathbb{R}^{n-1})$, and $\phi\in C^\infty(\mathbb{R}\times\mathbb{R}^{n-1}\times\mathbb{R}^{n-1})$ is real-valued.  We assume throughout that each phase satisfies the \textbf{Carleson--Sj\"olin condition}. For each operator, this means that the phase $\phi(y,\xi)$ satisfies the nondegeneracy condition:
	\begin{equation}
		\text{rank}\frac{\partial^2\phi}{\partial \xi\partial y}=n-1
	\end{equation}
	and the curvature condition: for each $y$, the hypersurface
	\begin{equation}\label{sphi}
		S_\phi(y)=\{\nabla_{y}\phi(y,\xi):\xi\in \text{supp}_\xi a\}
	\end{equation}
	has everywhere nonvanishing Gaussian curvature.  
	
	We say that the phase satisfies the \textbf{elliptic Carleson--Sj\"olin condition} if the curvature condition is replaced by the stronger elliptic condition that the second fundamental form of $S_\phi(y)$ is uniformly definite.  Equivalently, writing
	\[
	P(y,\xi)=\nabla_{y}\phi(y,\xi),
	\quad A_j(y,\xi)=\partial_{\xi_j}P(y,\xi),\quad N=\frac{\bigwedge_{j=1}^{n-1}A_j}{|\bigwedge_{j=1}^{n-1}A_j|},
	\]
	the matrix 
	\[\big\langle \partial^2_{\xi\xi} P(y,\xi),N(y,\xi)\big\rangle\]
	is definite for all $(y,\xi)\in\operatorname{supp} a$.  The model elliptic phase is $\phi(t,x,\xi)=x\cdot\xi+\frac12t|\xi|^2$.  The local Fourier integral operators arising in Sogge's spectral cluster parametrix for Laplace eigenfunctions on smooth closed manifolds also satisfy this elliptic Carleson--Sj\"olin condition. See \cite[Section 5]{fio} and Lemma~\ref{ellip}.
    
	\subsection{Main results}
	For $n\ge3$ define
	\[
	p_0=\frac{n+1}{n-1},\qquad
	p_1=\frac{n-1}{n-2},\qquad
	p_2=\frac{n+1}{n},\qquad
	p_3=\frac{2(n+1)}{3(n-1)}.
	\]
	
	We first present the bilinear estimates. This case already contains the main endpoint phenomenon: in dimension three, general Carleson--Sj\"olin phases have a sharp logarithmic loss at $p=2$.
	
	\begin{theorem}[Bilinear estimates]\label{thm1}
		Let $n\ge2$, $1\le p\le2$ and $1\ll\lambda\le\mu$. Let $T_{\lambda,\phi_1}^{a_1}$ and $T_{\mu,\phi_2}^{a_2}$ be Carleson--Sj\"olin operators as above. Then
		\begin{equation}\label{bie}
			\|T_{\lambda,\phi_1}^{a_1} f\,T_{\mu,\phi_2}^{a_2} g\|_{L^p(\mathbb R\times\mathbb R^{n-1})}
			\ls\,\E(n,2,p;\lambda,\mu)\|f\|_{L^2(\mathbb{R}^{n-1})}\|g\|_{L^2(\mathbb{R}^{n-1})}.
		\end{equation}
		
		If $n\ge4$, then
		\[
		\E(n,2,p;\lambda,\mu)=
		\begin{cases}
			\lambda^{\frac{n-1}{2}(1-\frac1p)},
			&1\le p\le p_1,\\[2mm]
			\displaystyle
			\min\left\{
			\lambda^{1/2}\mu^{\frac{n-2}{2}-\frac{n-1}{2p}},
			\lambda^{n-\frac32-\frac{n-1}{p}}
			\right\},
			&p_1\le p\le p_0,\\[4mm]
			\displaystyle
			\min\left\{
			\lambda^{\frac{2n-1}{2}-\frac{n+1}{p}}\mu^{\frac1p-\frac12},
			\lambda^{n-\frac32-\frac{n-1}{p}}
			\right\},
			&p_0\le p\le2.
		\end{cases}
		\]
		In the last two ranges, if $\mu\le \la^2$, then the minimum takes the first term. 
		
		If  $(n,p)=(3,2)$, then 
		\begin{equation}
			\E(n,2,p;\lambda,\mu)=\la^{\frac12}\sqrt{\log\la}.
		\end{equation}
		The logarithmic factor is sharp. 
		
		If $n=2,3$ and $(n,p)\ne (3,2)$, then
		\begin{equation}
			\E(n,2,p;\lambda,\mu)=\lambda^{\frac{n-1}{2}(1-\frac1p)}.
		\end{equation}

		Moreover, the estimates are sharp.
	\end{theorem}
	
	Having established the bilinear case, we now extend these results to products of an arbitrary number of functions. The following theorem provides sharp multilinear estimates for $k\ge 3$ frequency scales in the range $1\le p\le 2$.
	
	\begin{theorem}[Multilinear estimates for $ p\le2$]\label{thm2}
		Let $n\ge2$, $k\ge3$, $1\le p\le2$, $1\ll\lambda_1\le\cdots\le\lambda_k$. Let $T_{\lambda_j,\phi_j}^{a_j}$, $1\le j\le k$, be Carleson--Sj\"olin operators as above.  Then
		\begin{equation}\label{goal}
			\left\|\prod_{j=1}^k T_{\lambda_j,\phi_j}^{a_j}f_j\right\|_{L^p(\mathbb R\times\mathbb R^{n-1})}
			\ls\,\E(n,k,p;\lambda_1,\ldots,\lambda_k)\prod_{j=1}^k\|f_j\|_{L^2(\mathbb{R}^{n-1})} .
		\end{equation}
		
		If $n\ge5$, then
		\[
		\E =
		\begin{cases}
			\prod_{j=1}^{k-4}\la_j^{\frac{n-1}{2}}\cdot
			\min\Big\{
			\la_{k-3}^{\frac{n-1}{2}}
			\la_{k-2}^{\frac{3n-2}{2}-\frac{n+1}{p}}
			\la_{k-1}^{\frac1p-\frac12},\ \la_{k-3}^{\frac{2n-1}{2}-\frac{n+1}{2p}}
			\la_{k-2}^{\frac{n-2}{2}}
			\la_{k-1}^{\frac{n-1}{2}-\frac{n-1}{2p}}
			\Big\},
			&  1\le p\le p_2,\\
			\prod_{j=1}^{k-3}\la_j^{\frac{n-1}{2}}\cdot
			\la_{k-2}^{n-1-\frac{n+1}{2p}}
			\la_{k-1}^{\frac{n-1}{2}-\frac{n-1}{2p}},
			&  p_2\le p\le p_1,
			\\
			\prod_{j=1}^{k-3}\la_j^{\frac{n-1}{2}}\cdot
			\min\Big\{
			\la_{k-2}^{n-1-\frac{n+1}{2p}}
			\la_{k-1}^{\frac12}
			\la_k^{\frac{n-2}{2}-\frac{n-1}{2p}},\ 
			\la_{k-2}^{\frac n2-\frac1p}
			\la_{k-1}^{n-\frac32-\frac{n-1}{p}}
			\Big\},
			&  p_1\le p\le p_0,
			\\
			\prod_{j=1}^{k-3}\la_j^{\frac{n-1}{2}}\cdot
			\min\Big\{
			\la_{k-2}^{\frac{n-1}{2}}
			\la_{k-1}^{\frac{2n-1}{2}-\frac{n+1}{p}}
			\la_k^{\frac1p-\frac12},\ 
			\la_{k-2}^{\frac n2-\frac1p}
			\la_{k-1}^{\frac{2n-3}{2}-\frac{n-1}{p}}
			\Big\},
			& p_0\le p\le 2.
		\end{cases}
		\]
		In the first range, if $\la_{k-3}\la_{k-1}\le \la_{k-2}^2$ then the minimum takes the first term.  In the last two ranges, if $\la_{k-2}\la_{k}\le \la_{k-1}^2$ then the minimum takes the first term.
		
		If $n = 4$, then we have an extra endpoint $p_3=10/9$ and 
		\[
		\E=
		\begin{cases}
			\prod_{j=1}^{k-4}\la_j^{\frac32}\cdot
			\min\Big\{
			\la_{k-3}^{\frac{15}{4}-\frac{5}{2p}}
			\la_{k-2}^{\frac12}
			\la_{k-1}^{\frac74-\frac{3}{2p}},\ 
			\la_{k-3}^{\frac72-\frac{5}{2p}}
			\la_{k-2}
			\la_{k-1}^{\frac32-\frac{3}{2p}}
			\Big\},
			&  1\le p\le \frac{10}{9},
			\\
			\prod_{j=1}^{k-4}\la_j^{\frac32}\cdot
			\min\Big\{
			\la_{k-3}^{\frac32}
			\la_{k-2}^{5-\frac5p}
			\la_{k-1}^{\frac1p-\frac12},
			\,
			\la_{k-3}^{\frac72-\frac{5}{2p}}
			\la_{k-2}
			\la_{k-1}^{\frac32-\frac{3}{2p}}
			\Big\},
			&  \frac{10}{9}\le p\le \frac54,
			\\
			\prod_{j=1}^{k-3}\la_j^{\frac32}\cdot
			\la_{k-2}^{3-\frac{5}{2p}}
			\la_{k-1}^{\frac32-\frac{3}{2p}},
			&  \frac54\le p\le \frac32,
			\\
			\prod_{j=1}^{k-3}\la_j^{\frac32}\cdot
			\min\Big\{
			\la_{k-2}^{3-\frac{5}{2p}}
			\la_{k-1}^{\frac12}
			\la_k^{1-\frac{3}{2p}},
			\,
			\la_{k-2}^{2-\frac1p}
			\la_{k-1}^{\frac52-\frac3p}
			\Big\},
			&  \frac32\le p\le \frac53,
			\\
			\prod_{j=1}^{k-3}\la_j^{\frac32}\cdot
			\min\Big\{
			\la_{k-2}^{\frac32}
			\la_{k-1}^{\frac72-\frac5p}
			\la_k^{\frac1p-\frac12},
			\,
			\la_{k-2}^{2-\frac1p}
			\la_{k-1}^{\frac52-\frac3p}
			\Big\},
			& \frac53\le p\le 2.
		\end{cases}
		\]
		In the first two ranges, if $\la_{k-3}\la_{k-1}\le \la_{k-2}^2$ then the minimum takes the first term.  In the last two ranges, if $\la_{k-2}\la_{k}\le \la_{k-1}^2$ then the minimum takes the first term.
		
		If $n=3$, then $p_0=p_1=2$, $p_2=p_3=4/3$, and
		\[
		\E=
		\begin{cases}
			\E_0 \sqrt{\log\la_{k-2}}\sqrt{\log\la_{k-3}} ,& p=1,\\
			\E_0 \sqrt{\log\la_{k-2}}&  1< p\le \frac43,
			\\
			\E_0
			&  \frac43< p< 2,\\
			\E_0\sqrt{\log \la_{k-1}},& p=2,
		\end{cases}
		\] with 
	\begin{equation}\label{Eps0}
		\E_0=
		\begin{cases}
			\prod_{j=1}^{k-4}\lambda_j\cdot 
			\lambda_{k-3}^{\frac52-\frac{2}{p}}\,
			\lambda_{k-2}^{\frac12}\,
			\lambda_{k-1}^{1-\frac1p} ,
			&  1\le p\le \frac43,
			\\
			\prod_{j=1}^{k-3}\lambda_j\cdot 
			\lambda_{k-2}^{2-\frac{2}{p}}\,
			\lambda_{k-1}^{1-\frac1p},
			&  \frac43< p\le 2,
		\end{cases}
	\end{equation}
		The logarithmic factors are sharp. 

		If $n=2$, then
		\[\E=\begin{cases}
			\prod_{j=1}^{k-6}\lambda_j^{\frac12}\cdot
			\lambda_{k-5}^{\,\frac74-\frac{3}{2p}}\,
			\lambda_{k-4}^{\frac14}\,
			\lambda_{k-3}^{\frac14}\,
			\lambda_{k-2}^{\frac14}\,
			\lambda_{k-1}^{\,\frac12-\frac{1}{2p}}, & 1\le p\le \frac65,\\
			\prod_{j=1}^{k-5}\lambda_j^{\frac12}\cdot
			\lambda_{k-4}^{\,\frac32-\frac{3}{2p}}\,
			\lambda_{k-3}^{\frac14}\,
			\lambda_{k-2}^{\frac14}\,
			\lambda_{k-1}^{\,\frac12-\frac{1}{2p}},& \frac65\le p\le \frac32,\\
			\prod_{j=1}^{k-4}\lambda_j^{\frac12}\cdot
			\lambda_{k-3}^{\,\frac54-\frac{3}{2p}}\,
			\lambda_{k-2}^{\frac14}\,
			\lambda_{k-1}^{\,\frac12-\frac{1}{2p}},& \frac32\le p\le 2.
		\end{cases} \]
		The convention $\lambda_j=2$ for $j\le0$ keeps the displayed expression meaningful for small values of $k$. 
		
		Moreover, the estimates are sharp.
	\end{theorem}

	The preceding theorems  cover the range $1\le p\le2$.  For $p>2$, the multilinear estimates \eqref{goal} continue to hold with the following simpler expressions for $\E$. 
	
	\begin{theorem}[Multilinear estimates for $p>2$]\label{thm3}
		Let $n\ge2$, $k\ge2$, $p>2$, $1\ll\lambda_1\le\cdots\le\lambda_k$, and $p_c=\frac{2(n+1)}{n-1}.$	Let $T_{\lambda_j,\phi_j}^{a_j}$, $1\le j\le k$, be Carleson--Sj\"olin operators as above. 
		
		If $n\ge3$, then we have \eqref{goal} with
		\[
		\E =
		\begin{cases}
			\prod_{j=1}^{k-2}\lambda_j^{\frac{n-1}{2}}\cdot
			\lambda_{k-1}^{\frac{3(n-1)}4-\frac{n+1}{2p}}
			\lambda_k^{\frac{n-1}4-\frac{n-1}{2p}},
			& 2< p\le p_c,\\
			\prod_{j=1}^{k-1}\lambda_j^{\frac{n-1}{2}}\cdot
			\lambda_k^{\frac{n-1}{2}-\frac np},
			& p_c\le p\le\infty.
		\end{cases}
		\]
		
		If $n=2$, then we have \eqref{goal} with
		\[
		\E =
		\begin{cases}
			\prod_{j=1}^{k-3}\lambda_j^{\frac12}\cdot
			\lambda_{k-2}^{1-\frac{3}{2p}}
			\lambda_{k-1}^{\frac14}
			\lambda_k^{\frac14-\frac{1}{2p}},
			& 2< p\le3,\\
			\prod_{j=1}^{k-2}\lambda_j^{1/2}\cdot
			\lambda_{k-1}^{\frac34-\frac{3}{2p}}
			\lambda_k^{\frac14-\frac{1}{2p}},
			& 3\le p\le6,\\
			\prod_{j=1}^{k-1}\lambda_j^{\frac12}\cdot
			\lambda_k^{\frac12-\frac2p},
			& 6\le p\le\infty.
		\end{cases}
		\]
		The convention $\lambda_j=2$ for $j\le0$ keeps the displayed expression meaningful for small values of $k$. 
		
			Moreover, the estimates are sharp.
	\end{theorem} 
	Throughout, we denote the power expression in $\mathcal{E}(n,k,p;\la_1,\ldots,\la_k)$ by $\mathcal{E}_0(n,k,p;\la_1,...,\la_k)$. Equivalently,  $\E_0$ is given by \eqref{Eps0} if $n=3$ and $p\le 2$, and $\E_0=\E$ otherwise.  We establish sharp $k$-linear spectral cluster estimates for all $k\ge2$ and $1\le p\le \infty$. In particular, this resolves the problem raised by Burq--G\'erard--Tzvetkov regarding the logarithmic loss in three dimensions.
	
	\begin{theorem}[Multilinear spectral cluster estimates]\label{thm4}
		Let $(M,g)$ be a closed smooth Riemannian manifold of dimension $n\ge2$. Let $\Delta$ denote the Laplace operator associated to the metric $g$.  Let $\chi\in\mathcal{S}(\mathbb{R})$ and $\chi_\la=\chi(\sqrt{-\Delta}-\la)$. Let $k\ge2$, $1\le p\le\infty$ and $1\ll\lambda_1\le\cdots\le\lambda_k$.  Then
		\begin{equation}\label{eig}
			\left\|\prod_{j=1}^k \chi_{\lambda_j}f_j\right\|_{L^p(M)}
			\ls\,\E_0(n,k,p;\lambda_1,\ldots,\lambda_k)\prod_{j=1}^k\|f_j\|_{L^2(M)} .
		\end{equation}
		Moreover, the estimates are sharp.
	\end{theorem}
	
The bilinear and trilinear spectral cluster estimates for $p=2$ were proved by Burq--G\'erard--Tzvetkov \cite{bgtinv,bgt2004,bgt2005}, with the logarithmic loss in three dimensions.	The bilinear quasimode estimates for $p>2$ were proved by Guo-Han-Tacy \cite{ght}.

	\subsection{Background and relation to earlier work}\label{subsec:remarks-exponents}
	
	The linear versions of our oscillatory integral estimates $(\lambda_1=\cdots=\lambda_k)$ go back to the foundational work of H\"ormander \cite{hor1960,hor1973}, Carleson--Sj\"olin \cite{cs1972}, and Stein \cite{stein1986}.  Under the Carleson--Sj\"olin condition one has
	\begin{equation}\label{linosc}
		\|T_{\lambda,\phi}^a f\|_{L^p(\mathbb R\times\mathbb R^{n-1})}
		\lesssim \lambda^{\sigma(p)}\|f\|_{L^2(\mathbb R^{n-1})},
	\end{equation}
	where
	\begin{equation}
		\sigma(p)=
		\begin{cases}
			\frac{n-1}{2}\left(\frac12-\frac1p\right),&2\le p\le p_c,\\
			\frac{n-1}{2}-\frac np,&p_c\le p\le\infty,
		\end{cases}
		\qquad p_c=\frac{2(n+1)}{n-1}.
	\end{equation}
	This circle of results is closely connected with the restriction problem, curvature-driven oscillatory integrals, and multilinear harmonic analysis; see, for example, Wolff \cite{wolff}, Tao \cite{tao2001,tao2003,tao2004}, Bennett--Carbery--Tao \cite{bct2006}, Lee \cite{lee2006}, Bourgain--Guth \cite{bourgain-guth2011}, and Hani \cite{hani2012}.
	
	The corresponding linear spectral cluster estimates are due to Sogge \cite{soggeduke,sogge88,fio}.  If $(M,g)$ is a closed smooth Riemannian manifold, $\Delta$ is the Laplace operator associated to the metric $g$, $\chi\in\mathcal S(\mathbb R)$, and $\chi_\lambda=\chi(\sqrt{-\Delta}-\lambda)$, then
	\begin{equation}\label{sogge}
		\|\chi_\lambda f\|_{L^p(M)}
		\lesssim \lambda^{\sigma(p)}\|f\|_{L^2(M)},\qquad 2\le p\le\infty.
	\end{equation}
	
	Multilinear $L^2$ estimates for eigenfunctions were first established by Burq--G\'erard--Tzvetkov \cite{bgtinv,bgt2004,bgt2005}.  In particular, for $1\ll\lambda\le\mu$ they proved
	\begin{equation}\label{bgt}
		\|\chi_\lambda f\,\chi_\mu g\|_{L^2(M)}
		\lesssim \mathcal B(n,\lambda)\|f\|_{L^2(M)}\|g\|_{L^2(M)},
	\end{equation}
	where
	\[
	\mathcal B(n,\lambda)=
	\begin{cases}
		\lambda^{\frac{n-2}{2}},&n\ge4,\\
		\lambda^{\frac12}\sqrt{\log\lambda},&n=3,\\
		\lambda^{\frac14},&n=2.
	\end{cases}
	\]
	Except for the logarithmic loss in dimension three, these estimates are sharp on the round sphere by the standard highest-weight and zonal examples. For the three-dimensional logarithmic loss, Burq--G\'erard--Tzvetkov pointed out that  they have the same difficulty as the endpoint Strichartz estimates on $\mathbb{R}^2$ \cite[Remark 2.13]{bgt2005}.  Deng--Zhang--Zhao \cite{dzz} recently removed this loss on the round sphere $S^3$ by using the compact Lie group structure of $SU(2)$.
	
	Our approach is different and applies to arbitrary closed three-dimensional Riemannian manifolds. We prove log-free endpoint bilinear  spectral cluster estimates on every closed three-manifold.  At the same time, we show that ellipticity of the Riemannian distance phase function is essential: a hyperbolic Carleson--Sj\"olin model has sharp endpoint examples with the logarithmic factor.  Thus the paper identifies the precise geometric mechanism behind the endpoint logarithm. This  provides a robust harmonic-analytic tool for studying critical problems of dispersive equations on curved spaces.
	
	Beyond this endpoint issue, we establish sharp multilinear $L^p$ spectral cluster estimates for all $p\ge1$.  The range $1\le p\le2$ is substantially more delicate than the range $p>2$, because the sharp lower bounds are not generated by a single family of spherical harmonics.  We show that the extremal behavior is obtained by combining six model packet profiles---beam, beam block, envelope, envelope train, zonal, and zonal train.  See Sections \ref{sec11} and \ref{sec12}.

	\subsection{New phenomena in the range $p\le 2$}
	
	The estimates in Theorem~\ref{thm1} are substantially more subtle in the range $p\le 2$. For $p>2$, the estimates are closely related to the classical linear restriction and spectral cluster theory, where the sharp examples are already well understood. In this regime, the multilinear estimates are largely governed by the same concentration mechanisms that appear in the linear problem.
	
	New phenomena emerge once $p\le2$. Observe that if one sets $\lambda=\mu$ and $f=g$ in \eqref{bie}, then the bilinear estimate reduces to an $L^{2p}$ estimate for a single oscillatory integral operator. More generally, if all frequency scales are equal and all functions coincide, then \eqref{goal} becomes an $L^{kp}$ estimate. Thus the multilinear problem is naturally connected to the linear theory at exponent $kp$. The classical theory largely concerns the range $kp\ge 2k$, corresponding to $p\ge2$. In that regime the extremizing mechanisms are already known. However, Theorems~\ref{thm1} and~\ref{thm2} extend into the range
	\[
	2\le kp\le 2k,
	\qquad\text{equivalently}\qquad
	\frac{2}{k}\le p\le 2,
	\]
	where the linear theory provides much less guidance. In particular, there is no a priori reason to expect the same extremizing configurations that govern the classical $L^q$ estimates to remain dominant. As we shall see, this range exhibits new multilinear phenomena arising from interactions between different concentration profiles.
	
	For eigenfunctions on the sphere, the classical extremizers are of two distinct types. The first are the zonal harmonics, which concentrate near a point. The second are the highest-weight harmonics (or Gaussian beams), which concentrate along a geodesic. For linear estimates, these two families dominate in different exponent ranges: zonal harmonics are extremal at larger exponents, while Gaussian beams become extremal at smaller exponents. The transition occurs at the critical exponent
	\[
	p_c=\frac{2(n+1)}{n-1}.
	\]
	
	In the multilinear setting, however, one is no longer restricted to a single extremizing profile. The product structure allows different factors to exhibit different concentration patterns simultaneously. Consequently, a much richer collection of sharp examples becomes available. These mixed configurations have no analogue in the linear theory. Their interaction produces several competing lower bounds, and different configurations dominate in different regions of the $(p,\lambda_1,\dots,\lambda_k)$-parameter space.
	
	One of the main messages of this paper is that the threshold exponents 
		\[
	p_0=\frac{n+1}{n-1},\qquad
	p_1=\frac{n-1}{n-2},\qquad
	p_2=\frac{n+1}{n},\qquad
	p_3=\frac{2(n+1)}{3(n-1)}.
	\] arise from transitions between these competing extremal geometries. Heuristically, the endpoints $p_0=p_c/2$ and $p_3=p_c/3$ correspond to the classical transition between zonal-type and Gaussian-beam-type behavior inherited from the linear theory. In contrast, the endpoints $p_1$ and $p_2$ reflect a genuinely multilinear phenomenon arising from interactions between distinct concentration patterns. The piecewise structure of $\mathcal E(n,k,p;\lambda_1,\dots,\lambda_k)$ shows that no single family of extremizers can account for the entire range $1\le p\le2$.  Table \ref{tb6} summarizes the geometric nature of the sharp examples for bilinear estimates in the different regimes. See Sections \ref{sec11} and \ref{sec12}.	
	\begin{table}[h]
		\centering
		\begin{tabular}{lll}
			\hline
			Range & $\la$-factor & $\mu$-factor \\ \hline
			$1\le p\le p_1$ & Beam & Beam block \\
			$p_1\le p\le p_0,\ \mu\le \la^2$ & Envelope train & Beam \\
			$p_0\le p\le 2,\ \mu\le \la^2$ &
			Zonal train &
			Envelope train \\
			$p_1\le p\le 2,\ \mu\ge \la^2$ &
			Zonal train&
			Beam block \\
			\hline
		\end{tabular}
		\vspace{5pt}
		\caption{Extremizing configurations in Theorem~\ref{thm1}.}
		\label{tb6}
	\end{table}
	
	This phenomenon becomes even more pronounced for the $k$-linear estimates of Theorem~\ref{thm2}. In contrast to the linear theory, where only the zonal and Gaussian beam examples need to be considered, the multilinear problem requires one to analyze all possible combinations of these concentration profiles among the different factors. The various endpoints appearing in Theorem~\ref{thm2} may therefore be interpreted as transitions between different multilinear extremizing configurations. 
	
	To the best of our knowledge, this systematic interaction between zonal and Gaussian beam profiles has not previously appeared in the literature on restriction estimates, oscillatory integrals, or spectral cluster bounds. One of the principal novelties of the present paper is the identification of these mixed extremal configurations and the demonstration that they completely determine the sharp multilinear estimates in the range $p\le2$.
	
\subsection{Log-free endpoint estimate}\label{subsec:logfree-endpoint}

We now explain the proof idea behind the log-free endpoint estimate for three-dimensional spectral clusters.  The estimate is
\begin{equation}\label{eq:intro-logfree-bilinear}
	\|\chi_\lambda f\,\chi_\mu g\|_{L^2(M)}
	\lesssim \lambda^{1/2}\|f\|_{L^2(M)}\|g\|_{L^2(M)},
	\qquad 1\ll\lambda\le\mu,
\end{equation}
for every closed three-dimensional Riemannian manifold.  This is the only point in the spectral cluster theory where the general oscillatory-integral estimates of Theorem~\ref{thm1} are not sufficient: for arbitrary three-dimensional Carleson--Sj\"olin phases the endpoint bound has a sharp factor $\sqrt{\log\lambda}$.  Section~\ref{sec2} shows that this logarithm is forced by hyperbolic asymptotic directions.  The spectral cluster phase is different.  It is the Riemannian distance phase, and its ellipticity can be exploited in a way that avoids the dyadic summation which produces the logarithm.

The first ingredient is a common-phase bilinear estimate.  Suppose two operators have the same elliptic Carleson--Sj\"olin phase $\varphi(x,\xi)$ in dimension three, but possibly different amplitudes and frequencies $1\le \alpha\le\beta$:
\[
T^b_\nu h(x)=\nu\int e^{i\nu\varphi(x,\xi)}b(x,\xi,\nu)h(\xi)\,d\xi .
\]
Proposition~\ref{prop:singlephase} proves
\[
\|T^{b_1}_\alpha f\,T^{b_2}_\beta g\|_{L^2}\lesssim
\alpha^{1/2}\|f\|_2\|g\|_2 .
\]
The proof linearizes the product and writes $\sigma=\alpha/\beta$.  For the pair of frequency variables $(\xi,\eta)$ one introduces
\[
v=\eta+\sigma\xi,
\qquad
u=\frac{\sigma|\xi-\eta|^2}{2(1+\sigma)},
\qquad
\omega=\frac{\xi-\eta}{|\xi-\eta|}.
\]
For fixed $\omega$, the variables $(u,v)$ are three-dimensional.  The definiteness of the second fundamental form implies that the map
\[
(u,v)\longmapsto \nabla_x\{\sigma\varphi(x,\xi(u,v))+\varphi(x,\eta(u,v))\}
\]
is uniformly nondegenerate.  H\"ormander's $L^2$ oscillatory integral theorem then gives a gain $\beta^{-3/2}$ in these three variables, and the Jacobian of the above change of variables converts this into the factor $\alpha^{1/2}$.  No angular-scale decomposition is used, and hence no logarithmic loss appears.

The second ingredient is that the distance phases occurring in Sogge's spectral cluster parametrix can be reduced to this common-phase situation.  After localizing in a small normal-coordinate ball and writing the integration variable in polar coordinates $y=\exp_o(r\omega)$, a frozen radial piece of the spectral projector has the form
\[
T_{\lambda,r}h(x)=\lambda\int e^{-i\lambda d_g(x,Y_r(\xi))}
b_r(x,\xi,\lambda)h(\xi)\,d\xi,
\qquad Y_r(\xi)=\exp_o(r\Omega(\xi)).
\]
The phases for different radii $r$ are not literally the same, so Proposition~\ref{prop:singlephase} cannot be applied directly.  We introduce a larger outer radius $R$ and the outer phase
\[
\Phi_R(x,\theta)=-d_g(x,Y_R(\theta)).
\]
The scaled normal-coordinate expansion for the distance function and Gauss's lemma show that $\Phi_R$ is an elliptic Carleson--Sj\"olin phase on each sufficiently small cap.  Stationary phase in the intermediate angular variable then gives the outer-radius factorization
\[
T_{\lambda,r}=E_{\lambda,R}^{A_r}U_{\lambda,r}+R_{\lambda,r}.
\]
Here $E_{\lambda,R}^{A_r}$ has the common outer phase $\Phi_R$, the angular propagator $U_{\lambda,r}$ is uniformly $L^2$-bounded, and the remainder $R_{\lambda,r}$ is harmless.  Thus two frozen pieces with radii $r_1,r_2$ are compared only after they have both been propagated to the same outer sphere.  The common-phase estimate then yields the uniform frozen bound
\[
\|T_{\alpha,r_1}f\,T_{\beta,r_2}g\|_{L^2(U)}
\lesssim \alpha^{1/2}\|f\|_2\|g\|_2,
\qquad 1\le\alpha\le\beta .
\]

There are two geometric configurations which require separate bookkeeping.  If the two angular caps are separated from both the parallel and antipodal relations, the corresponding covector surfaces are uniformly transverse, and a direct H\"ormander $L^2$ argument gives the same frozen estimate.  If the caps are antipodal, one uses signed phases $\Phi_r^\sigma(x,\xi)=-\sigma d_g(x,Y_r^\sigma(\xi))$, $\sigma=\pm1$.  Section~\ref{sec7} proves the signed factorization and verifies the exact critical-value identity needed to reduce the antipodal case again to the same positive outer phase.

Finally, Sogge's spectral cluster parametrix writes $\chi_\lambda=S_\lambda+R_\lambda$, with a rapidly smoothing remainder.  The main term $S_\lambda$ is decomposed into finitely many angular caps and a bounded interval of radii.  The frozen two-radius estimate is uniform in the radii, while the polar-coordinate decomposition preserves the $L^2$ norm after summing over caps and integrating in $r$.  Minkowski's inequality and Cauchy's inequality in the radial variables therefore assemble the frozen estimates into \eqref{eq:intro-logfree-bilinear}.

	\subsection{Paper structure} In Section \ref{sec2}, we construct a hyperbolic Carleson--Sj\"olin phase and prove that the logarithmic factors in three-dimensional oscillatory integral theorems are sharp.  In Section \ref{sec3}, we prove a log-free endpoint bilinear oscillatory integral theorem with a common elliptic Carleson--Sj\"olin phase. In Section \ref{sec4}, we prove some useful estimates on the Riemannian distance function, which is exactly the phase in Sogge's spectral cluster parametrix for Laplace eigenfunctions on smooth closed manifolds. In Section \ref{sec5}, we prove an outer-radius factorization for Sogge's spectral cluster parametrix using stationary phase, and then  prove a two-radius bilinear estimate on a common cap. In Section \ref{sec7}, we prove a two-radius bilinear estimate on antipodal caps. In Section \ref{sec8},  we complete the proof of log-free  spectral cluster estimates in all dimensions. In Section \ref{sec10}, we prove the multilinear oscillatory integral estimates. In Section \ref{sec11}, we construct the model packet profiles on the sphere. In Section \ref{sec12}, we use the model packet profiles to prove the sharpness assertions in Theorems \ref{thm1}--\ref{thm4}. 
	
	\subsection{Notation} 	Throughout the paper, \(X\lesssim Y\) means that \(X\le CY\), where the
	constant \(C\) is independent of the frequencies \(\lambda_j\).  We write \(X\approx Y\) if
	\(X\lesssim Y\) and \(Y\lesssim X\).  Moreover, \(X\gg Y\) means that
	\(X\ge CY\) for a sufficiently large constant \(C>0\). 
    
    The implicit constants $C$ may depend on $n,k,p$, on the fixed local amplitudes, on finitely many seminorms of the participating phases, and on the quantitative constants in the Carleson--Sj\"olin hypotheses.   

    For two sets $U,V$, we write $U\Subset V$, if $\overline{U}$ is compact and $\overline{U}\subset V$.
    
    The empty product is interpreted as $1$, and we set $\lambda_j=2$ for $j\le0$.
	\subsection{Acknowledgments}
	
	The authors are supported in part by the National Key R\&D Program of China
	2024YFA1015300.  C.Z. is also supported in part by NSFC Grant 12371097.
	Z.Z. is also supported in part by NSFC Grant 12501065.

	\section{Sharpness of the logarithm in three dimensions}\label{sec2}
	
	In this section, we prove that the logarithmic factors in three-dimensional multilinear oscillatory integral estimates cannot be removed if we only assume the Carleson--Sj\"olin condition. The logarithmic factors come from hyperbolic asymptotic directions. This example suggests that one should exploit the ellipticity of the phase to prove log-free estimates (see Section \ref{sec3}).
	\subsection{The bilinear endpoint}
	\begin{proposition}\label{prop:hyperbolic-log-sharp}
		Let
		\[
		\Phi_{\mathrm{hyp}}(s,z,\xi)=z\cdot\xi+s\xi_1\xi_2,
		\qquad y=(s,z)\in\mathbb R\times\mathbb R^2,
		\quad \xi\in\mathbb R^2,
		\]
		and let $b\in C_0^\infty$ be equal to one for $0\le s\le1$, $|z|\le c_0$, and $|\xi|\le c_0$, with $c_0>0$ small.  Define
		\[
		T_{\nu,\Phi_{\mathrm{hyp}}} h(y)=\nu\int e^{i\nu\Phi_{\mathrm{hyp}}(y,\xi)}b(y,\xi)h(\xi)\,d\xi .
		\]
		Then, for every large $\lambda$, with $\mu=\lambda^2$, there are functions $f_\lambda,g_\mu$ with
		$\|f_\lambda\|_2=\|g_\mu\|_2=1$ such that
		\begin{equation}\label{eq:hyperbolic-log-lower}
			\|T_{\lambda,\Phi_{\mathrm{hyp}}} f_\lambda\,T_{\mu,\Phi_{\mathrm{hyp}}} g_\mu\|_{L^2(\mathbb R^3)}
			\gtrsim \lambda^{1/2}\sqrt{\log\lambda}.
		\end{equation}
		
	\end{proposition}
	
	\begin{proof}
		For the model phase,
		\[
		P(y,\xi)=\nabla_y\Phi_{\mathrm{hyp}}(y,\xi)=(\xi_1\xi_2,\xi_1,\xi_2),
		\]
		and the second fundamental form  $\text{II}[\omega,\omega]$ is a nonzero multiple of
		$\omega_1\omega_2$.  Thus it is Carleson--Sj\"olin but not elliptic.  It has the two asymptotic directions parallel to the coordinate axes.  The example below stacks wave packets along these two directions.
		
		Let
		\[
		J=\left\lfloor \frac{1}{10}\log_2\lambda\right\rfloor .
		\]
		Choose a small number $c>0$, depending only on the amplitude support.  For $1\le j\le J$ set
		\[
		R_j=
		\left\{\xi:
		2^j\lambda^{-1/2}\le \xi_1\le (1+c)2^j\lambda^{-1/2},
		0\le \xi_2\le c2^{-j}\lambda^{-1/2}
		\right\}.
		\]
		The rectangles are disjoint, lie in $\{|\xi|\le c_0\}$ for $\lambda$ large, and satisfy
		$|R_j|\approx\lambda^{-1}$ and
		\[
		|\xi_1\xi_2|\ls\lambda^{-1},\qquad \xi\in R_j.
		\]
		Let $f_j=|R_j|^{-1/2}{\bf 1}_{R_j}$ and
		\[
		f_\lambda=J^{-1/2}\sum_{j=1}^J f_j.
		\]
		Then $\|f_\lambda\|_2=1$.  For $0\le s\le1$ and $|z_1|,|z_2|\le c\lambda^{-1}$, the phase variation of
		$\lambda(z\cdot\xi+s\xi_1\xi_2)$ on each $R_j$ is bounded by a small absolute constant if $c$ is chosen small.  Hence all integrals have the same sign after taking real parts, and
		\begin{equation}\label{eq:hyperbolic-low-large}
			|T_{\lambda,\Phi_{\mathrm{hyp}}} f_\lambda(s,z)|
			\gtrsim
			\lambda J^{-1/2}\sum_{j=1}^J |R_j|^{1/2}
			\gtrsim \lambda^{1/2}\sqrt J .
		\end{equation}
		
		For the high-frequency factor take
		\[
		G=\{\eta:0\le\eta_1,\eta_2\le c\mu^{-1/2}\}
		=
		\{\eta:0\le\eta_1,\eta_2\le c\lambda^{-1}\},
		\qquad
		g_\mu=|G|^{-1/2}{\bf 1}_G .
		\]
		Then $\|g_\mu\|_2=1$ and, on the same box $0\le s\le1$, $|z_1|,|z_2|\le c\lambda^{-1}$, the phase variation of
		$\mu(z\cdot\eta+s\eta_1\eta_2)$ is also bounded by a small constant.  Therefore
		\begin{equation}\label{eq:hyperbolic-high-large}
			|T_{\mu,\Phi_{\mathrm{hyp}}} g_\mu(s,z)|
			\gtrsim \mu |G|^{1/2}
			\approx \lambda .
		\end{equation}
		The box
		\[
		Q_\lambda=\{0\le s\le1,\ |z_1|,|z_2|\le c\lambda^{-1}\}
		\]
		has measure $|Q_\lambda|\approx\lambda^{-2}$.  Combining \eqref{eq:hyperbolic-low-large} and \eqref{eq:hyperbolic-high-large},
		\[
		\|T_{\lambda,\Phi_{\mathrm{hyp}}} f_\lambda T_{\mu,\Phi_{\mathrm{hyp}}} g_\mu\|_{L^2}
		\ge
		\|T_{\lambda,\Phi_{\mathrm{hyp}}} f_\lambda T_{\mu,\Phi_{\mathrm{hyp}}} g_\mu\|_{L^2(Q_\lambda)}
		\gtrsim
		(\lambda^{1/2}\sqrt J)\lambda |Q_\lambda|^{1/2}
		\gtrsim \lambda^{1/2}\sqrt J.
		\]
		Since $J\approx\log\lambda$, this proves \eqref{eq:hyperbolic-log-lower}.  Replacing the characteristic functions above by nonnegative smooth cutoffs changes only the constants.
	\end{proof}
	
	\subsection{The multilinear lower bounds}\label{subsec:hyperbolic-klinear-lp}

	For $0<L\le1$ and $R\ge1$ write
	\[
	Q_{L,R}=\{(s,z):0\le s\le \gamma L,\ |z_1|,|z_2|\le \gamma R^{-1/2}\},
	\]
	where $\gamma>0$ is fixed sufficiently small.  We shall use the following elementary packets for the hyperbolic Carleson--Sj\"olin phase
	$z\cdot\xi+s\xi_1\xi_2$.
	
	\begin{lemma}[Local packets for the hyperbolic model]\label{lem:hyperbolic-local-packets}
		Let $\nu\ge1$.
		\begin{enumerate}[(i)]
			\item If $\nu L\ls 1$ and $\nu^2\ls R$, then there is $h$ with $\|h\|_2\approx1$ such that
			\[
			|T_{\nu,\Phi_{\mathrm{hyp}}} h(y)|\gtrsim \nu,
			\qquad y\in Q_{L,R}.
			\]
			\item If $LR\ls\nu$ and $R\ls\nu^2$, then there is $h$ with $\|h\|_2\approx1$ such that
			\[
			|T_{\nu,\Phi_{\mathrm{hyp}}} h(y)|\gtrsim R^{1/2},
			\qquad y\in Q_{L,R}.
			\]
			\item If $\nu L\gg1$, $\nu^2\ls R$, then there is $h$ with $\|h\|_2\approx1$ such that
			\[
			|T_{\nu,\Phi_{\mathrm{hyp}}} h(y)|
			\gtrsim
			\Big(\frac{\nu}{L}\Big)^{1/2}\sqrt{\log(\nu L)},
			\qquad y\in Q_{L,R}.
			\]
		\end{enumerate}
	\end{lemma}
	
	\begin{proof}
		The first two assertions are the standard box packets.  For (i), take $h$ to be a normalized smooth cutoff on a fixed small frequency square.  The assumptions make the phase variation bounded on $Q_{L,R}$, and the integral has size $\nu$.  For (ii), take $h$ supported in a square of side $R^{1/2}/\nu$.  The normalized integral then has size
		\[
		\nu\left(\frac{R}{\nu^2}\right)^{1/2}=R^{1/2},
		\]
		and the conditions $LR\lesssim\nu$ and $R\lesssim\nu^2$ make the $s\xi_1\xi_2$ and $z\cdot\xi$ phase variations bounded.
		
		For (iii), choose dyadic rectangles
		\[
		R_m=\Big\{\xi:
		2^m(\nu L)^{-1/2}\le \xi_1\le (1+\gamma)2^m(\nu L)^{-1/2},\quad
		0\le \xi_2\le \gamma2^{-m}(\nu L)^{-1/2}
		\Big\},
		\]
		with both coordinates  $\ls 1$.  There are $J\approx \log(\nu L)$ such rectangles, they are disjoint, and each has area $\approx(\nu L)^{-1}$.  Let
		\[
		h=J^{-1/2}\sum_m |R_m|^{-1/2}{\bf 1}_{R_m},
		\]
		with smooth nonnegative cutoffs in place of characteristic functions if desired.  On $Q_{L,R}$ we have
		\[
		\nu |s\xi_1\xi_2|\lesssim 1,
		\qquad
		\nu |z\cdot\xi|\lesssim 1,
		\]
		for $\xi\in R_m$.  Thus the real parts of the integrals are positive and
		\[
		|T_{\nu,\Phi_{\mathrm{hyp}}} h(y)|
		\gtrsim
		\nu J^{-1/2}\sum_m |R_m|^{1/2}
		\approx
		\Big(\frac{\nu}{L}\Big)^{1/2}J^{1/2}.
		\]
		This proves the lemma.
	\end{proof}
	
	\begin{proposition}\label{prop:hyperbolic-klinear-lp}
		Let $k\ge3$. Set
		\[
		m=\la_{k-4},\qquad a=\lambda_{k-3},\qquad b=\lambda_{k-2},\qquad c=\lambda_{k-1},\qquad e=\lambda_k .
		\]
		Assume $1\ll m\ll a\ll b\ll c\ll e$.
		\begin{enumerate}[(I)]
			\item Let $1<p\le4/3$ and 
			assume $c\gs b^2/a$.  Then there are $L^2$-normalized inputs such that
			\begin{equation}\label{eq:hyperbolic-klinear-low-p}
				\left\|\prod_{j=1}^k T_{\lambda_j,\Phi_{\mathrm{hyp}}}f_j\right\|_{L^p}
				\gtrsim
				\left(\prod_{j=1}^{k-4}\lambda_j\right)
				a^{\frac52-\frac2p}
				b^{1/2}
				c^{1-\frac1p}
				\sqrt{\log(b/a)} .
			\end{equation}
			In particular, if $a\le b^{1/2}$, the last factor is $\gtrsim\sqrt{\log b}$.
			
			\item 
			Assume  $e\gs c^2/b$.  Then there are $L^2$-normalized inputs such that
			\begin{equation}\label{eq:hyperbolic-klinear-p2}
				\left\|\prod_{j=1}^k T_{\lambda_j,\Phi_{\mathrm{hyp}}}f_j\right\|_{L^2}
				\gtrsim
				\left(\prod_{j=1}^{k-3}\lambda_j\right)
				b c^{1/2}
				\sqrt{\log(c/b)}.
			\end{equation}
			In particular, if $b\le c^{1/2}$, the last factor is $\gtrsim\sqrt{\log c}$.
			
			\item   Assume $c,e\gs b^2/m$. Then there are $L^2$-normalized inputs such that
			\begin{equation}\label{eq:hyperbolic-fourlinear-p1}
				\left\|\prod_{j=1}^k T_{\lambda_j,\Phi_{\mathrm{hyp}}}f_j\right\|_{L^1}
				\gtrsim
				\left(\prod_{j=1}^{k-4}\lambda_j\right)  a^{1/2}b^{1/2}
				\sqrt{\log (a/m)}\sqrt{\log (b/m)}.
			\end{equation}
			In particular, if $m\le a^{1/2}$,  the product of the last two factors is $\gtrsim\sqrt{\log a}\sqrt{\log b}$.
		\end{enumerate}
		
	\end{proposition}
	
	\begin{proof}
		For (I), use the common box $Q=Q_{a^{-1},ac}$, so $|Q|\approx a^{-2}c^{-1}$.  For $j\le k-4$ use Lemma~\ref{lem:hyperbolic-local-packets}(i), giving factors of size $\lambda_j$ on $Q$.  The $a$-factor is also of type (i), hence has size $a$ on $Q$. Since $c\gs b^2/a$,  Lemma~\ref{lem:hyperbolic-local-packets}(iii) gives the $b$-factor with size
		\[
		(ab)^{1/2}\sqrt{\log(b/a)}
		\]
		on $Q$.  Finally, Lemma~\ref{lem:hyperbolic-local-packets}(ii) gives the $c$- and $e$-factors with size $(ac)^{1/2}$ on $Q$.  Multiplying the pointwise lower bounds and then multiplying by $|Q|^{1/p}$ gives \eqref{eq:hyperbolic-klinear-low-p}.
		
		For (II), use $Q=Q_{b^{-1},be}$, so $|Q|\approx b^{-2}e^{-1}$.  The factors $j\le k-3$ and the $b$-factor are of type (i).  Since $e\gs c^2/b$, Lemma~\ref{lem:hyperbolic-local-packets}(iii) gives the $c$-factor with size
		\[
		(bc)^{1/2}\sqrt{\log(c/b)}
		\]
		on $Q$. Finally, Lemma~\ref{lem:hyperbolic-local-packets}(ii) gives the $e$-factor with  size $(be)^{1/2}$ on $Q$.  Since $|Q|^{1/2}\approx b^{-1}e^{-1/2}$, this gives \eqref{eq:hyperbolic-klinear-p2}.
		
		For (III), use $Q=Q_{m^{-1},b^2}$, so $|Q|\approx m^{-1}b^{-2}$. For $j\le k-4$ use Lemma~\ref{lem:hyperbolic-local-packets}(i), giving factors of size $\lambda_j$ on $Q$.  Lemma~\ref{lem:hyperbolic-local-packets}(iii) gives the $a$- and $b$-factors with sizes
		\[
		(am)^{1/2}\sqrt{\log (a/m)},
		\qquad
		(bm)^{1/2}\sqrt{\log (b/m)},
		\]
		respectively. Since $c,e\gs b^2/m$, Lemma~\ref{lem:hyperbolic-local-packets}(ii) gives the $c$- and $e$-factors with size $b$ on $Q$.  Since $|Q|\approx m^{-1}b^{-2}$, this gives  \eqref{eq:hyperbolic-fourlinear-p1}.
	\end{proof}
		
	\section{Common-phase bilinear estimate}\label{sec3}	
		
	In this section, we prove a simplified version of the log-free bilinear estimates in three dimensions by exploiting the ellipticity of the phase. We will reduce the general version to this simple version by factorization (see Section \ref{sec5}). Recall H\"ormander's $L^2$ oscillatory integral theorem. We include a short proof of this standard result for completeness.
	\begin{lemma}[H\"ormander \cite{hor1973}]\label{lem:l2}
		Let $K_1\Subset W\subset\R^3$, with $W$ open, and let $K_0\subset\R^3$ be bounded and measurable. Suppose $\psi(x,y)$ is real-valued and that $\psi$ and $c$ are bounded in $C^\infty$ in the $x$-variable on $W$, uniformly for $y\in K_0$.  Suppose also that $c(x,y)$ is supported in $K_1$ in the $x$-variable, uniformly for $y\in K_0$.
		Assume that for all $x\in W$ and all $y,y'\in K_0$,
		\begin{align}
			|\nabla_x\psi(x,y)-\nabla_x\psi(x,y')| &\ge c_0|y-y'|,\label{eq:l2-lower}\\
			\partial_x^\gamma\{\nabla_x\psi(x,y)-\nabla_x\psi(x,y')\} &=O_\gamma(|y-y'|)\qquad\text{for every multi-index }\gamma.\label{eq:l2-upper}
		\end{align}
		Then
		\[
		S_\lambda h(x)=\int e^{i\lambda\psi(x,y)}c(x,y)h(y)\,dy
		\]
		satisfies
		\[
		\|S_\lambda h\|_{L^2}\le C\lambda^{-3/2}\|h\|_{L^2(K_0)},\qquad \lambda\ge1 .
		\]
	\end{lemma}
	
	\begin{proof}
		Choose $\chi\in C_0^\infty(W)$ with $\chi=1$ on $K_1$. The kernel of $S_\lambda^*S_\lambda$ is
		\[
		K(y,y')=\int_W e^{i\lambda(\psi(x,y')-\psi(x,y))}\chi(x)\overline{c(x,y)}c(x,y')\,dx .
		\]
		The trivial estimate gives $|K(y,y')|\le C$. If $\lambda|y-y'|>1$, use the vector field
		\[
		L=\frac1{i\lambda}\frac{\nabla_x(\psi(x,y')-\psi(x,y))}{|\nabla_x(\psi(x,y')-\psi(x,y))|^2}\cdot\nabla_x .
		\]
		Then $L e^{i\lambda(\psi(x,y')-\psi(x,y))}=e^{i\lambda(\psi(x,y')-\psi(x,y))}$. Choose an integration-by-parts order $N>3$.  By \eqref{eq:l2-lower}--\eqref{eq:l2-upper} and the bounded $C^\infty$ control of $c$, each coefficient generated by $(L^*)^N$ is bounded by $C(\lambda|y-y'|)^{-N}$. Hence
		\[
		|K(y,y')|\le C_N(1+\lambda|y-y'|)^{-N}.
		\]
		Schur's test in the three-dimensional $y$ variable gives $\|S_\lambda^*S_\lambda\|_{2\to2}\ls\lambda^{-3}$, proving the claim.
	\end{proof}

	Let $\varphi\in C^\infty(\Omega_x\times\Omega_\xi;\R)$, where $x\in\R^3$ and $\xi\in\R^2$. Put
	\[
	P(x,\xi)=\nabla_x\varphi(x,\xi),\qquad A(x,\xi)=\partial_\xi P(x,\xi).
	\]
	Assume $\rank A=2$ and that, for each fixed $x$, the surface $\{P(x,\xi)\}$ has definite second fundamental form. Define
	\[
	T^b_\nu h(x)=\nu\int e^{i\nu\varphi(x,\xi)}b(x,\xi,\nu)h(\xi)\,d\xi
	\]
	with amplitudes supported in a sufficiently small compact patch and belonging to a bounded $C^\infty$ order-zero family.
	
	\begin{proposition}[Common-phase bilinear estimate]\label{prop:singlephase}
		After shrinking the patch, for $1\le \alpha\le\beta$,
		\[
		\|T^{b_1}_\alpha f\,T^{b_2}_\beta g\|_{L^2}
		\le C\alpha^{1/2}\|f\|_2\|g\|_2 .
		\]
	\end{proposition}
	
	\begin{proof}
		It is enough to prove the linearized estimate
		\[
		\|B_{\alpha,\beta}F\|_{L^2}\le C\alpha^{1/2}\|F\|_{L^2_{\xi,\eta}},
		\]
		where
		\[
		B_{\alpha,\beta}F(x)=\alpha\beta\iint e^{i\alpha\varphi(x,\xi)+i\beta\varphi(x,\eta)}b_1(x,\xi)b_2(x,\eta)F(\xi,\eta)\,d\xi d\eta .
		\]
		Set $\sigma=\alpha/\beta\in(0,1]$. At $\xi_0$, after shrinking the $x$ patch, the definiteness of the second fundamental form implies that
		\[
		M(x,\omega)=\bigl(\partial^2_{\xi\xi}P(x,\xi_0)[\omega,\omega],A_1(x,\xi_0),A_2(x,\xi_0)\bigr),\qquad \omega\in S^1,
		\]
		is invertible and there exists $c_0>0$, independent of $x$ and $\omega$, such that 
		\[ |M(x,\omega) z| \geq c_0 |z|, \qquad \forall z\in \R^3.\] 
		Choose a small number
		$\delta>0$ and concentric balls
		\[
		V_0=B(\xi_0, \delta) \Subset V_1=B(\xi_0, 4\delta)
		\]
		so that $A$ and $\partial^2_{\xi\xi}P$ on $V_1$ are close to their values at
		$\xi_0$.
		
		For $\xi\ne\eta$, set
		\[
		v=\eta+\sigma\xi,
		\qquad u=\frac{\sigma|\xi-\eta|^2}{2(1+\sigma)},
		\qquad \omega=\frac{\xi-\eta}{|\xi-\eta|}\in S^1 .
		\]
		The diagonal $\xi=\eta$ has four-dimensional measure zero and is discarded during the change of variables, and the final identity is recovered by density. The inverse map is
		\[
		\rho=\left(\frac{2(1+\sigma)u}{\sigma}\right)^{1/2},
		\qquad \xi=\frac{v+\rho\omega}{1+\sigma},
		\qquad \eta=\frac{v-\sigma\rho\omega}{1+\sigma},
		\]
		and the Jacobian is
		\[
		d\xi d\eta=\frac{1}{\sigma(1+\sigma)}\,du\,dv\,d\sigma_{S^1}(\omega).
		\]
		For fixed $\sigma,\omega$, define, for $u>0$,
		\[
		\Psi_{\sigma,\omega}(x,u,v)=\sigma\varphi(x,\xi(u,v))+\varphi(x,\eta(u,v)),\qquad R_{\sigma,\omega}=\nabla_x\Psi_{\sigma,\omega},
		\]
		and extend them to $u=0$ by their one-sided limits.  
		The function \(\Psi_{\sigma,\omega}\) need not be smooth in \(u\) at
\(u=0\).  This is harmless since Lemma~\ref{lem:l2} only requires smoothness
in the \(x\)-variable, together with the estimates for
\(R_{\sigma,\omega}\) as a function of \(y=(u,v)\).  These estimates will be
verified below.  

For each fixed \((u,v)\), the phase and the transformed
amplitude are bounded in \(C^\infty\) in \(x\), uniformly in \(\sigma\) and
\(\omega\). Let \(K_{\sigma,\omega}\) be the set of \(y=(u,v)\) for which the frequencies
\(\xi(y,\omega)\) and \(\eta(y,\omega)\) from the inverse formulas belong to
\(V_0\).  Extend the transformed amplitude by zero to a fixed bounded set
containing all \(K_{\sigma,\omega}\). If \(y=(u,v)\) and \(y'=(u',v')\) lie in \(K_{\sigma,\omega}\), then the
segment \(y_t=(1-t)y+ty'\) remains in a region where the frequencies lie in
\(V_1\).  Indeed, writing
\[
\rho_t=\left(\frac{2(1+\sigma)u_t}{\sigma}\right)^{1/2},
\]
one has
\[
\frac{v_t}{1+\sigma}\in V_0,
\qquad
\rho_t^2=(1-t)\rho^2+t(\rho')^2\le (2\delta)^2,
\]
and therefore
\[
\xi(y_t,\omega)=\frac{v_t+\rho_t\omega}{1+\sigma},
\qquad
\eta(y_t,\omega)=\frac{v_t-\sigma\rho_t\omega}{1+\sigma}
\]
belong to \(V_1\).

For every smooth vector-valued \(G(\xi)\), the following identities follow
from the inverse formulas and the chain rule.  Since \(\rho\) is independent
of \(v\),
\begin{align}
\partial_{v_j}\{\sigma G(\xi)+G(\eta)\}
&=
\frac{\sigma\partial_{\xi_j}G(\xi)+\partial_{\xi_j}G(\eta)}{1+\sigma}.
\label{eq:vder}
\end{align}
For the \(u\)-derivative, using
\[
\partial_u\rho=\frac{1+\sigma}{\sigma\rho},
\qquad
\xi-\eta=\rho\omega,
\]
we have, for \(u>0\),
\begin{align}
\partial_u\{\sigma G(\xi)+G(\eta)\}
&=
\frac{\{DG(\xi)-DG(\eta)\}[\omega]}{\rho} \notag\\
&=
\int_0^1
\partial^2_{\xi\xi}G(\eta+s\rho\omega)[\omega,\omega]\,ds .
\label{eq:uder}
\end{align}
Both derivative formulas have continuous one-sided extensions to \(u=0\). Writing
\[
\xi_t=\xi(y_t,\omega),\qquad \eta_t=\eta(y_t,\omega),
\]
we have
\[
D_{(u,v)}R_{\sigma,\omega}(x,y_t)
=
\bigl(
\partial_u R_{\sigma,\omega}(x,y_t),
\partial_{v_1}R_{\sigma,\omega}(x,y_t),
\partial_{v_2}R_{\sigma,\omega}(x,y_t)
\bigr).
\]
Applying \eqref{eq:vder}--\eqref{eq:uder} to \(G=P(x,\cdot)\), we obtain
\[
\partial_u R_{\sigma,\omega}(x,y_t)
=
\int_0^1
\partial^2_{\xi\xi}P
\bigl(x,(1-r)\eta_t+r\xi_t\bigr)[\omega,\omega]\,dr
=
\partial^2_{\xi\xi}P(x,\xi_0)[\omega,\omega]
+
e_0(x,y_t),
\]
and, for \(j=1,2\),
\[
\partial_{v_j}R_{\sigma,\omega}(x,y_t)
=
\frac{\sigma A_j(x,\xi_t)+A_j(x,\eta_t)}{1+\sigma}
=
A_j(x,\xi_0)+e_j(x,y_t).
\]
Since \(\xi_t,\eta_t\in V_1\), and since
\((1-r)\eta_t+r\xi_t\in V_1\) for \(0\le r\le1\), the choice of \(V_1\)
gives
\[
|e_0(x,y_t)|+|e_1(x,y_t)|+|e_2(x,y_t)|\ll c_0.
\]
Thus, we get
\[
D_{(u,v)}R_{\sigma,\omega}(x,y_t)
=
M(x,\omega)+E(x,y_t),
\]
where 
\(
E(x,y_t)=\bigl(e_0,e_1,e_2\bigr)
\).
Hence, by the mean value theorem along the
segment \(y_t\),
\[
R_{\sigma,\omega}(x,y')-R_{\sigma,\omega}(x,y)
=
\int_0^1D_{(u,v)}R_{\sigma,\omega}(x,y_t)(y'-y)\,dt,
\]
and therefore
\[
|R_{\sigma,\omega}(x,y')-R_{\sigma,\omega}(x,y)|
\gtrsim c_0|y'-y|,
\]
which gives the required estimate for Lemma~\ref{lem:l2}. 
		
Now using the Jacobian, $B_{\alpha,\beta}F$ becomes
		\[
		B_{\alpha,\beta}F(x)=\frac{\beta^2}{1+\sigma}\int_{S^1}S_{\beta,\sigma,\omega}H_{\sigma,\omega}(x)\,d\sigma_{S^1}(\omega),
		\]
		where
		\[
		S_{\beta,\sigma,\omega}h(x)=\int e^{i\beta\Psi_{\sigma,\omega}(x,u,v)}c_{\sigma,\omega}(x,u,v)h(u,v)\,du\,dv
		\]
		and $H_{\sigma,\omega}(u,v)=F(\xi(u,v),\eta(u,v))$. Lemma~\ref{lem:l2} gives
		\[
		\|S_{\beta,\sigma,\omega}h\|_2\le C\beta^{-3/2}\|h\|_2 .
		\]
		Minkowski's inequality, Cauchy's inequality in $\omega$, and the Jacobian yield
		\begin{align*}
			\|B_{\alpha,\beta}F\|_2
			&\le C\frac{\beta^{1/2}}{1+\sigma}
			\left(\int_{S^1}\int |H_{\sigma,\omega}(u,v)|^2\,du\,dv\,d\sigma_{S^1}(\omega)\right)^{1/2} \\
			&= C\frac{\beta^{1/2}}{1+\sigma}\bigl(\sigma(1+\sigma)\bigr)^{1/2}\|F\|_2
			= C\alpha^{1/2}\|F\|_2 .
		\end{align*}
		Taking $F=f\otimes g$ and using density proves the proposition.
	\end{proof}
	
	\section{Scaled geometry of the shell}\label{sec4}
	In this section, we obtain some useful estimates for the Riemannian distance function, which is exactly the phase in Sogge's spectral cluster parametrix for Laplace eigenfunctions on smooth closed manifolds (see Lemma \ref{soggelem}). These estimates will be used to  prove the factorization in Section \ref{sec5} by stationary phase.
	
	Let $(M,g)$ be a smooth closed three-dimensional Riemannian manifold and fix a point $o\in M$. Consider a strongly geodesically convex normal coordinate ball centered at $o$. In these coordinates, $o=0$ and
	\( \exp_o \) is a diffeomorphism on the domain $B(0,\eps_*)\subset T_oM$. 
	
	Fix $C_0\ge 2$. After choosing $\eps>0$ sufficiently small, set
	\[
	I=[\eps/C_0,C_0\eps],\qquad R=4C_0\eps .
	\]
	Thus $R-r\ge 3C_0\eps$ for every $r\in I$. Let
	\[
	\Omega:V\subset \R^2\longrightarrow S^2\subset \R^3
	\]
	be a smooth angular chart near a unit vector $\omega_0$. We fix one compact angular support $K_2$ and one convex open $V_0$, called a cap region, with
	\[
	K_2\Subset V_0\Subset V .
	\]
	All amplitudes are supported in $K_2$ in the angular variable.  The fixed inclusion $K_2\Subset V_0$ provides the angular margin used below for the small constants and for the inverse critical map.
	
	For $s\in I\cup\{R\}$, define
\begin{equation}\label{Phidef}
		Y_s(\theta)=\exp_o(s\Omega(\theta)),\qquad
	\Phi_s(x,\theta)=-d_g(x,Y_s(\theta)).
\end{equation}
	For $R>r$, we refer to the annular region between the geodesic spheres $Y_r$ and $Y_R$ as the \emph{shell region}. The $x$-support lies in
	\[
	U=\{x: |x|<\tau\eps\},
	\]
	where $\tau>0$ is chosen small after $C_0$ and the cap are fixed. We fix a compact amplitude support and a slightly larger open set
	\[
	K_1\Subset U_0\Subset U,
	\]
	and take all amplitudes supported in $K_1$ in the $x$ variable.  All geometric constructions below are made uniformly on the compact set $\overline{U_0}$.
	
	For $r\in I$ and $\lambda\ge 1$, let
	\begin{equation}\label{defop}
		T^a_{\lambda,r}h(x)=\lambda\int_{K_2} e^{i\lambda \Phi_r(x,\xi)}a_r(x,\xi,\lambda)h(\xi)\,d\xi .
	\end{equation}
	The amplitudes $a_r(x,\xi,\lambda)$ form a bounded $C^\infty$ order-zero family, uniformly in $r\in I$ and $\lambda\ge1$, and are supported in $K_1\times K_2$.  We regard $a_r$ as a smooth zero-extension to $U\times V_0$.  Thus later expressions such as $a_r(x,\Xi_r(x,\theta),\lambda)$ are well-defined.  
	\subsection{Scaled normal-coordinate estimates}
	Consider the normal coordinates centered at $o$ as before, and let $z=\eps Z\in B(0,\eps_*) \subset T_oM$, where $B(0,\eps_*)$ is the strong geodesically convex ball discussed above. For
	$0<\eps\ll1$, set
	\[
	g_{\eps,ij}(Z):=g_{ij}(\eps Z),
	\qquad \text{ where } Z \in B(0,\eps_*).
	\]
	If $z(t)=\eps Z(t)$, then length and distance scale as
	\[
	\lth_g(z)=\eps \lth_{g_\eps}(Z),
	\qquad
	d_g(\eps P,\eps Q)=\eps d_{g_\eps}(P,Q).
	\]
	We denote by $|Z|$ the length of $Z$ in the Euclidean metric. The normal-coordinate expansion gives, on every fixed ball,
	\[
	g_\eps(Z)=I_3+\eps^2h_\eps(Z),
	\]
	where the family $h_\eps$ is bounded in $C^\infty$.
	Here $I_3$ denotes the $3\times3$ Euclidean identity matrix in these coordinates.
	Consequently, on every fixed bounded set, the Euler--Lagrange equation for
	$g_\eps$-geodesics is a smooth $O(\eps^2)$ perturbation of the Euclidean
	geodesic equation:
	\[
	\ddot Z=\eps^2\mathcal N_\eps(Z,\dot Z).
	\]
	Here $\mathcal N_\eps$ is bounded in $C^\infty$ on bounded subsets of $(Z,\dot Z)$.  
	All constants below are uniform for sufficiently small $\eps$. They may depend on $\eps_*$ and on finitely many derivatives of $g$. We first show some fundamental results in Riemannian geometry.
	
	\begin{lemma}[Distance in scaled normal coordinates]\label{lem:scaled-distance}
		Fix $A,d_0>0$ and suppose
		\[
		P,Q\in B(0,A),
		\qquad |P-Q|\ge d_0 .
		\]
		For all sufficiently small $\eps$,
		\[
		d_g(\eps P,\eps Q)
		=\eps |P-Q|+\eps^3E_\eps(P,Q),
		\]
		where $E_\eps$ is bounded in $C^\infty$ on the separated set.  Equivalently,
		\[
		d_{g_\eps}(P,Q)=|P-Q|+\eps^2E_\eps(P,Q).
		\]
		The same estimate remains true after composing $P,Q$ with bounded $C^\infty$ parametrizations whose images stay in the separated set $|P-Q|\ge d_0$.
	\end{lemma}

	\begin{proof}
		For fixed \(\eps\), interpolate between the Euclidean metric and \(g_\eps\) by
		\[
		g_{\eps,u}:=I_3+u h_\eps,\qquad 0\le u\le \eps^2 .
		\]
		These metrics are bounded in \(C^\infty\), uniformly in \(\eps\) and \(u\), and
		are uniformly close to the Euclidean metric on every fixed ball.  Hence, on the
		separated compact set
		\[
		|P|,|Q|\le A,\qquad |P-Q|\ge d_0,
		\]
		the minimizing geodesic, and therefore its length, depends smoothly on
		\((P,Q,u)\), with bounds uniform in \(\eps\), since we assume $B(0,A)$ is strongly geodesically convex.
		
		Because \(g_{\eps,0}=I_3\),
		\[
		d_{g_{\eps,0}}(P,Q)=|P-Q|.
		\]
		By the fundamental theorem of calculus in \(u\),
		\[
		d_{g_\eps}(P,Q)-|P-Q|
		=
		\int_0^{\eps^2}
		\partial_u d_{g_{\eps,u}}(P,Q)\,du
		=
		\eps^2E_\eps(P,Q),
		\]
		where
		\[
		E_\eps(P,Q)
		:=
		\int_0^1
		\partial_u d_{g_{\eps,s\eps^2}}(P,Q)\,ds
		\]
		is bounded in \(C^\infty\) on the separated set.  Finally,
		\[
		d_g(\eps P,\eps Q)=\eps d_{g_\eps}(P,Q),
		\]
		which gives
		\[
		d_g(\eps P,\eps Q)=\eps |P-Q|+\eps^3E_\eps(P,Q).
		\]
		The parametrized version follows from the chain rule.
	\end{proof}

	\begin{lemma}\label{lem:flow}
		Fix $A,T>0$.  Let $Z_\eps(t;P,v)$ be the $g_\eps$-geodesic satisfying
		\[
		Z_\eps(0;P,v)=P,
		\qquad
		\dot Z_\eps(0;P,v)=v,
		\qquad
		|P|+|v|\le A .
		\]
		Then, for $|t|\le T$,
		\[
		Z_\eps(t;P,v)=P+tv+\eps^2R_\eps(t;P,v),
		\qquad
		\dot Z_\eps(t;P,v)=v+\eps^2\dot R_\eps(t;P,v),
		\]
		where $R_\eps$ is bounded in $C^\infty$ on the relevant compact set.
		Moreover, if $X,Y$ remain in a fixed bounded set and $|Y-X|\ge d_0>0$, then
		the terminal $g_\eps$-unit tangent vector at $Y$ of the minimizing
		$g_\eps$-geodesic from $X$ to $Y$, pointing from $X$ to $Y$, is
		\[
		\frac{Y-X}{|Y-X|}+O(\eps^2),
		\]
		with bounded $C^\infty$ dependence on the parameters.
	\end{lemma}
	
	\begin{proof}
		The rescaled geodesic equation is
		\[
		\ddot Z=\eps^2\mathcal N_\eps(Z,\dot Z),
		\]
		where $\mathcal N_\eps$ is bounded in $C^\infty$ on bounded sets.  Its
		integral form gives
		\[
		Z_\eps(t;P,v)=P+tv+
		\eps^2\int_0^t(t-s)\mathcal N_\eps
		(Z_\eps(s;P,v),\dot Z_\eps(s;P,v))\,ds .
		\]
		Standard ODE estimates on the fixed interval $|t|\le T$ give
		\[
		Z_\eps(t;P,v)=P+tv+\eps^2R_\eps(t;P,v),
		\qquad
		\dot Z_\eps(t;P,v)=v+\eps^2\dot R_\eps(t;P,v),
		\]
		with $R_\eps$ bounded in $C^\infty$ for $|P|+|v|\le A$.
		
		For the terminal tangent estimate, consider the time-one endpoint map
		\[
		\mathcal E_\eps(X,v):=Z_\eps(1;X,v).
		\]
		Then Lemma~\ref{lem:scaled-distance} gives
		\[
		\mathcal E_\eps(X,v)=X+v+O(\eps^2),
		\qquad
		D_v\mathcal E_\eps=I_3+O(\eps^2).
		\]
		By the inverse function theorem, up to reparametrization, the initial velocity of the geodesic from $X$ to $Y$ is
		\[
		v_\eps(X,Y)=Y-X+O(\eps^2).
		\]
		Therefore, we have the terminal velocity
		\[
		\dot Z_\eps(1;X,v_\eps(X,Y))=Y-X+O(\eps^2).
		\]
		Since $g_\eps=I_3+O(\eps^2)$ and $|Y-X|\ge d_0$, normalizing this vector with
		respect to $g_\eps$ changes the Euclidean normalization only by $O(\eps^2)$.
		This proves the stated terminal tangent estimate.
	\end{proof}
	
	\subsection{Shell phase and propagated phase}
	As before, the angular variables lie in the
	single chart
	\[
	K_2\Subset V_0\Subset V,
	\]
	and all geometric estimates below are made on the compact set
	$\overline{U_0}$, while amplitudes are supported in $K_1$.  We shall shrink
	$\tau$, $\eps$, and the convex cap $V_0$ only finitely many times,
	always preserving these fixed compact inclusions.
	
	It is convenient in this section to use scaled variables.  Write
	\[
	s=\eps\rho,
	\qquad
	\rho\in[1/C_0,C_0]\cup\{\rho_R\},
	\qquad
	\rho_R:=R/\eps=4C_0,
	\qquad
	x=\eps X .
	\]
	For $r=\eps\rho\in I$, define the shell phase and the propagated phase by
	\[
	S_r(\theta,\xi)=d_g(Y_R(\theta),Y_r(\xi)),
	\qquad
	\Psi_r(x,\theta,\xi)=\Phi_R(x,\theta)+S_r(\theta,\xi).
	\]
	Here $\Phi$ is defined in \eqref{Phidef}.
	The only distance input used in this section is the scaled expansion from
	Lemma~\ref{lem:scaled-distance}: on every separated compact set,
	\[
	d_g(\eps P,\eps Q)=\eps |P-Q|+\eps^3E_\eps(P,Q),
	\]
	where $E_\eps$ is bounded in $C^\infty$ on the relevant compact set.

	\begin{lemma}\label{lem:outer-elliptic}
		After shrinking $\tau$, $\eps$, and the convex angular cap $V_0\Subset V$, the following hold uniformly for
		\[
		s\in I\cup\{R\},
		\qquad x=\eps X\in U,
		\qquad \theta,\eta\in V_0.
		\]
		\begin{enumerate}[(i)]
			\item In the $C^2$ topology with respect to the variable $\theta$,
			\begin{equation}
				\nabla_x\Phi_s(\eps X,\theta)
				=\Omega(\theta)+O(\tau+\eps^2).
				\label{eq:grad-Phi-close-Omega}
			\end{equation}
			\item The map $\theta\mapsto\nabla_x\Phi_s(x,\theta)$ is bi-Lipschitz on
			$V_0$:
			\[
			|\nabla_x\Phi_s(x,\theta)-\nabla_x\Phi_s(x,\eta)|
			\ge c|\theta-\eta|.
			\]
			Moreover, the corresponding $x$-differentiated upper bounds hold:
			\begin{equation}
				|\partial_x^\gamma\{\nabla_x\Phi_s(x,\theta)-\nabla_x\Phi_s(x,\eta)\}|
				\le C_{\gamma,\eps}|\theta-\eta|
				\quad\text{for every multi-index }\gamma.
				\label{eq:x-derivative-difference}
			\end{equation}
			\item The outer phase $\Phi_R$ satisfies the elliptic Carleson--Sj\"olin
			condition on $U\times V_0$:
			\[
			\rank(\partial_\theta\nabla_x\Phi_R(x,\theta))=2,
			\]
			and the surface $\{\nabla_x\Phi_R(x,\theta):\theta\in V_0\}$ has
			definite second fundamental form.
		\end{enumerate}
	\end{lemma}
	
	\begin{proof}
		For $s=\eps\rho$, put $P=X$ and $Q=\rho\Omega(\theta)$.  Since
		$|X|\le\tau$ and $\rho\ge1/C_0$, choosing $\tau$ small gives
		\[
		|X-\rho\Omega(\theta)|\ge \frac1{2C_0}.
		\]
		Thus
		\[
		\eps^{-1}\Phi_s(\eps X,\theta)
		=-|X-\rho\Omega(\theta)|+
		\eps^2E_\eps(X,\rho\Omega(\theta)).
		\]
		Since $\nabla_x\Phi_s(\eps X,\theta)$ equals the $X$-gradient of the scaled
		phase,
		\[
		\nabla_x\Phi_s(\eps X,\theta)
		=\frac{\rho\Omega(\theta)-X}{|\rho\Omega(\theta)-X|}
		+O(\eps^2).
		\]
		The first term is $\Omega(\theta)+O(\tau)$, proving
		\eqref{eq:grad-Phi-close-Omega}.
		
		Let
		\[
		n_{\rho,X}(\theta)=\frac{\rho\Omega(\theta)-X}{|\rho\Omega(\theta)-X|}.
		\]
		Let \(\theta=(\theta_1,\theta_2)\) be local coordinates on \(S^2\). Then
		\[
		\partial_{\theta_a}n_{\rho,X}(\theta)
		=\frac{\rho}{|\rho\Omega-X|}
		(I-n_{\rho,X}\otimes n_{\rho,X})\partial_{\theta_a}\Omega,\quad a=1,2.
		\]
		At $X=0$ this is exactly $\partial_{\theta_a}\Omega$, since $\Omega\cdot\Omega=1,\ \Omega\cdot \partial_{\theta_a}\Omega=0$. Hence
		\[
		\partial_\theta\nabla_x\Phi_s(\eps X,\theta)
		=D\Omega(\theta)+O(\tau+\eps^2).
		\]
		After shrinking the angular cap and then $\tau,\eps$, the bi-Lipschitz estimate
		follows from the corresponding estimate for the chart $\Omega$.  The derivative
		bound \eqref{eq:x-derivative-difference} follows from
		$\partial_x^\gamma=\eps^{-|\gamma|}\partial_X^\gamma$ and the mean-value theorem
		in $\theta$ on the separated set.
		
		For $s=R$, the surface $\theta\mapsto\nabla_x\Phi_R(x,\theta)$ is a
		$C^2$-small perturbation of the unit sphere chart $\theta\mapsto\Omega(\theta)$.
		The rank condition and definiteness of the second fundamental form therefore
		persist after the same shrinking.
	\end{proof}
	
	\subsection{The Euclidean shell model}
	Set $\rho_R=4C_0$, and let $\rho\in[1/C_0,C_0]$. In this section, we consider the shell region between the radii $\rho$ and $\rho_R$ in Euclidean distance. More precisely, define
	\begin{align*}
		F_{\rho,X}(\theta,\xi)
		&:=-|X-\rho_R\Omega(\theta)|
		+|\rho_R\Omega(\theta)-\rho\Omega(\xi)|,
		\\
		G_\rho(\theta,\xi)
		&:=|\rho_R\Omega(\theta)-\rho\Omega(\xi)| .
	\end{align*}
	Let $G_S(\xi)=D\Omega(\xi)^TD\Omega(\xi)$ be the spherical metric in the chart. For a square matrix $A$, we write
	\[
	\Sym A=\frac12(A+A^T).
	\]
	
	\begin{lemma}\label{lem:model}
		After possibly shrinking $V_0\Subset V$, there are numbers $c_1>0$, $\tau_0>0$, and constants
		$c,C>0$ such that, whenever
		\[
		\rho\in[1/C_0,C_0],\qquad |X|\le\tau_0,
		\qquad \xi\in V_0,
		\qquad \theta\in V,
		\qquad |\theta-\xi|\le c_1,
		\]
		one has
		\begin{align*}
			\partial_{\theta\theta}^2F_{\rho,X}(\theta,\xi)
			&\ge cI_2,
			\\
			-\Sym(\partial_{\theta\xi}^2G_\rho(\theta,\xi))
			&\ge cI_2.
		\end{align*}
		At $X=0$ and $\theta=\xi$,
		\begin{equation}
			\partial_{\theta\theta}^2F_{\rho,0}(\xi,\xi)
			=\frac{\rho_R\rho}{\rho_R-\rho}G_S(\xi),
			\qquad
			\partial_{\theta\xi}^2G_\rho(\xi,\xi)
			=-\frac{\rho_R\rho}{\rho_R-\rho}G_S(\xi).
			\label{eq:Euclidean-identities}
		\end{equation}
	\end{lemma}
	
	\begin{proof}
		Write $U(\theta,\xi)=\rho_R\Omega(\theta)-\rho\Omega(\xi)$ and
		$L=|U|$.  For any parameters $a,b$,
		\begin{equation}
			\partial_{ab}L
			=\frac{\partial_aU\cdot\partial_bU+U\cdot\partial_{ab}U}{L}
			-\frac{(U\cdot\partial_aU)(U\cdot\partial_bU)}{L^3}.
			\label{eq:distance-Hessian-formula}
		\end{equation}
		At $\theta=\xi$,
		\[
		U=(\rho_R-\rho)\Omega(\xi),
		\qquad L=\rho_R-\rho.
		\]
		Using
		\[
		\Omega\cdot\Omega=1,\qquad \Omega\cdot\partial_{\xi_a}\Omega=0,
		\qquad
		\Omega\cdot\partial_{\xi_a}\partial_{\xi_b}\Omega=-\partial_{\xi_a}\Omega\cdot\partial_{\xi_b}\Omega,
		\]
		formula \eqref{eq:distance-Hessian-formula} gives
		\begin{align*}
			\partial_{\theta_a\theta_b}^2G_\rho(\xi,\xi)
			&=\frac{\rho_R^2\partial_{\xi_a}\Omega\cdot\partial_{\xi_b}\Omega
				+(\rho_R-\rho)\rho_R\Omega\cdot\partial_{\xi_a}\partial_{\xi_b}\Omega}{\rho_R-\rho}  \\
			&=\frac{\rho_R\rho}{\rho_R-\rho}\partial_{\xi_a}\Omega\cdot\partial_{\xi_b}\Omega,\\
			\partial_{\theta_a\xi_b}^2G_\rho(\xi,\xi)
			&=\frac{-\rho_R\rho\,\partial_{\xi_a}\Omega\cdot\partial_{\xi_b}\Omega}{\rho_R-\rho}
			=-\frac{\rho_R\rho}{\rho_R-\rho}\partial_{\xi_a}\Omega\cdot\partial_{\xi_b}\Omega.
		\end{align*}
		This gives the second identity in \eqref{eq:Euclidean-identities}. Since $-|0-\rho_R\Omega(\theta)|=-\rho_R$ is constant, the first identity in
		\eqref{eq:Euclidean-identities} follows as well.
		
		After shrinking the cap, $G_S\ge m_SI_2$ on $V$.  Also
		\[
		0<\kappa_-\le \frac{\rho_R\rho}{\rho_R-\rho}\le \kappa_+<\infty,
		\qquad \rho\in[1/C_0,C_0],
		\]
		because $\rho_R-\rho\ge3C_0$.  Hence the two matrices in
		\eqref{eq:Euclidean-identities} are uniformly positive at $(X,\theta)=(0,\xi)$.
		The functions are smooth on the shell, and by continuity, the lower
		bounds persist for $|X|\le\tau_0$ and $|\theta-\xi|\le c_1$ after choosing
		$V_0$, $c_1$, and $\tau_0$ sufficiently small.
	\end{proof}

	\subsection{Critical map and nonstationary phase estimates}
	\begin{lemma}\label{lem:critical-geometry}
		After choosing $V_0\Subset V$ and $c_1$ as above, taking $\tau\le\tau_0$, and
		taking $\eps$ sufficiently small, the following hold uniformly for
		\[
		r\in I,
		\qquad x\in\overline{U_0},
		\qquad \xi\in V_0 .
		\]
		\begin{enumerate}[(i)]
			\item The equation
			\[
			\partial_\theta\Psi_r(x,\theta,\xi)=0
			\]
			has a unique solution $\theta=\Theta_r(x,\xi)$ in $|\theta-\xi|\le c_1$, and
			\begin{equation}
				|\Theta_r(x,\xi)-\xi|\le C(\tau+\eps^2).
				\label{eq:Theta-close}
			\end{equation}
			\item On the support region $|\theta-\xi|\le c_1$,
			\begin{align}
				\partial_{\theta\theta}^2\Psi_r(x,\theta,\xi)
				&\ge c\eps I_2,                                      \label{eq:Psi-Hess-lower}\\
				-\Sym(\partial_{\theta\xi}^2S_r(\theta,\xi))
				&\ge c\eps I_2.                                      \label{eq:S-mixed-lower}
			\end{align}
			Consequently,
			\begin{equation}
				|\partial_\theta\Psi_r(x,\theta,\xi)|
				\ge c\eps|\theta-\Theta_r(x,\xi)|
				\label{eq:theta-nonstationary}
			\end{equation}
			whenever the segment from $\theta$ to $\Theta_r(x,\xi)$ stays in
			$|\vartheta-\xi|\le c_1$.
			\item If the segment between $\theta$ and $\theta'$ stays in
			$|\vartheta-\xi|\le c_1$, then
			\begin{equation}
				|\nabla_\xi(S_r(\theta,\xi)-S_r(\theta',\xi))|
				\ge c\eps|\theta-\theta'|.
				\label{eq:xi-nonstationary}
			\end{equation}
			\item For fixed $x,r$, the map $\xi\mapsto\Theta_r(x,\xi)$ is a diffeomorphism
			from $V_0$ onto its image.  In scaled variables $X=x/\eps$,
			\begin{equation}
				\partial_\xi\Theta_r(\eps X,\xi)
				=I_2+O(\tau+\diam V_0 +\eps^2).
				\label{eq:Dxi-Theta-close}
			\end{equation}
			The family $\Theta_r$ is bounded in $C^\infty$ in the scaled variables $(X,\xi)$, uniformly in $r\in I$. The unscaled $x$-derivatives may cost powers of $\eps^{-1}$.
		\end{enumerate}
	\end{lemma}
	
	\begin{proof}
		Write $x=\eps X$ and $r=\eps\rho$.		
        
		(i) and (ii).  On the shell,
		\begin{align*}
			\eps^{-1}\Phi_R(\eps X,\theta)
			&=-|X-\rho_R\Omega(\theta)|+O(\eps^2),\\
			\eps^{-1}S_r(\theta,\xi)
			&=|\rho_R\Omega(\theta)-\rho\Omega(\xi)|+O(\eps^2).
		\end{align*}
		Therefore
		\begin{equation}
			\eps^{-1}\Psi_r(\eps X,\theta,\xi)
			=F_{\rho,X}(\theta,\xi)+O(\eps^2).
			\label{eq:Psi-model}
		\end{equation}
		Differentiating twice and using Lemma~\ref{lem:model} gives
		\eqref{eq:Psi-Hess-lower} and \eqref{eq:S-mixed-lower} for small $\eps$.
		
		By definition, at $X=0$, $\partial_\theta F_{\rho,0}(\xi,\xi)=0$.  Hence, for
		$|X|\le\tau$,
		\[
		|\partial_\theta\Psi_r(\eps X,\xi,\xi)|
		\le C\eps(\tau+\eps^2).
		\]
		Let $g(\theta)=\partial_\theta\Psi_r(\eps X,\theta,\xi)$.  The Hessian lower
		bound gives the strong monotonicity estimate
		\begin{equation}
			(g(\theta)-g(\eta))\cdot(\theta-\eta)
			\ge c\eps|\theta-\eta|^2
			\label{eq:strong-monotone}
		\end{equation}
		whenever the segment between $\theta$ and $\eta$ lies in $|\vartheta-\xi|\le c_1$.
		Choose $R_0=A_0(\tau+\eps^2)<c_1/2$, with $A_0$ large.  On
		$|\theta-\xi|=R_0$,
		\begin{equation}\label{eq:outwardpt}
		g(\theta)\cdot(\theta-\xi)
		\ge c\eps R_0^2-C\eps(\tau+\eps^2)R_0>0.
		\end{equation}
		
		Let \(B=\overline{B(\xi,R_0)}\), and let \(\pi:\mathbb R^2 \to B\) be the closest-point projection onto \(B\).  In other words, $\pi(x)=p$ if and only if $(x-p)\cdot(y-p)\le0$ for all $y\in B$. Define
\[
T(\theta)=\pi(\theta-g(\theta)).
\]
Then $T$ is continuous and by Brouwer's fixed point theorem, \(T\) has a fixed point
\(\theta_*\in B\).  If \(\theta_*\in\partial B\), then \(
g(\theta_*)\cdot(\theta_*-\xi)\le0, \)
contradicting \eqref{eq:outwardpt}.  Hence
\(\theta_*\in B(\xi,R_0)\) and \(g(\theta_*)=0\).
		
		The strict monotonicity \eqref{eq:strong-monotone} gives
		uniqueness in $|\theta-\xi|\le c_1$, and \eqref{eq:Theta-close} follows.
		
		   For \eqref{eq:theta-nonstationary}, integrate the Hessian along the segment
		from $\Theta_r(x,\xi)$ to $\theta$:
		\[
		\partial_\theta\Psi_r(x,\theta,\xi)
		=\left(\int_0^1
		\partial_{\theta\theta}^2\Psi_r(x,\Theta_r+t(\theta-\Theta_r),\xi)\,dt
		\right)(\theta-\Theta_r).
		\]
		The lower bound \eqref{eq:Psi-Hess-lower} gives the claim.  
        
        (iii). For \eqref{eq:xi-nonstationary}, write $h=\theta-\theta'$ and
		$\theta_t=\theta'+th$.  Then
		\[
		\nabla_\xi(S_r(\theta,\xi)-S_r(\theta',\xi))
		=\left(\int_0^1\partial_{\theta\xi}^2S_r(\theta_t,\xi)^T\,dt\right)h.
		\]
		The symmetric part of the averaged matrix is negative definite with size
		$\eps$ by \eqref{eq:S-mixed-lower}, and hence its action on $h$ has norm at least
		$c\eps|h|$, proving \eqref{eq:xi-nonstationary}.
		
		(iv). Differentiating the critical equation with respect to $\xi$ gives
		\[
		H_r(x,\xi)\partial_\xi\Theta_r(x,\xi)+B_r(x,\xi)=0,
		\]
		where
		\[
		H_r=\partial_{\theta\theta}^2\Psi_r(x,\Theta_r(x,\xi),\xi),
		\qquad
		B_r=\partial_{\theta\xi}^2S_r(\Theta_r(x,\xi),\xi).
		\]
		Using \eqref{eq:Theta-close}, the model identities
		\eqref{eq:Euclidean-identities}, and the comparison \eqref{eq:Psi-model},
		\begin{align*}
			H_r(\eps X,\xi)&=\eps B_0(\xi,\rho)
			+O(\eps(\tau+\diam V_0+\eps^2)),\\
			B_r(\eps X,\xi)&=-\eps B_0(\xi,\rho)
			+O(\eps(\tau+\diam V_0+\eps^2)),
		\end{align*}
		where
		\[
		B_0(\xi,\rho)=\frac{\rho_R\rho}{\rho_R-\rho}G_S(\xi)\ge c_0I_2.
		\]
		Thus
		\[
		\partial_\xi\Theta_r=-H_r^{-1}B_r
		=I_2+O(\tau+\diam V_0+\eps^2),
		\]
		which is \eqref{eq:Dxi-Theta-close}.  After shrinking the parameters,
		$\|\partial_\xi\Theta_r-I_2\|\le1/10$, so $\Theta_r(x,\cdot)$ is injective on
		the convex set $V_0$ and is a local diffeomorphism everywhere.
		
		Higher derivatives are obtained by differentiating
		$\partial_\theta\Psi_r(x,\Theta_r(x,\xi),\xi)=0$.  At each order, the highest
		new derivative of $\Theta_r$ appears linearly with coefficient $H_r$. All other
		terms involve only derivatives of $\Psi_r$ and lower-order derivatives of
		$\Theta_r$.  In scaled variables the derivatives of $\Psi_r$ are $O(\eps)$,
		while $H_r^{-1}=O(\eps^{-1})$.  Induction gives the stated bounded $C^\infty$ control in the scaled variables.
	\end{proof}

	\subsection{Inverse critical map and fixed angular cutoff}\label{sec:fixed-collar}\label{sec:fixed-margin}
	\begin{lemma}\label{lem:collar}
		There is an open set $W$ with
		\begin{equation}
			K_2\Subset W\Subset V_0,
			\label{eq:W-choice}
		\end{equation}
		together with a cutoff $\chi\in C_0^\infty(V_0\times V_0)$ such that, after shrinking
		$\tau$, $\eps$, and the angular cap if necessary, the following hold for every
		$x\in\overline{U_0}$ and $r\in I$.
		\begin{enumerate}[(i)]
			\item
			\begin{equation}
				\Theta_r(x,K_2)\Subset W\Subset \Theta_r(x,V_0).
				\label{eq:collar-inclusions}
			\end{equation}
			Thus the inverse critical map
			\[
			\xi=\Xi_r(x,\theta)
			\quad\Longleftrightarrow\quad
			\Theta_r(x,\xi)=\theta
			\]
			is well-defined on $\overline{U_0}\times W$, and its derivatives are bounded uniformly
			in $r$. The unscaled $x$-derivatives may cost powers of $\eps^{-1}$.
			\item The cutoff $\chi$ equals one on a fixed neighborhood of all critical pairs
			\[
			\{(\Theta_r(x,\xi),\xi):x\in\overline{U_0},\ r\in I,\ \xi\in K_2\}
			\]
			and satisfies
			\[
			\supp\chi
			\Subset \{(\theta,\xi)\in V_0\times V_0:|\theta-\xi|<c_1\}.
			\]
			Consequently, if $\chi(\theta,\xi)\chi(\theta',\xi)\ne0$ and
			$\theta_t=(1-t)\theta'+t\theta$, then
			\begin{equation}
				|\theta_t-\xi|<c_1,
				\qquad 0\le t\le1.
				\label{eq:segment-condition}
			\end{equation}
		\end{enumerate}
	\end{lemma}
	
	\begin{proof}
		Choose once and for all an open set $W$ satisfying \eqref{eq:W-choice}.  By
		Lemma~\ref{lem:critical-geometry},
		\begin{equation}
			\Theta_r(x,\xi)=\xi+O(\tau+\eps^2),
			\qquad
			D_\xi\Theta_r(x,\xi)=I_2+O(\tau+\diam V_0+\eps^2),
			\label{eq:Theta-close-C1}
		\end{equation}
		uniformly for $x\in\overline{U_0}$, $r\in I$, and $\xi\in V_0$.  After shrinking the
		parameters, the map $\xi\mapsto\Theta_r(x,\xi)$ is a uniform $C^1$-small
		perturbation of the identity on $V_0$.  Since $K_2\Subset W\Subset V_0$, the
		$C^0$ part of \eqref{eq:Theta-close-C1} gives
		$\Theta_r(x,K_2)\Subset W$.
		
		We next prove $W\Subset\Theta_r(x,V_0)$.  Fix $\theta\in W$.  Since
		$W\Subset V_0$, there is a small ball $B_\theta\Subset V_0$ centered at
		$\theta$, with a radius independent of $\theta\in W$.  On this ball set
		\[
		\mathcal T(\xi)=\xi+\theta-\Theta_r(x,\xi).
		\]
		The $C^1$-closeness in \eqref{eq:Theta-close-C1} makes $\mathcal T$ a
		contraction, and the $C^0$-closeness makes it map $B_\theta$ into itself.
		Hence $\mathcal T$ has a unique fixed point, equivalently
		$\Theta_r(x,\xi)=\theta$.  This proves the right inclusion in
		\eqref{eq:collar-inclusions}.  Differentiating
		$\Theta_r(x,\Xi_r(x,\theta))=\theta$ gives the derivative bounds for $\Xi_r$.
		
		It remains to choose $\chi$.  Since
		$\Theta_r(x,\xi)=\xi+O(\tau+\eps^2)$ for $\xi\in K_2$, all critical pairs lie,
		after shrinking $\tau$ and $\eps$, in an arbitrarily small fixed neighborhood
		of the diagonal over $K_2$.  Choose $\chi\in C_0^\infty(V_0\times V_0)$ equal
		to one on such a neighborhood and supported in $|\theta-\xi|<c_1$.  The segment
		condition follows because the set $\{|\theta-\xi|<c_1\}$ is convex in the
		$\theta$-variable for fixed $\xi$.
	\end{proof}
	
	\subsection{Exact critical value}
	\begin{lemma}[Geodesic extension and exact critical value]\label{lem:critical-value}
		For every $x\in\overline{U_0}$, $r\in I$, and $\xi\in K_2$,
		\[
		\Phi_R(x,\Theta_r(x,\xi))+S_r(\Theta_r(x,\xi),\xi)=\Phi_r(x,\xi).
		\]
	\end{lemma}
	
	\begin{proof}
		Write
		\[
		x=\eps X,
		\qquad r=\eps\rho,
		\qquad y=Y_r(\xi)=\exp_o(\eps\rho\Omega(\xi)).
		\]
		Let $\gamma_{xy}$ be the minimizing geodesic from $x$ to $y$.  By
		Lemma~\ref{lem:flow}, its terminal unit tangent at $y$, in scaled coordinates,
		is
		\[
		v_y
		=\frac{\rho\Omega(\xi)-X}{|\rho\Omega(\xi)-X|}
		+O(\eps^2)
		=\Omega(\xi)+O(\tau)+O(\eps^2).
		\]
		Extend the same $g_\eps$-geodesic past $y$ by solving
		\[
		Z(0)=\rho\Omega(\xi),
		\qquad
		\dot Z(0)=v_y.
		\]
		For $0\le t\le\rho_R-\rho+1$, Lemma~\ref{lem:flow} gives
		\begin{equation}
			Z(t)=\rho\Omega(\xi)+t\Omega(\xi)+O(\tau)+O(\eps^2),
			\qquad
			\dot Z(t)=\Omega(\xi)+O(\tau)+O(\eps^2).
			\label{eq:extended-geodesic}
		\end{equation}
		Consequently,
		\[
		\frac{d}{dt}|Z(t)|=1+O(\tau)+O(\eps^2)\ge\frac12
		\]
		after shrinking $\tau$ and $\eps$.  Since $|Z(0)|=\rho<\rho_R$ while
		$|Z(\rho_R-\rho+1)|>\rho_R$, there is a unique transverse hitting time $t_*$
		with $|Z(t_*)|=\rho_R$.  Write
		\[
		z=\exp_o(\eps Z(t_*))=Y_R(\theta_*).
		\]
		The leading radial expression in \eqref{eq:extended-geodesic} implies
		\[
		|\theta_*-\xi|\le C(\tau+\eps^2).
		\]
		Thus $\theta_*$ lies in the critical neighborhood where Lemma~\ref{lem:critical-geometry}
		applies.
		\begin{figure}[htbp]
	\centering
	\includegraphics[width=0.8\textwidth]{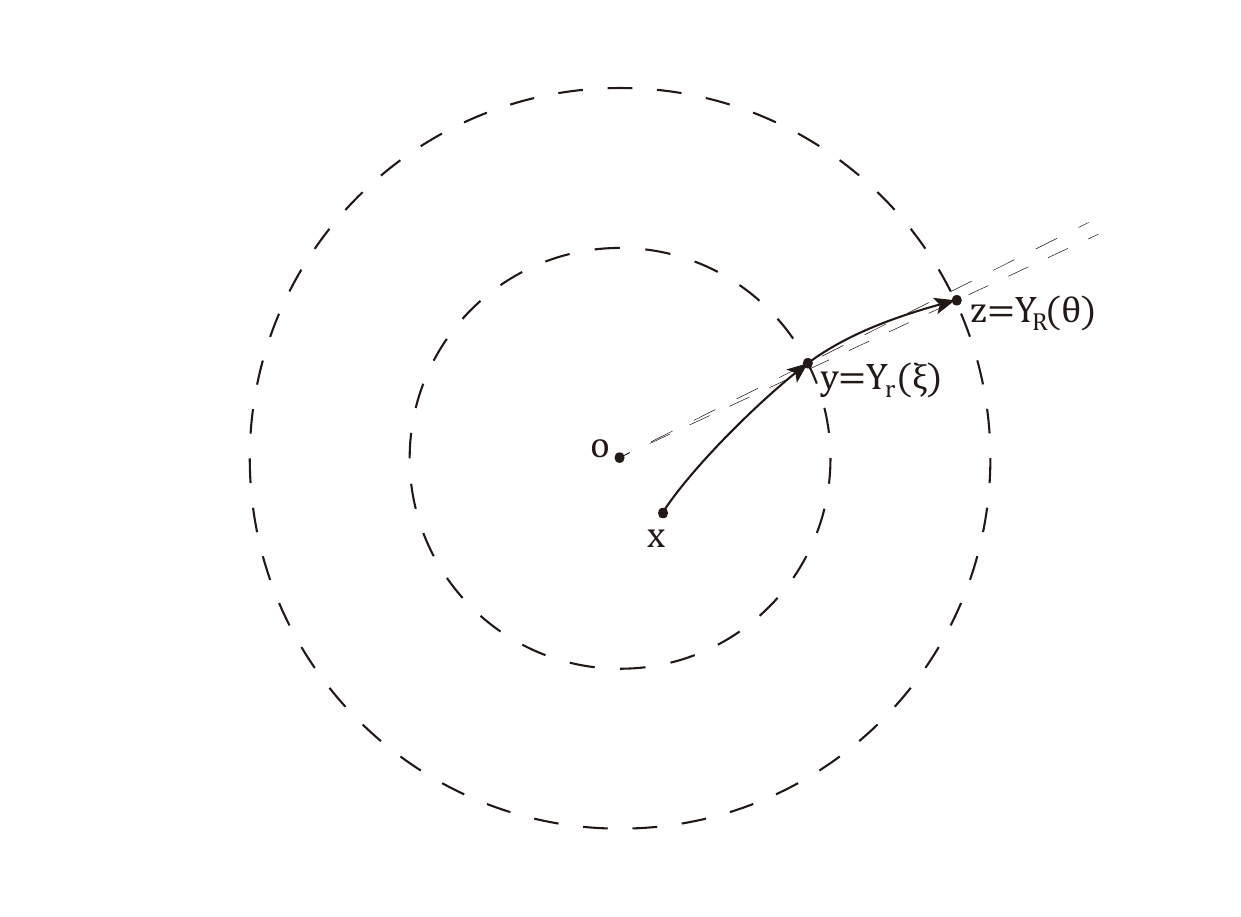}
	\caption{Geodesic from $x$ to $y$ and continued to $z$.}
	\label{figlm47}
\end{figure}
		
		The coordinate ball was chosen to be strongly geodesically convex, and $\eps$ is small
		so that the unique minimizing geodesic segments connecting $x$, $y$, and $z$ remain in it. See Figure~\ref{figlm47}. Applying the first variation formula to variations of the endpoint $z=Y_R(\theta)$ along the outer sphere gives directly
		\[
		\partial_\theta\{-d_g(x,Y_R(\theta))+d_g(Y_R(\theta),y)\}
		\big|_{\theta=\theta_*}=0.
		\]
		Since $y=Y_r(\xi)$, this is exactly
		\[
		\partial_\theta\{\Phi_R(x,\theta)+S_r(\theta,\xi)\}
		\big|_{\theta=\theta_*}=0.
		\]
		By uniqueness of the critical point,
		\[
		\theta_* = \Theta_r(x,\xi).
		\]
		Finally, distance additivity along the minimizing geodesics connecting $x$, $y$ and $z$ gives
		\[
		d_g(x,z)=d_g(x,y)+d_g(y,z).
		\]
		Therefore
		\begin{align*}
			\Phi_R(x,\theta_*)+S_r(\theta_*,\xi)
			&=-d_g(x,z)+d_g(z,y)\\
			&=-d_g(x,y)=\Phi_r(x,\xi).
		\end{align*}
		Substituting $\theta_* = \Theta_r(x,\xi)$ proves the identity.
	\end{proof}
	
	\section{Outer-radius factorization}\label{sec5}
	This section proves the analytic consequences of the geometric lemmas from
	Section~\ref{sec4}.  We keep the frozen radius $r\in I$ as an external
	parameter throughout; no $r$-derivatives are used.  We shall prove that the oscillatory integral operator $T_{\la,r}^a$ in \eqref{defop} can be factorized into the following two operators:
	\begin{align*}
		U_{\lambda,r}h(\theta)
		&=\lambda\int e^{i\lambda S_r(\theta,\xi)}\chi(\theta,\xi)h(\xi)\,d\xi,\\
		E^A_{\lambda,R}h(x)
		&=\lambda\int e^{i\lambda\Phi_R(x,\theta)}A(x,\theta,\lambda)h(\theta)\,d\theta.
	\end{align*}
	The constants below may depend on the fixed cap data and on powers of
	$\eps^{-1}$, but are uniform in $r\in I$ and $\lambda\ge1$. 
	
	\subsection{Auxiliary estimates}
	
	\begin{lemma}\label{lem:angular}
		For every $r\in I$ and every $\lambda\ge1$,
		\[
		\norm{U_{\lambda,r}h}_{L^2_\theta}\le C_\eps\norm{h}_{L^2(K_2)}.
		\]
	\end{lemma}
	
	\begin{proof}
		The kernel of $U_{\lambda,r}U_{\lambda,r}^*$ is
		\[
		K_r(\theta,\theta')
		=\lambda^2\int_{K_2} e^{i\lambda\{S_r(\theta,\xi)-S_r(\theta',\xi)\}}
		\chi(\theta,\xi)\overline{\chi(\theta',\xi)}\,d\xi .
		\]
		If the integrand is nonzero, the fixed cutoff construction in
		Section~\ref{sec:fixed-margin} implies that the segment
		$\theta_t=\theta'+t(\theta-\theta')$ lies in $|\theta_t-\xi|<c_1$.  Hence
		Lemma~\ref{lem:critical-geometry} gives
		\[
		\abs{\nabla_\xi\{S_r(\theta,\xi)-S_r(\theta',\xi)\}}
		\ge c\eps|\theta-\theta'| .
		\]
		The higher $\xi$-derivatives of the scaled phase difference satisfy the corresponding $O(|\theta-\theta'|)$ bounds on the same support.  Integration by parts then gives
		\[
		|K_r(\theta,\theta')|
		\le C_N\lambda^2(1+\lambda\eps|\theta-\theta'|)^{-N},
		\qquad N>2.
		\]
		Schur's test on the two-dimensional angular patch yields
		\[
		\sup_\theta\int |K_r(\theta,\theta')|\,d\theta'
		+\sup_{\theta'}\int |K_r(\theta,\theta')|\,d\theta
		\le C_N\eps^{-2}.
		\]
		Thus $\|U_{\lambda,r}\|_{2\to2}\le C_\eps$.
	\end{proof}
	
	\begin{lemma}\label{lem:linear}
		Let $K\Subset U$ be fixed.  Let
		$A(x,\theta,\lambda)$ be a bounded $C^\infty$ order-zero amplitude family supported in $K\times V_0$.  Then
		\begin{align}
			\norm{E^A_{\lambda,R}h}_{L^2(U)}&\le C_\eps\norm{h}_{L^2_\theta},\label{eq:linear-L2}\\
			\norm{E^A_{\lambda,R}h}_{L^\infty(U)}&\le C\lambda\norm{h}_{L^2_\theta}.\notag\end{align}
	\end{lemma}
	
	\begin{proof}
		The $L^\infty$ bound follows immediately from Cauchy's inequality in the
		compact angular patch.  For the $L^2$ bound, the kernel of
		$(E^A_{\lambda,R})^*E^A_{\lambda,R}$ is
		\[
		K_s(\theta,\eta)
		=\lambda^2\int_U e^{i\lambda(\Phi_s(x,\eta)-\Phi_s(x,\theta))}
		\overline{A(x,\theta,\lambda)}A(x,\eta,\lambda)\,dx .
		\]
		Lemma~\ref{lem:outer-elliptic} gives, on $K\times V_0$,
		\[
		|\nabla_x\Phi_s(x,\eta)-\nabla_x\Phi_s(x,\theta)|\ge c|\eta-\theta|,
		\]
		with the corresponding differentiated upper bounds, where the constants may
		depend on $\eps$.  Repeated integration by parts in $x$ therefore gives
		\[
		|K_s(\theta,\eta)|
		\le C_{N,\eps}\lambda^2(1+\lambda|\theta-\eta|)^{-N},
		\qquad N>2.
		\]
		Schur's test on the two-dimensional angular variables proves
		\eqref{eq:linear-L2}.
	\end{proof}
	
	\subsection{Stationary phase}
	
	We shall use the following consequences of Section~\ref{sec4}.  For
	$x\in\overline{U_0}$, $r\in I$, and $\xi\in K_2$, the phase
	\[
	\Psi_r(x,\theta,\xi)=\Phi_R(x,\theta)+S_r(\theta,\xi)
	\]
	has a unique critical point $\Theta_r(x,\xi)$ on the support region, with exact
	critical value
	\begin{equation}\label{eq:section4-critical-value}
		\Psi_r(x,\Theta_r(x,\xi),\xi)=\Phi_r(x,\xi).
	\end{equation}
	The critical Hessian
	\begin{equation}\label{eq:section4-H-def}
		H_r(x,\xi)=\partial_{\theta\theta}^2
		\Psi_r(x,\Theta_r(x,\xi),\xi)
	\end{equation}
	satisfies $H_r\ge c\eps I_2$ and $|\det H_r|\ge c\eps^2$.  For later use set
	\[
	q_r(x,\xi):=(2\pi)e^{i\pi\sig(H_r(x,\xi))/4}
	|\det H_r(x,\xi)|^{-1/2},
	\]
	where $\sig(H)$ is the number of positive eigenvalues minus the number of
	negative eigenvalues.  Then $|q_r|\approx\eps^{-1}$.
	Moreover
	$\Xi_r$ is the inverse critical map on $\overline{U_0}\times W$, so
	\[
	\Xi_r(x,\Theta_r(x,\xi))=\xi,
	\qquad
	\Theta_r(x,\Xi_r(x,\theta))=\theta .
	\]
	The cutoff $\chi$ is equal to one on a fixed neighborhood of the critical graph
	and satisfies the segment condition from Section~\ref{sec:fixed-margin}.  Hence,
	by the nonstationary estimate in Lemma~\ref{lem:critical-geometry} and
	compactness, there is a smaller fixed neighborhood of the critical graph such
	that on the support of $\chi$ outside this neighborhood,
	\begin{equation}\label{eq:section4-nonstationary}
		|\partial_\theta\Psi_r(x,\theta,\xi)|\gs\eps .
	\end{equation}
	
	Choose $\psi_1\in C_0^\infty(U)$ and $\psi\in C_0^\infty(W)$ such
	that $\supp\psi_1\Subset U_0,\ 
	\psi_1=1\text{ near }K_1,$
	and $\psi=1$ on a fixed neighborhood of
	\[
	\{\Theta_r(x,\xi):x\in K_1,\ r\in I,\ \xi\in K_2\}.
	\]
	The support of $\psi$ is compactly contained in the collar where $\Xi_r$ is
	defined.  
	
	\begin{proposition}\label{prop:sogge-stationary}
		Let $K\Subset U$ be a compact set on which the critical geometry above holds,
		and let $B_r(x,\xi,\theta,\lambda)$ be a bounded $C^\infty$ order-zero
		amplitude family supported where $\chi$ is supported.  Then, for
		$x\in K$, $\xi\in K_2$, and $r\in I$,
		\begin{align}\label{eq:sogge-sp-formula}
			&\int e^{i\lambda\Psi_r(x,\theta,\xi)}
			B_r(x,\xi,\theta,\lambda)\chi(\theta,\xi)\,d\theta  \\
			&\quad =
			e^{i\lambda\Phi_r(x,\xi)}
			\left
			[
			\lambda^{-1}q_r(x,\xi)
			B_r(x,\xi,\Theta_r(x,\xi),\lambda)
			+\lambda^{-2}\mathcal E_{\lambda,r}^B(x,\xi)
			\right],
			\notag
		\end{align}
		where $\mathcal E_{\lambda,r}^B$ is a bounded $C^\infty$ amplitude, with
		bounds depending only on finitely many seminorms of $B_r$ and on the fixed
		geometric data.  The same assertion applies to any smooth phase family
		satisfying the same critical point, nonstationary, Hessian, and exact
		critical-value hypotheses, with $q$ defined from the corresponding Hessian.
	\end{proposition}
	
	\begin{proof}
		Choose a smooth cutoff $\zeta_r(x,\theta,\xi)$ equal to one near the critical
		graph and supported where $\chi=1$.  On the support of $(1-\zeta_r)\chi$,
		\eqref{eq:section4-nonstationary} gives
		\[
		|\partial_\theta\Psi_r(x,\theta,\xi)|\gs\eps .
		\]
		For this non-stationary part, repeated integration by parts in $\theta$ gives, for every $L$ and fixed $N$,
		\[
		\|I_{\lambda,{\rm ns}}\|_{C^N_{x,\xi}}
		\le C_{L,N,\eps}\lambda^{-L}.
		\]
		Taking $L$ large allows this term to be absorbed into the remainder.
		
		On the critical part, the variable coefficient version of the stationary phase theorem (see e.g. Sogge \cite[Corollary 1.1.8]{fio}) applies uniformly in $r$.  The critical value is
		\eqref{eq:section4-critical-value} and the Hessian is
		\eqref{eq:section4-H-def}, so the critical contribution is
		\[
		e^{i\lambda\Phi_r(x,\xi)}
		\left[
		\lambda^{-1}q_r(x,\xi)
		B_r(x,\xi,\Theta_r(x,\xi),\lambda)
		+\lambda^{-2}\mathcal E_{\lambda,r}^B(x,\xi)
		\right],
		\]
		with $\mathcal E_{\lambda,r}^B$ bounded in $C^\infty$.  Adding the
		nonstationary contribution gives \eqref{eq:sogge-sp-formula}.  The final
		variant follows because the proof uses only the listed hypotheses.
	\end{proof}
	
	\subsection{Outer-radius factorization}
	
	\begin{proposition}[Outer-radius factorization]\label{prop:factorization}
		For every bounded $C^\infty$ order-zero amplitude $a_r$ satisfying the hypotheses in the definition of \(T^a_{\lambda,r}\) in \eqref{defop}, define
		\begin{equation}\label{eq:A-sogge-def}
			A_r(x,\theta,\lambda)
			:=\psi_1(x)\psi(\theta)
			q_r(x,\Xi_r(x,\theta))^{-1}
			a_r(x,\Xi_r(x,\theta),\lambda)
		\end{equation}
		wherever $\psi_1(x)\psi(\theta)\ne0$, and set $A_r=0$ elsewhere.  This is a
		smooth order-zero amplitude because $\supp\psi_1\Subset U_0$
		and $\supp\psi\Subset W$, where $\Xi_r$ is defined.  Then
		\[
		T^a_{\lambda,r}h
		=E^{A_r}_{\lambda,R}U_{\lambda,r}h+R_{\lambda,r}h,
		\qquad h\in L^2(K_2),
		\]
		and
		\[
		\|R_{\lambda,r}h\|_{L^\infty(U)}+
		\|R_{\lambda,r}h\|_{L^2(U)}
		\le C_\eps\|h\|_{L^2(K_2)},
		\qquad \lambda\ge1,\ r\in I.
		\]
		The amplitude $A_r$ belongs to a bounded $C^\infty$ order-zero class in the
		active variables $(x,\theta)$.
	\end{proposition}
	
	\begin{proof}
		For a trial outer amplitude $A$, the composition kernel of
		$E^A_{\lambda,R}U_{\lambda,r}$ is
		\begin{equation}\label{eq:composition-kernel-section4}
			K^A_{\lambda,r}(x,\xi)
			=\lambda^2\int e^{i\lambda\Psi_r(x,\theta,\xi)}
			A(x,\theta,\lambda)\chi(\theta,\xi)\,d\theta .
		\end{equation}
		Apply Proposition~\ref{prop:sogge-stationary}, with $K=\overline{U_0}$, to
		$B_r(x,\xi,\theta,\lambda)=A_r(x,\theta,\lambda)$.  For
		$x\in K_1$ and $\xi\in K_2$, the cutoff choices and the inverse identity give
		\[
		\psi_1(x)=1,
		\qquad
		\psi(\Theta_r(x,\xi))=1,
		\qquad
		\Xi_r(x,\Theta_r(x,\xi))=\xi.
		\]
		Therefore
		\begin{equation}\label{eq:leading-match-sogge}
			q_r(x,\xi)A_r(x,\Theta_r(x,\xi),\lambda)=a_r(x,\xi,\lambda).
		\end{equation}
		Multiplying the stationary-phase expansion by the outside factor $\lambda^2$ in
		\eqref{eq:composition-kernel-section4}, we get, on $\overline{U_0}\times K_2$,
		\[
		K^{A_r}_{\lambda,r}(x,\xi)
		=\lambda e^{i\lambda\Phi_r(x,\xi)}
		q_r(x,\xi)A_r(x,\Theta_r(x,\xi),\lambda)
		+\mathcal K_{\lambda,r}(x,\xi),
		\]
		where the kernel remainder satisfies
		\[
		\|\mathcal K_{\lambda,r}\|_{L^\infty(\overline{U_0}\times K_2)}\le C_\eps .
		\]
		Together with \eqref{eq:leading-match-sogge}, this gives the desired original
		kernel plus the bounded kernel remainder on $K_1\times K_2$.
		
Since \(a_r\) is supported in \(K_1\times K_2\), the factor
\(a_r(x,\Xi_r(x,\theta),\lambda)\) in \(A_r\) implies the kernel of
\(T^a_{\lambda,r}-E^{A_r}_{\lambda,R}U_{\lambda,r}\) is the
stationary-phase remainder, hence is uniformly \(O_\varepsilon(1)\) on
\(U\times K_2\).  The symbol
bounds for \(A_r\) follow by differentiating \eqref{eq:A-sogge-def} and using
the bounds for \(q_r\), \(\Xi_r\), and \(a_r\).
	\end{proof}
	\subsection{Two-radius bilinear estimate on a common cap}\label{sec6}
	For the operators in \eqref{defop}, we prove the following bilinear estimates by combining the outer-radius factorization in Proposition \ref{prop:factorization} and the common-phase bilinear estimate in Proposition \ref{prop:singlephase}. \begin{theorem}\label{thm:main}
		After shrinking $\eps$, $\tau$, and the caps if necessary,  for  every $r_1,r_2\in I$ and $1\le \alpha\le\beta$,
		\[
		\bigl\|T^{a_1}_{\alpha,r_1}f\,T^{a_2}_{\beta,r_2}g\bigr\|_{L^2(U)}
		\ls \alpha^{1/2}
		\|f\|_{L^2(K_2)}\|g\|_{L^2(K_2)} .
		\]
	\end{theorem}
	Apply Proposition~\ref{prop:factorization} to both frozen radii.  Write
	\[
	Q_{\nu,r}h:=E^{A_r}_{\nu,R}U_{\nu,r}h,
	\qquad
	T^a_{\nu,r}h=Q_{\nu,r}h+R_{\nu,r}h .
	\]
	For the main-main term, Lemma~\ref{lem:outer-elliptic}(iii) allows Proposition~\ref{prop:singlephase} to be applied to the common outer phase $\Phi_R$.  Lemma~\ref{lem:angular} controls the angular propagators.  Hence, for $1\le\alpha\le\beta$,
	\begin{align*}
		\|Q_{\alpha,r_1}f\,Q_{\beta,r_2}g\|_{L^2(U)}
		\ls\alpha^{1/2}
		\|U_{\alpha,r_1}f\|_2
		\|U_{\beta,r_2}g\|_2 
		\ls\alpha^{1/2}
		\|f\|_2\|g\|_2 .
	\end{align*}
	The remainders are harmless and only contribute $O(1)$. This proves Theorem~\ref{thm:main}.

	\section{Two-radius bilinear estimate on antipodal caps}\label{sec7}
	
	This section proves the signed verification needed for antipodal caps. The stationary-phase and factorization machinery has already been proved in Sections~\ref{sec4}--\ref{sec5}. Here we check that the signed antipodal phases satisfy the same hypotheses. The only new geometric ingredient is the critical-value identity for the negative ray in Lemma~\ref{lem:signed-critical-value}.
	
	For $\sigma\in\{+1,-1\}$ define
	\[
	Y_r^\sigma(\xi)=\exp_o(\sigma r\Omega(\xi)),
	\qquad
	\Phi_r^\sigma(x,\xi)=-\sigma d_g(x,Y_r^\sigma(\xi)).
	\]
	  In normal
	coordinates at $o$,
	\[
	\nabla_x\Phi_r^\sigma(0,\xi)=\Omega(\xi),
	\qquad \sigma=\pm1,
	\]
	so both signs have $x$-covectors in the same small cap.  Put
	\[
	S_r^\sigma(\theta,\xi)
	=d_g(Y_R^+(\theta),Y_r^\sigma(\xi)),
	\qquad
	\Psi_r^\sigma(x,\theta,\xi)
	=\Phi_R^+(x,\theta)+S_r^\sigma(\theta,\xi),
	\]
	where
	\[
	Y_R^+(\theta)=\exp_o(R\Omega(\theta)),
	\qquad
	\Phi_R^+(x,\theta)=-d_g(x,Y_R^+(\theta)).
	\]
	For $\sigma=+1$ this is exactly the notation of
	Sections~\ref{sec4}--\ref{sec5}.
	
	\subsection{Signed critical geometry}

	\begin{lemma}[Signed critical map and common cutoffs]
		\label{lem:signed-critical-geometry}\label{lem:signed-collar-cutoff}
		After the cap, $c_1$, $\tau$, and $\eps$ are chosen sufficiently small, the
		following hold uniformly for $\sigma=\pm1$, $r\in I$, $x\in\overline{U_0}$, and
		$\xi\in V_0$.
		\begin{enumerate}[(i)]
			\item The equation
			\[
			\partial_\theta\Psi_r^\sigma(x,\theta,\xi)=0
			\]
			has a unique solution $\theta=\Theta_r^\sigma(x,\xi)$ in the region
			$|\theta-\xi|\le c_1$, and
			\begin{equation}\label{eq:signed-Theta-close}
				|\Theta_r^\sigma(x,\xi)-\xi|\le C(\tau+\eps^2).
			\end{equation}
			\item On $|\theta-\xi|\le c_1$,
			\begin{align}
				\sigma\,\partial_{\theta\theta}^2\Psi_r^\sigma(x,\theta,\xi)
				&\ge c\eps I_2,\label{eq:signed-Psi-Hess-lower}\\
				-\sigma\,\Sym\partial_{\theta\xi}^2S_r^\sigma(\theta,\xi)
				&\ge c\eps I_2.\notag\end{align}
			Consequently, whenever the segment from $\theta$ to
			$\Theta_r^\sigma(x,\xi)$ remains in $|\vartheta-\xi|\le c_1$,
			\begin{equation}\label{eq:signed-theta-nonstationary}
				|\partial_\theta\Psi_r^\sigma(x,\theta,\xi)|
				\ge c\eps |\theta-\Theta_r^\sigma(x,\xi)|.
			\end{equation}
			\item If the segment from $\theta'$ to $\theta$ remains in
			$|\vartheta-\xi|\le c_1$, then
			\begin{equation}\label{eq:signed-xi-nonstationary}
				|\nabla_\xi(S_r^\sigma(\theta,\xi)-S_r^\sigma(\theta',\xi))|
				\ge c\eps |\theta-\theta'|.
			\end{equation}
			\item For fixed $x,r,\sigma$, the map
			$\xi\mapsto\Theta_r^\sigma(x,\xi)$ is a diffeomorphism from
			$V_0$ onto its image, and in scaled variables
			\[
			\partial_\xi\Theta_r^\sigma(\eps X,\xi)
			=I_2+O(\tau+\diam V_0+\eps^2).
			\]
			The family is bounded in $C^\infty$ in the scaled variables $(X,\xi)$, and the unscaled $x$-derivatives may cost powers of $\eps^{-1}$.
			\item The same open set $W$ and the same angular cutoff $\chi$ from
			Lemma~\ref{lem:collar} work for both signs:
			\begin{equation}\label{eq:signed-W-inclusions}
				\Theta_r^\sigma(x,K_2)\Subset W\Subset\Theta_r^\sigma(x,V_0).
			\end{equation}
			Thus the inverse
			\[
			\xi=\Xi_r^\sigma(x,\theta),
			\qquad
			\Theta_r^\sigma(x,\Xi_r^\sigma(x,\theta))=\theta,
			\]
			is defined on $\overline{U_0}\times W$ and is bounded in $C^\infty$ in the active variables.  Moreover,
			$\chi=1$ on a fixed neighborhood of every signed critical pair
			$(\Theta_r^\sigma(x,\xi),\xi)$ with $x\in\overline{U_0}$, $r\in I$, and
			$\xi\in K_2$, and its support satisfies the segment condition
			\eqref{eq:segment-condition} for both signs.
		\end{enumerate}
	\end{lemma}
	
	\begin{proof}
		We only record the changes from the unsigned argument.  Write $x=\eps X$ and
		$r=\eps\rho$.  The signed Euclidean models are
		\[
		F_{\rho,X}^\sigma(\theta,\xi)
		=-|X-\rho_R\Omega(\theta)|
		+|\rho_R\Omega(\theta)-\sigma\rho\Omega(\xi)|,
		\]
		and
		\[
		G_\rho^\sigma(\theta,\xi)
		=|\rho_R\Omega(\theta)-\sigma\rho\Omega(\xi)|.
		\]
		At $X=0$ and $\theta=\xi$, the same computation as in
		Lemma~\ref{lem:model}, with $\rho\Omega(\xi)$ replaced by
		$\sigma\rho\Omega(\xi)$, gives
		\[
		\partial_{\theta\theta}^2F_{\rho,0}^\sigma(\xi,\xi)
		=\sigma\frac{\rho_R\rho}{\rho_R-\sigma\rho}G_S(\xi),
		\qquad
		\partial_{\theta\xi}^2G_\rho^\sigma(\xi,\xi)
		=-\sigma\frac{\rho_R\rho}{\rho_R-\sigma\rho}G_S(\xi).
		\]
		Since $\rho_R-\sigma\rho$ is uniformly positive for both signs, continuity
		implies, after shrinking the cap, $c_1$, and $\tau$, that
		\[
		\sigma\partial_{\theta\theta}^2F_{\rho,X}^\sigma\ge cI_2,
		\qquad
		-\sigma\Sym\partial_{\theta\xi}^2G_\rho^\sigma\ge cI_2
		\]
		on the region $|X|\le\tau$, $|\theta-\xi|\le c_1$.
		
		The scaled distance expansion gives
		\[
		\eps^{-1}\Psi_r^\sigma(\eps X,\theta,\xi)
		=F_{\rho,X}^\sigma(\theta,\xi)+O(\eps^2),
		\qquad
		\eps^{-1}S_r^\sigma(\theta,\xi)
		=G_\rho^\sigma(\theta,\xi)+O(\eps^2).
		\]
		Thus
		\[
		\sigma\partial_{\theta\theta}^2\Psi_r^\sigma\ge c\eps I_2,
		\qquad
		-\sigma\Sym\partial_{\theta\xi}^2S_r^\sigma\ge c\eps I_2.
		\]
		Moreover
		\[
		\partial_\theta\Psi_r^\sigma(\eps X,\xi,\xi)
		=O(\eps(\tau+\eps^2)).
		\]
		The proof of Lemma~\ref{lem:critical-geometry} now applies verbatim, with
		$\partial_{\theta\theta}^2\Psi_r$ replaced by
		$\sigma\partial_{\theta\theta}^2\Psi_r^\sigma$ and with
		$-\Sym\partial_{\theta\xi}^2S_r$ replaced by
		$-\sigma\Sym\partial_{\theta\xi}^2S_r^\sigma$.  Indeed, the first of the
		last two inequalities gives the same strong monotonicity for
		$\sigma\partial_\theta\Psi_r^\sigma$, while the second gives the same angular
		propagation estimate in the $\xi$ variable.  This proves the existence and
		uniqueness of $\Theta_r^\sigma$, the bound
		\eqref{eq:signed-Theta-close}, and the two nonstationary estimates
		\eqref{eq:signed-theta-nonstationary}--\eqref{eq:signed-xi-nonstationary}.
		The differentiated implicit-equation argument from Lemma~\ref{lem:critical-geometry}
		also gives
		\[
		\partial_\xi\Theta_r^\sigma(\eps X,\xi)
		=I_2+O(\tau+\diam V_0+\eps^2),
		\]
		and the stated $C^\infty$ boundedness in scaled variables.  After shrinking
		the parameters, $\Theta_r^\sigma(x,\cdot)$ is therefore a diffeomorphism from
		$V_0$ onto its image.
		
		Finally, the common inverse neighborhood and cutoff follow from the same
		compactness and contraction argument as Lemma~\ref{lem:collar}.  For both
		signs,
		\[
		\Theta_r^\sigma(x,\xi)=\xi+O(\tau+\eps^2),
		\qquad
		D_\xi\Theta_r^\sigma(x,\xi)=I_2+O(\tau+\diam V_0+\eps^2).
		\]
		Since there are only two signs, after shrinking the parameters once more the
		same open set $W$ and the same cutoff $\chi$ work simultaneously for
		$\sigma=\pm1$.  This gives \eqref{eq:signed-W-inclusions}, the inverse maps
		$\Xi_r^\sigma$ on $\overline{U_0}\times W$, and the common segment condition.
	\end{proof}
	
	\subsection{Signed critical value}
	
	\begin{lemma}[Signed critical value]\label{lem:signed-critical-value}
		For $\sigma=\pm1$,
		\begin{equation}\label{eq:signed-critical-value}
			\Phi_R^+(x,\Theta_r^\sigma(x,\xi))
			+S_r^\sigma(\Theta_r^\sigma(x,\xi),\xi)
			=\Phi_r^\sigma(x,\xi),
		\end{equation}
		for $x\in\overline{U_0}$, $r\in I$, and $\xi\in K_2$.
	\end{lemma}
	
	\begin{proof}
		The case $\sigma=+1$ is Lemma~\ref{lem:critical-value}.  We prove the case
		$\sigma=-1$.  Write
		\[
		x=\eps X,
		\qquad r=\eps\rho,
		\qquad y=Y_r^-(\xi)=\exp_o(-r\Omega(\xi)).
		\]
		Let $\gamma$ be the minimizing geodesic from $y$ to $x$.  By
		Lemma~\ref{lem:flow}, in scaled normal coordinates its terminal unit tangent
		at $x$, pointing from $y$ to $x$, is
		\[
		v_x
		=\frac{X+\rho\Omega(\xi)}{|X+\rho\Omega(\xi)|}+O(\eps^2)
		=\Omega(\xi)+O(\tau+\eps^2).
		\]
		
		Continue this geodesic past $x$, and write the continued scaled curve as
		$Z(t)$ with
		\[
		Z(0)=X,
		\qquad \dot Z(0)=v_x.
		\]
		Again by Lemma~\ref{lem:flow}, for $0\le t\le\rho_R+2$,
		\[
		Z(t)=X+t\Omega(\xi)+O(\tau+\eps^2),
		\qquad
		\dot Z(t)=\Omega(\xi)+O(\tau+\eps^2).
		\]
		Hence $H(t)=|Z(t)|^2-\rho_R^2$ is strictly increasing on
		$[\rho_R-1,\rho_R+1]$ and changes sign there. The continued geodesic
		therefore has a unique transverse first intersection with the positive outer
		sphere. Write it as
		\[
		z=Y_R^+(\theta_*).
		\]
		The same expansion gives
		\[
		\frac{Z(t_*)}{|Z(t_*)|}=\Omega(\xi)+O(\tau+\eps^2),
		\]
		and the bi-Lipschitz property of the angular chart yields
		\[
		|\theta_*-\xi|\le C(\tau+\eps^2).
		\]
		Thus $\theta_*$ lies in the signed critical neighborhood.
		\begin{figure}[htbp]
	\centering
	\includegraphics[width=0.8\textwidth]{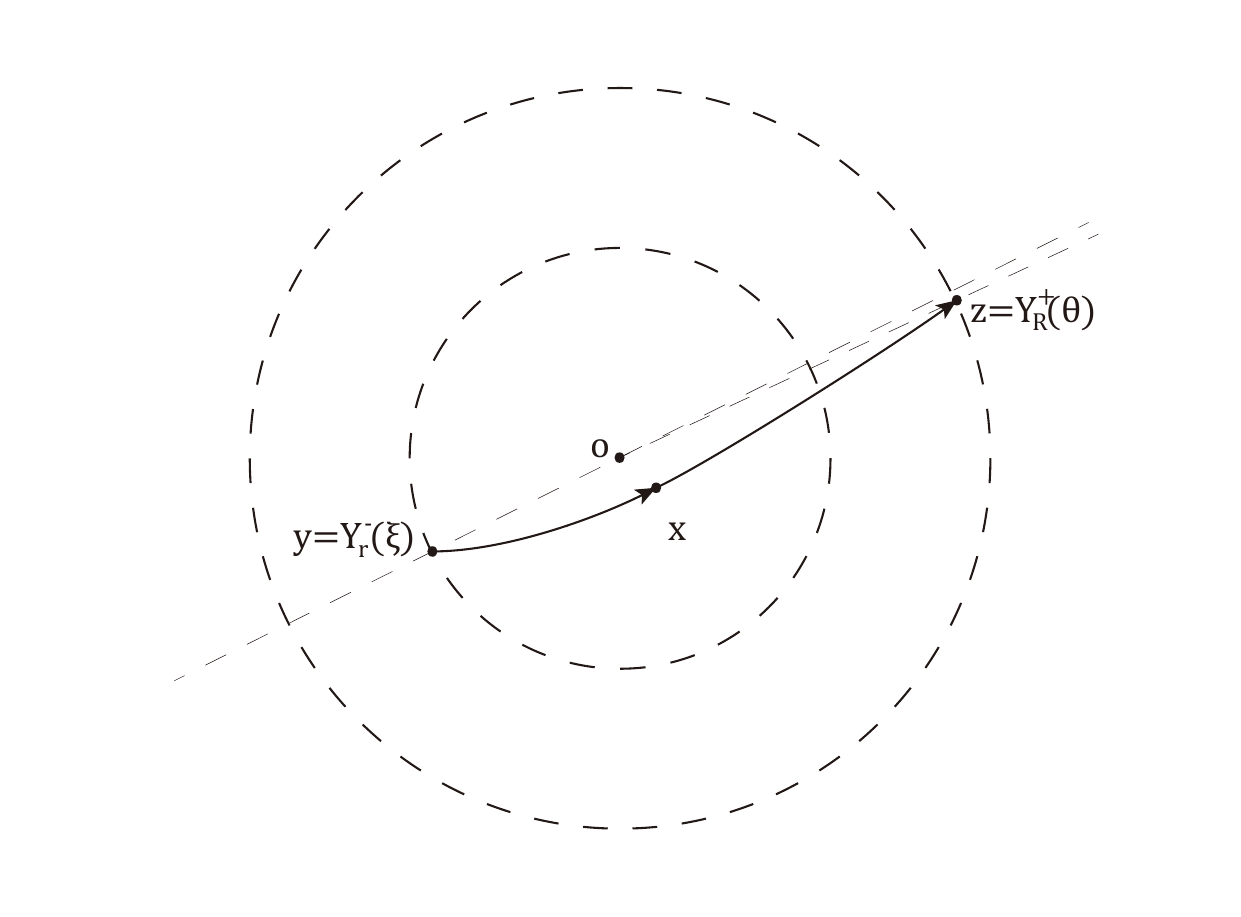}
	\caption{Geodesic from $y$ to $x$ and continued to $z$.}
	\label{figlm72}
\end{figure}
		Since the whole curve connecting $y$, $x$, and $z$ lies in the fixed strongly convex normal
		ball, the relevant subsegments are minimizing.  The first variation along the
		positive outer sphere gives
		\[
		\partial_\theta\{-d_g(x,Y_R^+(\theta))+d_g(Y_R^+(\theta),y)\}
		\big|_{\theta=\theta_*}=0,
		\]
		or equivalently
		\[
		\partial_\theta\{\Phi_R^+(x,\theta)+S_r^-(\theta,\xi)\}
		\big|_{\theta=\theta_*}=0.
		\]
		By the uniqueness part of Lemma~\ref{lem:signed-critical-geometry},
		$\theta_* = \Theta_r^-(x,\xi)$.
		
		Finally, distance additivity along the minimizing geodesic segment
		$y\to x\to z$ gives
		\[
		d_g(z,y)=d_g(z,x)+d_g(x,y).
		\]
		Therefore
		\[
		\begin{aligned}
			\Phi_R^+(x,\Theta_r^-(x,\xi))
			+S_r^-(\Theta_r^-(x,\xi),\xi)
			&=-d_g(x,z)+d_g(z,y)\\
			&=d_g(x,y)=\Phi_r^-(x,\xi),
		\end{aligned}
		\]
		which proves \eqref{eq:signed-critical-value} for $\sigma=-1$.
	\end{proof}
	
	\subsection{Signed factorization and signed two-radius estimate}
	
	\begin{proposition}[Signed factorization]\label{prop:signed-factorization}
		For $\sigma=\pm1$ define
		\[
		U_{\lambda,r}^\sigma h(\theta)
		=\lambda\int_{K_2} e^{i\lambda S_r^\sigma(\theta,\xi)}\chi(\theta,\xi)h(\xi)\,d\xi.
		\]
		Then
		\begin{equation}\label{eq:signed-angular-bound}
			\|U_{\lambda,r}^\sigma h\|_2\le C_\eps\|h\|_{L^2(K_2)}.
		\end{equation}
		Moreover, for every bounded $C^\infty$ order-zero amplitude $a_r$ supported in
		$K_1\times K_2$, define
		\[
		T_{\lambda,r}^{\sigma,a}h(x)
		=\lambda\int_{K_2} e^{i\lambda\Phi_r^\sigma(x,\xi)}a_r(x,\xi,\lambda)h(\xi)\,d\xi .
		\]
		Then
		\begin{equation}\label{eq:signed-factorization-formula}
			T_{\lambda,r}^{\sigma,a}h
			=E_{\lambda,R}^{A_r^\sigma}U_{\lambda,r}^\sigma h
			+R_{\lambda,r}^\sigma h,
		\end{equation}
		where the outer operator always has the positive outer phase,
		\[
		E_{\lambda,R}^{A}q(x)
		=\lambda\int e^{i\lambda\Phi_R^+(x,\theta)}A(x,\theta,\lambda)q(\theta)\,d\theta,
		\]
		and
		\begin{equation}\label{eq:signed-factorization-remainder}
			\|R_{\lambda,r}^\sigma h\|_{L^\infty(U)}+
			\|R_{\lambda,r}^\sigma h\|_{L^2(U)}
			\le C_\eps\|h\|_{L^2(K_2)}.
		\end{equation}
		Here \(A_r^\sigma\) belongs to a bounded \(C^\infty\) order-zero class in the
active variables \((x,\theta)\), uniformly in \(\sigma=\pm1\), \(r\in I\),
and \(\lambda\ge1\).  
	\end{proposition}
	
	\begin{proof}
		The angular estimate is the proof of Lemma~\ref{lem:angular}, with
		\eqref{eq:signed-xi-nonstationary} replacing \eqref{eq:xi-nonstationary}.  The
		same cutoff satisfies the same segment condition for both signs, so the Schur
		kernel estimate is unchanged.
		
		For the factorization, set
		\[
		H_r^\sigma(x,\xi)
		=\partial_{\theta\theta}^2\Psi_r^\sigma
		(x,\Theta_r^\sigma(x,\xi),\xi)
		\]
		and
		\[
		q_r^\sigma(x,\xi)
		=(2\pi)e^{i\pi\sig(H_r^\sigma(x,\xi))/4}
		|\det H_r^\sigma(x,\xi)|^{-1/2}.
		\]
		By \eqref{eq:signed-Psi-Hess-lower}, $|q_r^\sigma|\approx\eps^{-1}$, and all
		needed derivatives have bounds depending only on the fixed $C^\infty$ data and on $\eps$.  Keep the same $x$-cutoff $\psi_1$.  Choose
		$\psi_{\rm sgn}\in C_0^\infty(W)$ equal to one on a fixed
		neighborhood of
		\[
		\{\Theta_r^\sigma(x,\xi):\sigma=\pm1,\ x\in 
		K_1,\ r\in I,\ \xi\in K_2\}.
		\]
		Define
		\[
		A_r^\sigma(x,\theta,\lambda)
		:=\psi_1(x)\psi_{\rm sgn}(\theta)
		\{q_r^\sigma(x,\Xi_r^\sigma(x,\theta))\}^{-1}
		a_r(x,\Xi_r^\sigma(x,\theta),\lambda)
		\]
		wherever $\psi_1(x)\psi_{\rm sgn}(\theta)\ne0$, and set it equal to zero
		elsewhere.  This is a smooth order-zero amplitude because
		$\supp\psi_1\Subset U_0$, because
		$\supp\psi_{\rm sgn}\Subset W$ where $\Xi_r^\sigma$ is defined,
		and because $q_r^\sigma$ is elliptic.
		
		The composition kernel of
		$E_{\lambda,R}^{A_r^\sigma}U_{\lambda,r}^\sigma$ is
		\[
		\lambda^2\int e^{i\lambda\Psi_r^\sigma(x,\theta,\xi)}
		A_r^\sigma(x,\theta,\lambda)\chi(\theta,\xi)\,d\theta.
		\]
		Apply Proposition~\ref{prop:sogge-stationary} in its final general form, with
		$\Psi_r,\Theta_r,\Phi_r,H_r$ replaced by
		$\Psi_r^\sigma,\Theta_r^\sigma,\Phi_r^\sigma,H_r^\sigma$.  The required
		unique critical point, collar, cutoff, nonstationary estimate, derivative
		bounds, definite Hessian, and exact critical value are precisely the contents
		of Lemmas~\ref{lem:signed-critical-geometry} and
		\ref{lem:signed-critical-value}.  After multiplying by the outside factor
		$\lambda^2$, the composed kernel equals
		\[
		\lambda e^{i\lambda\Phi_r^\sigma(x,\xi)}
		q_r^\sigma(x,\xi)A_r^\sigma(x,\Theta_r^\sigma(x,\xi),\lambda)
		+\mathcal K_{\lambda,r}^\sigma(x,\xi),
		\]
		where
		\[
		\|\mathcal K_{\lambda,r}^\sigma\|_{L^\infty(\overline{U_0}\times K_2)}
		\le C_\eps .
		\]
		For $x\in K_1$ and $\xi\in K_2$,
		\[
		\psi_1(x)=1,
		\qquad
		\psi_{\rm sgn}(\Theta_r^\sigma(x,\xi))=1,
		\qquad
		\Xi_r^\sigma(x,\Theta_r^\sigma(x,\xi))=\xi.
		\]
		Therefore
		\[
		q_r^\sigma(x,\xi)A_r^\sigma(x,\Theta_r^\sigma(x,\xi),\lambda)
		=a_r(x,\xi,\lambda),
		\]
		and the leading stationary-phase term is
		\[
		\lambda e^{i\lambda\Phi_r^\sigma(x,\xi)}a_r(x,\xi,\lambda).
		\]
		As in Proposition~\ref{prop:factorization}, the kernel of
\(T_{\lambda,r}^{\sigma,a}-E_{\lambda,R}^{A_r^\sigma}U_{\lambda,r}^\sigma\)
is reduced to the stationary-phase remainder, which is uniformly \(O_\varepsilon(1)\) on
\(U\times K_2\).  This proves \eqref{eq:signed-factorization-formula}--
		\eqref{eq:signed-factorization-remainder}. The stated symbol bounds for \(A_r^\sigma\) follow by differentiating its
definition and using the uniform bounds for \(q_r^\sigma\),
\((q_r^\sigma)^{-1}\), \(\Xi_r^\sigma\), and \(a_r\).
	\end{proof}
	
	\begin{theorem}[Signed frozen two-radius estimate]\label{thm:signed-main}
		After the common choices above, for every
		$\sigma_1,\sigma_2\in\{+1,-1\}$, every $r_1,r_2\in I$, and every
		$1\le\alpha\le\beta$,
		\[
		\bigl\|T_{\alpha,r_1}^{\sigma_1,a_1}f\,
		T_{\beta,r_2}^{\sigma_2,a_2}g\bigr\|_{L^2(U)}
		\le C_\eps
		\alpha^{1/2}
		\|f\|_{L^2(K_2)}\|g\|_{L^2(K_2)}.
		\]
	\end{theorem}
	
	\begin{proof}
		Use the signed factorization \eqref{eq:signed-factorization-formula} for both
		factors.  The main-main term has the common positive outer phase $\Phi_R^+$,
		so Proposition~\ref{prop:singlephase} applies exactly as in Section~\ref{sec6}.  The
		signed angular bounds \eqref{eq:signed-angular-bound} control the inputs.  All
		remainders are harmless and only contribute $O(1)$.
		This proves Theorem~\ref{thm:signed-main}.
	\end{proof}

	\section{Log-free  spectral cluster estimates}\label{sec8}
	
	In this section, we prove the log-free multilinear estimates for spectral clusters on any closed Riemannian manifold. The proof uses Sogge's spectral cluster parametrix and reduces to frozen radial pieces, which are then treated by Theorems \ref{thm1}-\ref{thm3} and Theorem \ref{thm:signed-main}.
	
 As in \cite[Lemma 2.2]{bgt2005}, a first reduction in the proof of Theorem \ref{thm4} is that it suffices to prove it for one fixed nontrivial function $\chi$. 

\begin{lemma}\label{chilem}
	Suppose that the assertion of Theorem \ref{thm4} holds for a bump function $\chi_0 \in \mathcal{S}(\mathbb{R})$
	which is not identically zero. Then it holds for any other choice of the bump function.
\end{lemma}

\begin{proof}
	Let \(P=\sqrt{-\Delta}\).  Choose \(t_0\) with \(\chi_0(t_0)\ne0\), and
	then choose \(h>0\) so small that \(\chi_0\) is bounded away from zero on
	\((t_0-2h,t_0+2h)\).  Consider a partition of unity of the real line given by
	\(\rho\in C_0^\infty((-h,h))\) with
	\[
	\sum_{m\in h\mathbb Z}\rho(t-t_0-m)=1 .
	\]
	For \(\widetilde\chi\in\mathcal S(\mathbb R)\), set
	\[
	b_m(t)=
	\frac{\rho(t-t_0-m)\widetilde\chi(t)}{\chi_0(t-m)},
	\qquad m\in h\mathbb Z .
	\]
	The denominator is harmless on the support of the numerator, and rapid decay
	of \(\widetilde\chi\) gives
	\[
	\widetilde\chi(t)=\sum_{m\in h\mathbb Z}b_m(t)\chi_0(t-m),
	\qquad
	\|b_m\|_\infty\le C_N(1+|m|)^{-N}.
	\]
	
	Expanding the product, fix \(m_1,\ldots,m_k\in h\mathbb Z\), and write
	\[
	\nu_j=2+|\lambda_j+m_j|,\qquad
	g_j=b_{m_j}(P-\lambda_j)f_j .
	\]
	Up to a rearrangement in the increasing order $\nu_1\le \dots \le \nu_k$, the assumed estimate for \(\chi_0\) gives
	\[
	\left\|
	\prod_{j=1}^k\chi_0(P-\lambda_j-m_j)g_j
	\right\|_{L^p}
	\lesssim
	\mathcal E_0(n,k,p;\nu_1,\ldots,\nu_k)
	\prod_{j=1}^k\|g_j\|_2 .
	\]
	Because
	\[ \frac{2+\lambda_j}{(1+|m_j|)} \leq \nu_j \leq (1+|m_j|) (2+\lambda_j),\]
	from the explicit formula for \(\mathcal E_0\), there is some \(N_0=N_0(n,k,p)\) such that
	\[
	\mathcal E_0(n,k,p;\nu_1,\ldots,\nu_k)
	\lesssim
	\mathcal E_0(n,k,p;\lambda_1,\ldots,\lambda_k)
	\prod_{j=1}^k(1+|m_j|)^{N_0} .
	\]
	Therefore,
	\[
	\left\|
	\prod_{j=1}^k\widetilde\chi(P-\lambda_j)f_j
	\right\|_{L^p}
	\lesssim
	\mathcal E_0(n,k,p;\lambda_1,\ldots,\lambda_k)
	\prod_{j=1}^k\|f_j\|_2       
	\left(\sum_{m_1,\ldots,m_k}
	\prod_{j=1}^k(1+|m_j|)^{N_0-N}\right).
	\]
	Choosing \(N>N_0+2\), the last sum is finite.  This proves the desired estimate
	for \(\widetilde\chi\).
\end{proof}

Following \cite[Chapter 4]{fio}, thanks to Lemma \ref{chilem}, it is sufficient to prove Theorem \ref{thm4} with \( \chi \) such that \( \hat{\chi}(\tau) \) is supported in the set
\[
\{ \tau \in \mathbb{R} : \varepsilon \leq \tau \leq 2\varepsilon \},
\]
where \( \varepsilon > 0 \) is a small number to be determined later. We can write
\[
\chi_{\lambda} f = \frac{1}{2\pi} \int_{\varepsilon}^{2\varepsilon} e^{-i\lambda\tau} \hat{\chi}(\tau) (e^{i\tau \sqrt{-\Delta}} f) \, d\tau.
\]

For \( \varepsilon \ll 1 \) and \( |\tau| \leq 2\varepsilon \), a partition of unity on \( M \) allows us to represent \( e^{i\tau \sqrt{-\Delta}} \) as a Fourier integral operator (see e.g. \cite{hor1968}). Therefore \( \chi_{\lambda} \) can also be represented as a Fourier integral operator. After a stationary phase argument (see \cite[Chapter 5]{fio}), we can represent \( \chi_{\lambda} f \) as follows.

\begin{lemma}\label{soggelem}
	There exists \( \varepsilon_0 > 0 \) such that for every $0<\varepsilon<\varepsilon_0$ and every \( N \geq 1 \), we have the splitting
	\[
	\chi_{\lambda} f = S_{\lambda} f + R_{\lambda} f,
	\]
	with
	\[
	\| R_{\lambda} f \|_{H^k(M)} \leq C_{N,k} \lambda^{k-N} \| f \|_{L^2(M)}, \quad k = 0, \ldots, N.
	\]
	Moreover there exist \( \tau > 0 \) and, for every \( x_0 \in M \), a system of coordinates \( U \subset \mathbb{R}^n \), containing \( 0 \in \mathbb{R}^n \), such that for \( x \in U \), \( |x| \leq \tau\eps \),
	\[
	S_{\lambda} f(x) = \lambda^{\frac{n-1}{2}}\int_{\mathbb{R}^n} e^{i\lambda \varphi(x,y)} a(x,y,\lambda) f(y) \, dy,
	\]
	where \( a(x,y,\lambda) \) is a polynomial in \( \lambda^{-1} \) with smooth coefficients supported in the set
	\[
	\{(x,y) \in U \times U : |x| \leq \tau\eps \ll \varepsilon / C \ll |y| \ll C\varepsilon\}
	\]
	and \( -\varphi(x,y) = d_g(x,y) \) is the geodesic distance between \( x \) and \( y \).
\end{lemma}
In view of Lemma \ref{soggelem}, to prove \eqref{eig}, it is enough to show
\begin{equation}\label{tb}
	\left\|\prod_{j=1}^k S_{\lambda_j}f_j\right\|_{L^p(M)}
	\ls\,\E_0(n,k,p;\lambda_1,\ldots,\lambda_k)\prod_{j=1}^k\|f_j\|_{L^2(M)} .
\end{equation}
The products involving remainders $R_{\la_j}$ are straightforward to estimate.

Next, we represent \( y \) in geodesic coordinates as \( y = \exp_o(r\omega) \), \( r > 0 \), \( \omega \in S^{n-1} \). For \( |x| \leq \tau\eps \) and \( \omega \in S^{n-1} \), we define the frozen phase \( \varphi_r \),
\[
\varphi_r(x, \omega) = \varphi(x, \exp_o(r\omega)).
\]
Using Gauss's lemma, there exists $\eps>0$ such that for every $r\in [\eps/C,C\eps]$ and every $\omega\in S^{n-1}$, we have
\begin{equation}\label{gauss}
	\nabla_x\varphi_r(0,\omega)=\omega.
\end{equation}

Fix a point \( \underline{\omega} \in S^{n-1} \). The set
\[
S_{\varphi_r}(x)= \left\{ \nabla_x \varphi_r(x, \omega): \omega \in S^{n-1}, \omega \sim \underline{\omega} \right\}
\]
is a smooth hypersurface in \( \mathbb{R}^n \). Indeed, assuming  \( \underline{\omega} = (1, 0, \ldots, 0) \), then \( (w_1 = \omega_2, \ldots, w_{n-1} = \omega_n) \) is a system of coordinates on \( S^{n-1} \) and according to \eqref{gauss}, \( \nabla_w \nabla_x \varphi_r \) has rank \( n-1 \).
\begin{lemma}\label{ellip}
	The hypersurface \( 	S_{\varphi_r}(x) \) has uniformly definite second fundamental form: for any local coordinate system \( w \in \mathbb{R}^{n-1} \)   near \( \underline{\omega} \in S^{n-1} \), if we denote by \( \pm N(x, w) \) the normal unit vectors to the surface \( 	S_{\varphi_r}(x) \) at the point \( \nabla_x \varphi_r(x, w) \), then for \( x \) close to 0,
	\[
	\big\langle \partial^2_{ww} \nabla_x \varphi_r(x, w), N(x, w) \big\rangle
	\]
	is definite.
\end{lemma}

\begin{proof}
	Set
	\[
	\II_{ij}(r,x,w)=\big\langle \partial_{w_iw_j}^2\nabla_x\varphi_r(x,w),N(x,w)\big\rangle .
	\]
	In normal coordinates, as discussed above,Gauss's lemma gives
	\(
	\nabla_x\varphi_r(0,w)=\omega(w), \ \forall r\in[\eps/C,C\eps].
	\)
	For a local choice of normal, write
	\(N(0,w)=\sigma\omega(w),\) where \( \sigma\in\{\pm1\}.\)
	
	Since \(|\omega(w)|^2=1\), differentiating gives
	\[
	\II_{ij}(r,0,w)
	=-\sigma\,\langle\partial_{w_i}\omega,\partial_{w_j}\omega\rangle .
	\]
	Thus, for \(\zeta\in\mathbb R^{n-1}\),
	\[
	-\sigma\sum_{i,j}\II_{ij}(r,0,w)\zeta_i\zeta_j
	=
	\sum_{i,j}\langle\partial_{w_i}\omega,\partial_{w_j}\omega\rangle\zeta_i\zeta_j 
	=\left|\sum_i\zeta_i\partial_{w_i}\omega\right|^2
	\ge c|\zeta|^2 .
	\]
	By smoothness, uniformly for \(r\in[\eps/C,C\eps]\) and \(|x|\le\tau\eps\),
	\[
	\|\II(r,x,w)-\II(r,0,w)\|\le c/2.
	\]
	Hence
	\[
	-\sigma\sum_{i,j}\II_{ij}(r,x,w)\zeta_i\zeta_j
	\ge \frac c2|\zeta|^2,
	\]
	which is the claimed uniform definiteness.
\end{proof}

We note that there exists a smooth positive function \( \kappa(r, \omega) \) such that \( dy = \kappa(r, \omega) \, dr \, d\omega \).  
For \( r \in [\varepsilon / C, C\varepsilon] \) and \( \lambda \ge 1 \), define the operator acting on functions on \( S^{n-1} \) by
\[
(T_{\lambda,\varphi_r}^{a_r} f)(x)
=\lambda^{\frac{n-1}{2}}\int_{S^{n-1}} e^{i\lambda \varphi_r(x, \omega)} a_r(x, \omega, \lambda) f(\omega) \, d\omega,
\]
where \( a_r(x, \omega, \lambda) = \kappa(r, \omega) a(x, \exp_o(r\omega), \lambda) \).  Then
\[
S_\lambda f(x) = \int_{\eps/C}^{C\eps} (T_{\lambda,\varphi_r}^{a_r} f_r)(x) \, dr,
\]
where \( f_r(\omega) = f(r, \omega) \). Minkowski's inequality shows that \eqref{tb} will be a consequence of 
\begin{equation}\label{tb2}
	\left\|\prod_{j=1}^k T_{\lambda_j,\varphi_{r_j}}^{a_{r_j}}f_j\right\|_{L^p(M)}
	\ls\,\E_0(n,k,p;\lambda_1,\ldots,\lambda_k)\prod_{j=1}^k\|f_{j}\|_{L^2(S^{n-1})}.
\end{equation}
uniformly for $1\ll\la_1\le ...\le \la_k$ and $r_1,...,r_k\in [\eps/C,C\eps]$.  Therefore, by Lemma \ref{ellip}, the bounds for Carleson-Sj\"olin operators in Theorems \ref{thm1}--\ref{thm3} give \eqref{tb2}, except for the case $n=3$ and $p\le 2$. We shall prove log-free estimates in this case separately in the following.

\subsection{Log-free estimates in three dimensions}Let $M$ be a closed Riemannian manifold of dimension $n=3$. We shall use  Theorem~\ref{thm:signed-main} to establish log-free endpoint estimates in three dimensions.
	\begin{lemma}[Transverse frozen cap estimate]\label{lem:transverse-frozen}
		Let
		\[
		T_{\lambda,j}h(x)
		=\lambda\int e^{i\lambda\phi_j(x,\xi)}b_j(x,\xi,\lambda)h(\xi)\,d\xi,
		\qquad j=1,2,
		\]
		be oscillatory integral operators on a fixed compact patch
		$U\Subset\mathbb R^3$, with $\xi\in V_j\Subset\mathbb R^2$.  Assume the phases
		and order-zero amplitudes form a bounded $C^\infty$ family on the supports.  Put
		\[
		P_j(x,\xi)=\nabla_x\phi_j(x,\xi).
		\]
		Assume $\operatorname{rank}\partial_\xi P_j=2$ and that the associated covector
		surfaces are uniformly transverse:
		\begin{equation}\label{eq:transverse-normal-lower}
			|n_1(x,\xi)\times n_2(x,\eta)|\ge c_0>0,
		\end{equation}
		where $n_j$ is a smooth unit normal to $P_j(x,V_j)$.  Then, for
		$1\le\alpha\le\beta$,
		\begin{equation}\label{eq:transverse-frozen-estimate}
			\|T_{\alpha,1}f\,T_{\beta,2}g\|_{L^2(U)}
			\le C\alpha^{1/2}\|f\|_2\|g\|_2 .
		\end{equation}
		The constant is uniform for bounded $C^\infty$ families satisfying these hypotheses.
	\end{lemma}
	
	\begin{proof}
		It suffices to prove the linearized estimate
		\[
		B_{\alpha,\beta}F(x)
		=\alpha\beta\iint e^{i\alpha\phi_1(x,\xi)+i\beta\phi_2(x,\eta)}
		b_1(x,\xi,\alpha)b_2(x,\eta,\beta)F(\xi,\eta)\,d\xi d\eta .
		\]
		After a finite subdivision of the frequency supports, transversality allows us
		to write $\xi=(s,t)$ so that
		\begin{equation}\label{eq:transverse-minor}
			\left|\det\bigl(\partial_sP_1(x,\xi),\partial_{\eta_1}P_2(x,\eta),
			\partial_{\eta_2}P_2(x,\eta)\bigr)\right|\ge c_1>0
		\end{equation}
		on the support.  Indeed, the two $\eta$-columns span the tangent plane of the
		second covector surface, and \eqref{eq:transverse-normal-lower} says that some
		tangent direction of the first surface is uniformly transverse to that plane.
		
		Set $\sigma=\alpha/\beta$ and $u=\sigma s$.  For fixed $t$, put
		$y=(u,\eta)$ and
		\[
		\Psi_t(x,u,\eta)=\sigma\phi_1(x,u/\sigma,t)+\phi_2(x,\eta).
		\]
		Then
		\[
		\nabla_x\Psi_t=\sigma P_1(x,u/\sigma,t)+P_2(x,\eta),
		\]
		and differentiation with respect to $(u,\eta_1,\eta_2)$ gives the columns in
		\eqref{eq:transverse-minor}.  Thus, after the subdivision and uniformly in
		$\sigma$ and $t$,
		\begin{align*}
			|\nabla_x\Psi_t(x,y)-\nabla_x\Psi_t(x,y')|&\ge c|y-y'|,
			\\
			|\partial_x^\gamma\{\nabla_x\Psi_t(x,y)-\nabla_x\Psi_t(x,y')\}|&\le C_\gamma |y-y'|,
			\qquad |\gamma|\le N,
		\end{align*}
		with $N>3$.  The change $u=\sigma s$ removes the small factor $\sigma$ from the
		first column.
		
		For fixed $t$, let
		\[
		S_{\beta,t}h(x)=\int e^{i\beta\Psi_t(x,u,\eta)}
		c_t(x,u,\eta,\alpha,\beta)h(u,\eta)\,du d\eta,
		\]
		where $c_t=b_1(x,u/\sigma,t,\alpha)b_2(x,\eta,\beta)$.  Lemma~\ref{lem:l2},
		applied in the three variables $(u,\eta)$, gives
		\[
		\|S_{\beta,t}h\|_{L^2_x}\le C\beta^{-3/2}\|h\|_2
		\]
		uniformly in $t$ and $\sigma$.  Since $ds=\sigma^{-1}du$ and
		$\alpha\beta/\sigma=\beta^2$,
		\[
		B_{\alpha,\beta}F(x)=\beta^2\int S_{\beta,t}H_t(x)\,dt,
		\qquad H_t(u,\eta)=F(u/\sigma,t,\eta).
		\]
		Minkowski and Cauchy's inequality in the bounded $t$-interval give
		\[
		\|B_{\alpha,\beta}F\|_2
		\le C\beta^{1/2}
		\left(\int\|H_t\|_2^2\,dt\right)^{1/2}
		=C\beta^{1/2}\sigma^{1/2}\|F\|_2
		=C\alpha^{1/2}\|F\|_2.
		\]
		Taking $F=f\otimes g$ proves \eqref{eq:transverse-frozen-estimate}.
	\end{proof}
	
	\begin{proposition}\label{lem:local-main-piece}
		Let $o$ be the center of a sufficiently small normal-coordinate ball and set
		$U=\{|x|<\tau\eps\}$ in these coordinates.  Suppose
		\[
		S^{(j)}_\lambda f(x)
		=\lambda\int e^{-i\lambda d_g(x,y)}a_j(x,y,\lambda)f(y)\,dy,
		\qquad j=1,2,
		\]
		where the amplitudes are supported in $x\in K_1\Subset U$, where
		\[
		d_g(o,y)\approx\eps,
		\qquad d_g(x,y)\approx\eps,
		\]
		and where the amplitudes form a bounded order-zero family.  Then, for
		$1\le\alpha\le\beta$,
		\begin{equation}\label{eq:local-main-piece-estimate}
			\|S^{(1)}_\alpha f\,S^{(2)}_\beta g\|_{L^2(U)}
			\le C\alpha^{1/2}\|f\|_{L^2(M)}\|g\|_{L^2(M)} .
		\end{equation}
	\end{proposition}
	
	\begin{proof}
		Use geodesic polar coordinates around $o$ on the $y$-annulus.  After finite
		radial and angular partitions of unity, each angular piece has the form
		\begin{equation}\label{eq:radial-angular-piece}
			S_{\lambda,\nu}^{(j)}f_j(x)
			=\int_I \bigl(T^{(j,\nu)}_{\lambda,r}h^{(j,\nu)}_r\bigr)(x)\,dr,
		\end{equation}
		where $I\subset[\eps/C,C\eps]$ is fixed,
		$Y_\nu(r,\xi)=\exp_o(r\Omega_\nu(\xi))$, and
		\[
		T^{(j,\nu)}_{\lambda,r}h(x)
		=\lambda\int e^{-i\lambda d_g(x,Y_\nu(r,\xi))}
		b^{(j,\nu)}_r(x,\xi,\lambda)h(\xi)\,d\xi .
		\]
		Splitting the polar Jacobian symmetrically between the amplitude and the input
		gives
		\begin{equation}\label{eq:polar-L2-control}
			\sum_\nu\int_I\|h^{(j,\nu)}_r\|_{L^2_\xi}^2\,dr
			\le C_\eps\|f_j\|_{L^2(M)}^2,
			\qquad f_1=f,\quad f_2=g .
		\end{equation}
		The frozen amplitudes satisfy the symbol bounds required in the frozen
		estimates.
		
		We claim that every frozen cap pair satisfies
		\begin{equation}\label{eq:frozen-uniform-cap-pair}
			\|T^{(1,\nu_1)}_{\alpha,r_1}f\,
			T^{(2,\nu_2)}_{\beta,r_2}g\|_{L^2(U)}
			\le C\alpha^{1/2}\|f\|_2\|g\|_2,
		\end{equation}
		for all $r_1,r_2\in I$.  The angular partition is chosen fine enough that each
		cap pair is either parallel/antipodal or uniformly separated from both of these
		relations.  In the first case there is a common angular chart and there are signs
		$\sigma_j\in\{+1,-1\}$ such that, after conjugating the factors with
		$\sigma_j=-1$ if necessary, the phases are precisely the signed phases
		$\Phi_{r_j}^{\sigma_j}$.  Theorem~\ref{thm:signed-main} then gives
		\eqref{eq:frozen-uniform-cap-pair}.
		
		In the remaining case, the two direction sets are separated from both
		$\omega_1=\omega_2$ and $\omega_1=-\omega_2$.  The scaled distance estimate gives
		\[
		\nabla_x[-d_g(x,\exp_o(r\omega))]=\omega+O(\tau+\eps^2)
		\]
		with smooth dependence on $\omega$.  Thus, after taking $\tau$ and $\eps$ small,
		the associated covector surfaces have uniformly transverse normal fields.
		Lemma~\ref{lem:transverse-frozen} gives \eqref{eq:frozen-uniform-cap-pair} for
		these cap pairs as well.
		
		Finally combine \eqref{eq:radial-angular-piece}, \eqref{eq:frozen-uniform-cap-pair},
		Minkowski's inequality, Cauchy's inequality in the two radial variables, and
		\eqref{eq:polar-L2-control}.  For each cap pair,
		\[
		\begin{aligned}
			\|S_{\alpha,\nu_1}^{(1)}f\,S_{\beta,\nu_2}^{(2)}g\|_{L^2(U)}
			&\le \int_I\int_I
			\|T^{(1,\nu_1)}_{\alpha,r_1}h^{(1,\nu_1)}_{r_1}\,
			T^{(2,\nu_2)}_{\beta,r_2}h^{(2,\nu_2)}_{r_2}\|_2\,dr_1dr_2 \\
			&\le C\alpha^{1/2}
			\left(\int_I\|h^{(1,\nu_1)}_{r_1}\|_2^2\,dr_1\right)^{1/2}
			\left(\int_I\|h^{(2,\nu_2)}_{r_2}\|_2^2\,dr_2\right)^{1/2}.
		\end{aligned}
		\]
		Summing over the finitely many cap pairs proves
		\eqref{eq:local-main-piece-estimate}.
	\end{proof}

 		By Lemmas~\ref{chilem}-\ref{soggelem} we may choose $\chi\in\mathcal{S}(\mathbb{R})$ so that
	\[
	\chi_\lambda=S_\lambda+R_{\lambda}.
	\]
	The remainder is harmless. It remains to estimate $S_{\lambda}f\,S_{\mu}g$.    Both operators satisfy the hypotheses of Proposition~\ref{lem:local-main-piece}.  Consequently, for $1\le\lambda\le\mu$,
	\[
	\|S_{\lambda}f\,S_{\mu}g\|_{L^2(M)}
	\ls\lambda^{1/2}\|f\|_2\|g\|_2 .
	\]
	This proves  the log-free endpoint bilinear estimate
	\begin{equation}\label{bi}
		\left\| \chi_{\lambda}f\chi_\mu g\right\|_{L^2(M)}
		\lesssim
		\la^{\frac12}
		\|f\|_{L^2(M)}\|g\|_{L^2(M)}.
	\end{equation}
	
For $k\ge3$, $ 1\le p\le2$, $1\ll\lambda_1\le\cdots\le\lambda_k$.
		Write
		$u_j=\chi_{\lambda_j}f_j$ with $\|f_j\|_{L^2(M)}=1$.  By interpolation, it suffices to prove the log-free endpoint estimates.  The standard linear bounds are uniform for this bounded amplitude family:
		\[
		\|u_j\|_{L^\infty}\lesssim \lambda_j,
		\qquad
		\|u_j\|_{L^4}\lesssim \lambda_j^{1/4}.
		\]
		The bilinear endpoint from \eqref{bi}  gives,
		for $\lambda_i\le\lambda_j$, $\|u_i u_j\|_{L^2}\lesssim \lambda_i^{1/2}.$
		At $p=1$,
		\[
		\|u_1u_2\cdots u_k\|_{L^1}
		\lesssim
		\prod_{j=1}^{k-4}\|u_j\|_{L^\infty}
		\|u_{k-3}u_{k-1}\|_{L^2}
		\|u_{k-2}u_k\|_{L^2}
		\lesssim
		\prod_{j=1}^{k-4}\lambda_j\,
		\lambda_{k-3}^{1/2}\lambda_{k-2}^{1/2}.
		\]
		At $p=4/3$,
		\[
		\|u_1u_2\cdots u_k\|_{L^{4/3}}
		\lesssim
		\prod_{j=1}^{k-4}\|u_j\|_{L^\infty}
		\|u_{k-3}\|_{L^\infty}
		\|u_{k-2}u_k\|_{L^2}
		\|u_{k-1}\|_{L^4}
		\lesssim
		\prod_{j=1}^{k-4}\lambda_j\,
		\lambda_{k-3}\lambda_{k-2}^{1/2}\lambda_{k-1}^{1/4}.
		\]
		At $p=2$,
		\[
		\|u_1u_2\cdots u_k\|_{L^2}
		\lesssim
		\prod_{j=1}^{k-3}\|u_j\|_{L^\infty}
		\|u_{k-2}\|_{L^\infty}
		\|u_{k-1}u_k\|_{L^2}
		\lesssim
		\prod_{j=1}^{k-3}\lambda_j\,
		\lambda_{k-2}\lambda_{k-1}^{1/2}.
		\]

	\section{Proof of multilinear oscillatory integral estimates}\label{sec10}
In this section, we prove the upper bounds in Theorems \ref{thm1}--\ref{thm3}.	Fix $n\ge2$, $k\ge2$. Let $d=n-1$. Let $u_j=T_{\la_j,\phi_j}^{a_j}f_j$ and $\|f_j\|_{L^2(\mathbb{R}^{d})}=1$ for $j=1,2,\ldots,k$.  To give a unified argument, we use the convention that the empty product is \(1\) and \(\lambda_j=2\) for all \(j\le0\).
	
	For $d\ge3$, set
	\[
	q_0=\frac{2(d+2)}{d-2},\qquad
	r_0=\frac{2d(d+2)}{d^2+4},\qquad
	q_1=\frac{2d}{d-2}.
	\]
	For $d\ge2$, set $q_2=\frac{2(d+2)}d$.  Then, for $d\ge3$,
	\[\quad \frac1{p_0}=\frac1{q_0}+\frac12,\quad
	\frac1{p_0}=\frac1{q_1}+\frac1{r_0},\quad \frac1{p_1}=\frac1{q_1}+\frac12, \]
	and for $d\ge2$,
	\[
	\frac1{p_2}=\frac1{q_2}+\frac12.
	\]
	In this section, for $1\le p\le 2$, we first prove bilinear and multilinear estimates for $d\ge3$, and then prove them for $d=1$ and $d=2$ separately. Finally, we prove similar estimates for $p>2$.
	
	\subsection{Keel-Tao's theorem}\label{sec:keel-tao-upper}
	From now on set $d=n-1$.  The proof is local in phase space. We first reduce the operator to the usual Carleson--Sj\"olin normal form.

			\begin{lemma}[Common normal coordinates]\label{lem:normal-form-products}
			Fix $k\ge2$.  Let $T_{\lambda_j,\phi_j}^{a_j}$, $1\le j\le k$, be Carleson--Sj\"olin operators, with amplitudes $a_j$ supported in fixed compact coordinate patches.  After a finite decomposition in $y$ and in each of the $k$ frequency variables, each localized piece of the product $\prod_{j=1}^kT_{\lambda_j,\phi_j}^{a_j}f_j$ admits orthonormal
			coordinates
			$y=(s,z)\in\mathbb R\times\mathbb R^d$ common to all $j$, such that
			\[
			\det\partial^2_{z\xi}\phi_j(s,z,\xi)\ne0
			\]
			on the support of its localized amplitude.  In these coordinates each hypersurface $S_{\phi_j}(s,z)$
			is locally a graph $\{\bigl(H^{(j)}_{s,z}(\zeta),\zeta\bigr)\},$
			and the curvature condition  is equivalent to 
			\begin{equation}\label{curcond}
				\det\partial^2_{\zeta\zeta}H^{(j)}_{s,z}(\zeta)\ne0.
			\end{equation}
		
		\end{lemma}

\begin{proof}
The assertion is local in the product support.  Fix
\[
y_0\in \R^{d+1},\qquad \eta_j^0\in \R^d,\quad j=1,\ldots,k,
\]
with \((y_0,\eta_j^0)\in\operatorname{supp}a_j\).  By the
 rank condition, the image
\[
\operatorname{Image}\partial_\xi\nabla_y\phi_j(y_0,\eta_j^0)
\subset \R^{d+1}
\]
is a \(d\)-dimensional subspace for each \(j\).  Choose a unit vector \(\ell\)
outside the finite union of these subspaces.  Let
\(E:\R^d\to \ell^\perp\subset\R^{d+1}\) be an isometry, and write
\[
y=y_0+s\ell+Ez,\qquad s\in\R,\quad z\in\R^d .
\]
Then
\[
\partial_{z\xi}^2\phi_j
=
E^T\partial_\xi\nabla_y\phi_j .
\]
Since \(\ker E^T=\R\ell\) and \(\ell\notin
\operatorname{Image}\partial_\xi\nabla_y\phi_j(y_0,\eta_j^0)\), the projection
\(E^T\) is injective on this image.  Hence, using the rank condition again,
\[
\det \partial_{z\xi}^2\phi_j(0,0,\eta_j^0)\ne0 .
\]
After shrinking the supports, this remains true throughout the localized
piece.  Compactness gives the required finite decomposition.

On each such piece, the map $\xi\mapsto \zeta=\partial_z\phi_j(s,z,\xi)$
is a local diffeomorphism.  Using \(\zeta\) as the new frequency variable, we
may write
\[
\partial_s\phi_j(s,z,\xi)=H^{(j)}_{s,z}(\zeta).
\]
Therefore
\[
S_{\phi_j}(s,z)
=
\{(H^{(j)}_{s,z}(\zeta),\zeta)\}.
\]
So  the curvature condition is equivalent to
\[
\det\partial_{\zeta\zeta}^2H^{(j)}_{s,z}\ne0.
\]
\end{proof}

	On such common normal coordinates, each localized factor has the form
	\begin{equation}\label{Tla}
		T_{\la,\Phi} u(s, z) = \la^{d/2} \int_{\mathbb{R}^d} e^{i\la\Phi(s, z, \xi)} b(s, z, \xi) u(\xi) \, d\xi,
	\end{equation}
	where \( b \in C_0^\infty \), \( \det \partial^2_{z\xi} \Phi \neq 0 \), and the graph curvature condition \eqref{curcond} holds.  The following Strichartz estimate follows from the Keel--Tao theorem \cite{keeltao} in the truncated-decay setting, together with the standard logarithmic substitute at the one excluded endpoint.
	\begin{proposition}\label{kt}
		Let $d\ge 1$, $2\le a,b\le \infty$, and suppose
		\[
		\frac2a+\frac db\le \frac d2.
		\]
		If $(a,b,d)\ne (2,\infty,2)$, then
		\begin{equation}\label{kt-main-estimate}
			\|T_{\la,\Phi} u\|_{L_s^aL_z^b}\ls \la^{\frac d2-\frac 1a-\frac db}\|u\|_{L^2}.
		\end{equation}
		The only logarithmic endpoint is the two-dimensional endpoint
		$(a,b,d)=(2,\infty,2)$, where
		\begin{equation}\label{kt-2d-log-endpoint}
			\|T_{\la,\Phi} u\|_{L_s^2L_z^\infty}\ls \la^{\frac 12}\sqrt{\log\la}\|u\|_{L^2}.
		\end{equation}
	\end{proposition}
	\begin{proof}
		We prove the estimate on one normal-form patch. The finite sum of patches only changes the
		constant.  Write
		\[
		U_\lambda(s)u(z):=T_{\lambda,\Phi}u(s,z).
		\]
		The nondegeneracy \(\det\partial_{z\xi}^2\Phi\ne 0\) gives the fixed-time \(L^2\) bound uniform in \(s\) and \(\lambda\):
		\[
		\|U_\lambda(s)u\|_{L_z^2}\le C\|u\|_{L_\xi^2}.
		\]

		We also need the usual dispersive estimate. The kernel of \(U_\lambda(s)U_\lambda(s')^*\) is
		\[
		K_\lambda(s,z;s',z')=
		\lambda^d\int e^{i\lambda(\Phi(s,z,\xi)-\Phi(s',z',\xi))}
		b(s,z,\xi)\overline{b(s',z',\xi)}\,d\xi .
		\]
		If \(|s-s'|\lesssim\lambda^{-1}\), the trivial bound gives \(|K_\lambda|\lesssim\lambda^d\).
		If \(|s-s'|\gtrsim\lambda^{-1}\), stationary phase in \(\xi\), using the graph curvature
		condition, gives
		\[
		|K_\lambda(s,z;s',z')|\le C\lambda^d(\lambda |s-s'|)^{-d/2}.
		\]
		Combining the two regions,
		\[
		\|U_\lambda(s)U_\lambda(s')^*F\|_{L_z^\infty}
		\le C\lambda^d(1+\lambda |s-s'|)^{-d/2}\|F\|_{L_z^1}.
		\]
		
		Now rescale both time and space to put this in the normalized Keel--Tao form.  Set
		\[
		\widetilde U_\lambda(t)u(x)
		:=\lambda^{-d/2}U_\lambda(t/\lambda)u(x/\lambda).
		\]
		Then
		\[
		\|\widetilde U_\lambda(t)u\|_{L_x^2}\le C\|u\|_{L^2}
		\]
		and the previous kernel bound becomes
		\[
		\|\widetilde U_\lambda(t)\widetilde U_\lambda(t')^*F\|_{L_x^\infty}
		\le C(1+|t-t'|)^{-d/2}\|F\|_{L_x^1}.
		\]
		Thus the Keel--Tao theorem applies with dispersive exponent \(\sigma=d/2\) in the
		truncated-decay case.  Since Keel--Tao admissibility is
		\[
		\frac1a+\frac{d/2}{b}\le \frac{d/2}{2},
		\]
		it is exactly the condition
		\[
		\frac2a+\frac db\le \frac d2.
		\]
		The only excluded Keel--Tao endpoint \((a,b,\sigma)=(2,\infty,1)\) is precisely
		\((a,b,d)=(2,\infty,2)\).  Therefore, for every other admissible pair,
		\[
		\|\widetilde U_\lambda(t)u\|_{L_t^aL_x^b}\le C\|u\|_{L^2}.
		\]
		Undoing the rescaling gives
		\[
		\|\widetilde U_\lambda u\|_{L_t^aL_x^b}
		=\lambda^{-d/2+d/b+1/a}\|T_{\lambda,\Phi} u\|_{L_s^aL_z^b},
		\]
		with the usual interpretation when \(a=\infty\) or \(b=\infty\).  Hence
		\[
		\|T_{\lambda,\Phi} u\|_{L_s^aL_z^b}
		\le C\lambda^{d/2-1/a-d/b}\|u\|_{L^2},
		\]
		which proves \eqref{kt-main-estimate}.
		
		When \((a,b,d)=(2,\infty,2)\), the same dyadic \(TT^*\) argument is borderline: the time kernel
		has size \((1+\lambda |s-s'|)^{-1}\), whose dyadic summation over \(\lambda^{-1}\lesssim
		|s-s'|\lesssim1\) contributes \(\log\lambda\) at the \(TT^*\) level.  Taking the square root gives
		\[
		\|T_{\lambda,\Phi} u\|_{L_s^2L_z^\infty}\le C\lambda^{1/2}\sqrt{\log\lambda}\|u\|_{L^2}.
		\]
		This is the asserted endpoint estimate.
	\end{proof}
	We shall repeatedly use the following mixed H\"older bookkeeping. If $\sum_{j=1}^k \frac1{a_j}=\sum_{j=1}^k\frac1{b_j}=\frac1p$, then
	\[\Big\|\prod_{j=1}^ku_j\Big\|_{L_{s,z}^p}\ls \prod_{j=1}^k\|u_j\|_{L_s^{a_j}L_z^{b_j}}.\]
	Since each normal-form patch is compact in the \((s,z)\) variables, the same estimate is valid when the two sums are at most \(1/p\), after using the finite-measure embeddings.
	
	\subsection{Bilinear estimates for $d\ge3$} We record the endpoint bilinear estimates for $d\ge3$. The remaining ranges follow by interpolation. We shall use Proposition \ref{kt} with mixed H\"older.
	
	At $p=2$ we have
	\[\|u_1u_2\|_{L^{2}}\ls  \|u_1\|_{L_s^2 L_z^{\infty}}\|u_2\|_{L_s^{\infty} L_z^{2}}\ls  \la_1^{\frac{d-1}2} .\]
	At $p=p_0$, we have 
	\[\|u_1u_2\|_{L^{p_0}}\ls  \|u_1\|_{L_s^2 L_z^{q_1}}\|u_2\|_{L_s^{q_0} L_z^{r_0}}\ls  \la_1^{\frac12}\la_2^{\frac1{q_0}}.\]
	At $p=p_1$, we have 
	\[\|u_1u_2\|_{L^{p_1}}\ls  \|u_1\|_{L_s^2 L_z^{q_1}}\|u_2\|_{L_s^\infty L_z^2}\ls  \la_1^{\frac12}.\]
	At $p=1$, we have
	\[\|u_1u_2\|_{L^1}\ls  \|u_1\|_{L_s^\infty L_z^2}\|u_2\|_{L_s^\infty L_z^2}\ls  1.\]

	\subsection{Multilinear estimates for \texorpdfstring{$d\ge3$}{d >= 3}} Fix $k\ge3$ and $d\ge3$. We prove the endpoint estimates, and the remaining ranges follow by interpolation. We shall use Proposition \ref{kt} with mixed H\"older and the bilinear  estimates obtained above.
	
	At $p=2$, apply H\"older and bilinear $L^2$ estimates
	\[\|u_1u_2\cdots u_k\|_{L^{2}}\ls  \|u_1\|_{ L^{\infty}}\cdots\|u_{k-2}\|_{L^\infty}\|u_{k-1}u_k\|_{L^{2}}\ls  \prod_{j=1}^{k-2}\la_j^{\frac d2}\cdot \la_{k-1}^{\frac{d-1}2}.\]
	At $p=p_0$,  we apply H\"older and bilinear $L^{p_0}$ estimates
	\begin{align*}
		\|u_1u_2\cdots u_k\|_{L^{p_0}}&\ls  \|u_1\|_{ L^{\infty}}\cdots\|u_{k-2}\|_{L^\infty}\|u_{k-1}u_k\|_{L^{p_0}}\\
		&\ls  \prod_{j=1}^{k-2}\la_j^{\frac d2}\cdot \la_{k-1}^{\frac12}\la_k^{\frac{d-2}{2(d+2)}}.
	\end{align*}
	Moreover, mixed H\"older gives
	\begin{align*}
		\|u_1u_2\cdots u_k\|_{L^{p_0}}&\ls  \|u_1\|_{ L^{\infty}}\cdots\|u_{k-3}\|_{L^\infty}\|u_{k-2}\|_{ L_s^{q_0}L_z^{\infty}}\|u_{k-1}\|_{L_s^2 L_z^{q_0}}\|u_k\|_{L_s^{\infty}L_z^2}\\
		&\ls  \prod_{j=1}^{k-3}\la_j^{\frac d2}\cdot  \la_{k-2}^{\frac d2-\frac{d-2}{2(d+2)}}\la_{k-1}^{\frac12+\frac{d-2}{d+2}}.
	\end{align*}
	The first bound is smaller than the second bound if $\la_{k-2}\la_{k}\ll\la_{k-1}^2$.  We therefore take the minimum of these two bounds.
	
	At $p=p_1$,  mixed H\"older gives
	\begin{align*}
		\|u_1u_2\cdots u_k\|_{L^{p_1}}&\ls  \|u_1\|_{ L^{\infty}}\cdots\|u_{k-3}\|_{L^\infty}\|u_{k-2}\|_{ L_s^{q_1}L_z^{\infty}}\|u_{k-1}\|_{L_s^2 L_z^{q_1}}\|u_k\|_{L_s^{\infty}L_z^2}\\
		&\ls  \prod_{j=1}^{k-3}\la_j^{\frac d2}\cdot \la_{k-2}^{\frac{d}2-\frac{d-2}{2d}}\la_{k-1}^{\frac12}.
	\end{align*}
	
	At $p=p_2$,  mixed H\"older gives
	\begin{align*}
		\|u_1u_2\cdots u_k\|_{L^{p_2}}&\ls  \|u_1\|_{ L^{\infty}}\cdots\|u_{k-3}\|_{L^\infty}\|u_{k-2}\|_{ L_s^{2}L_z^{\infty}}\|u_{k-1}\|_{L^{q_2}}\|u_k\|_{L_s^{\infty}L_z^2}\\
		&\ls  \prod_{j=1}^{k-3}\la_j^{\frac d2}\cdot \la_{k-2}^{\frac{d-1}2}\la_{k-1}^{\frac{d}{2(d+2)}}.
	\end{align*}
	At $p=1$,   H\"older and bilinear $L^2$ estimates give a basic bound
	\begin{align}\nonumber
		\|u_1u_2\cdots u_k\|_{L^{1}}&\ls  \|u_1\|_{ L^{\infty}}\cdots\|u_{k-4}\|_{L^\infty}\|u_{k-3}u_{k-1}\|_{L^2}\|u_{k-2}u_k\|_{L^{2}}\\ 
		&\ls  \prod_{j=1}^{k-4}\la_j^{\frac d2}\cdot  \la_{k-3}^{\frac {d-1}2}\la_{k-2}^{\frac{d-1}2}.\label{basic}
	\end{align}
	Moreover, for $d\ge4$, $(2,d)$ is admissible: $\frac2 2+\frac d d=2\le \frac d2$, so  mixed H\"older  gives
	\begin{align*}
		\|u_1u_2\cdots u_k\|_{L^{1}}&\ls  \|u_1\|_{ L^{\infty}}\cdots\|u_{k-3}\|_{L^\infty}\|u_{k-2}\|_{ L_s^{2}L_z^{d}}\|u_{k-1}\|_{L_s^{2} L_z^{q_1}}\|u_k\|_{L_s^{\infty}L_z^2}\\
		&\ls  \prod_{j=1}^{k-3}\la_j^{\frac d2}\cdot \la_{k-2}^{\frac{d-3}2}\la_{k-1}^{\frac12}.
	\end{align*}
	and for $d=3$  mixed H\"older gives
	\begin{align*}
		\|u_1u_2\cdots u_k\|_{L^{1}}&\ls \|u_1\|_{ L^{\infty}}\cdots\|u_{k-4}\|_{L^\infty}\|u_{k-3}\|_{L_s^4L_z^\infty}\|u_{k-2}\|_{L_s^2L_z^6}\|u_{k-1}\|_{L_s^4L_z^3}\|u_k\|_{L_s^\infty L_z^2}\\
		&\ls  \prod_{j=1}^{k-4}\la_j^{\frac 32}\cdot \la_{k-3}^{\frac 54}\la_{k-2}^{\frac12}\la_{k-1}^{\frac14}.
	\end{align*}
	Note that these two bounds are smaller than the basic bound \eqref{basic} if $\la_{k-3}\la_{k-1}\ll\la_{k-2}^2$. Thus we take the minimum of them.

	When $d=3$, there is an extra endpoint $p_3=\frac{10}9$.  Mixed H\"older gives
	\begin{align*}
		\|u_1u_2\cdots u_k\|_{L^{p_3}}&\ls \|u_1\|_{L^\infty}...\|u_{k-3}\|_{L^\infty}\|u_{k-2}\|_{ L_s^{2}L_z^{6}}\|u_{k-1}\|_{L_s^{5/2} L_z^{30/7}}\|u_k\|_{L_s^{\infty}L_z^2}\\
		&\ls \prod_{j=1}^{k-3}\la_j^{\frac 32}\cdot \la_{k-2}^{\frac12}\la_{k-1}^{\frac 25},
	\end{align*}
	and
	\begin{align*}
		\|u_1u_2\cdots u_k\|_{L^{p_3}}&\ls \|u_1\|_{L^\infty}...\|u_{k-4}\|_{L^\infty}\|u_{k-3}\|_{L_s^4L_z^\infty}\|u_{k-2}\|_{L_s^2L_z^\infty}\|u_{k-1}\|_{L_s^{20/3}L_z^{5/2}}\|u_{k}\|_{L_s^\infty L_z^2}\\
		&\ls \prod_{j=1}^{k-4}\la_j^{\frac 32}\cdot  \la_{k-3}^{\frac 54}\la_{k-2}\la_{k-1}^{\frac 3{20}}.
	\end{align*}
	The first bound is smaller than the second bound if $\la_{k-3}\la_{k-1}\ll \la_{k-2}^2$. Thus we take the minimum of them.
	\subsection{Multilinear estimates for \texorpdfstring{$d=1$}{d = 1}} Fix $k\ge2$ and $d=1$.  We shall prove multilinear estimates directly by applying Proposition \ref{kt} with mixed H\"older.
	
	At $p=2$, mixed H\"older gives
	\begin{align*}
		\|u_1u_2\cdots u_k\|_{L^2}
		&\ls \prod_{j=1}^{k-3}\|u_j\|_{L^\infty}
		\|u_{k-2}\|_{L_s^4L_z^\infty}
		\|u_{k-1}\|_{L_s^4L_z^\infty}
		\|u_k\|_{L_s^\infty L_z^2}\\
		&\ls \prod_{j=1}^{k-3}\la_j^{1/2}\cdot \la_{k-2}^{1/4}\la_{k-1}^{1/4}.
	\end{align*}
	At $p=\frac32$, we have
	\begin{align*}
		\|u_1u_2\cdots u_k\|_{L^{3/2}}
		&\ls \prod_{j=1}^{k-4}\|u_j\|_{L^\infty}
		\|u_{k-3}\|_{L_s^4L_z^\infty}
		\|u_{k-2}\|_{L_s^4L_z^\infty}
		\|u_{k-1}\|_{L_s^6L_z^6}
		\|u_k\|_{L_s^\infty L_z^2}\\
		&\ls \prod_{j=1}^{k-4}\la_j^{1/2}\cdot
		\la_{k-3}^{1/4}\la_{k-2}^{1/4}\la_{k-1}^{1/6}.
	\end{align*}
	At $p=\frac65$, we have
	\begin{align*}
		\|u_1u_2\cdots u_k\|_{L^{6/5}}
		&\ls \prod_{j=1}^{k-5}\|u_j\|_{L^\infty}
		\|u_{k-4}\|_{L_s^4L_z^\infty}
		\|u_{k-3}\|_{L_s^4L_z^\infty}
		\|u_{k-2}\|_{L_s^4L_z^\infty}
		\|u_{k-1}\|_{L_s^{12}L_z^3}
		\|u_k\|_{L_s^\infty L_z^2}\\
		&\ls \prod_{j=1}^{k-5}\la_j^{1/2}\cdot
		\la_{k-4}^{1/4}\la_{k-3}^{1/4}\la_{k-2}^{1/4}\la_{k-1}^{1/12}.
	\end{align*}
	At $p=1$, we have
	\begin{align*}
		\|u_1u_2\cdots u_k\|_{L^1}
		&\ls \prod_{j=1}^{k-6}\|u_j\|_{L^\infty }
		\prod_{m=k-5}^{k-2}\|u_m\|_{L_s^4L_z^\infty}
		\|u_{k-1}\|_{L_s^\infty L_z^2}
		\|u_k\|_{L_s^\infty L_z^2}\\
		&\ls \prod_{j=1}^{k-6}\la_j^{1/2}\cdot
		\la_{k-5}^{1/4}\la_{k-4}^{1/4}\la_{k-3}^{1/4}\la_{k-2}^{1/4}.
	\end{align*}

	In particular, the exponents $p=\frac 32, \frac65$ are not regarded as endpoints when \(k=2,3\), since they follow by interpolating between the \(p=1\) and \(p=2\) estimates proved above.
	\subsection{Multilinear estimates for $d=2$}\label{subsec:d2-kt-upper}
	In the general case the endpoint estimate \eqref{kt-2d-log-endpoint} has a log factor, so we prove the stated bounds directly by mixed H\"older. 
	
	Let $1\le p\le 2$ and $p'=\frac{p}{p-1}$ and $q=\frac{2p}{2-p}$. For $k=2$, 
	\[\|u_1u_2\|_{L^p}\ls \|u_1\|_{L_s^{p'}L_z^{q}}\|u_2\|_{L_s^\infty L_z^2}\ls \begin{cases}
		\la_1^{1-\frac1p}, &1\le p<2\\
		\la_1^{\frac12}\sqrt{\log\la_1}, &p=2.
	\end{cases}\]

	Now assume \(k\ge3\).   For \(1\le p\le \frac 43\) and $r=\frac{2p}{4-3p}$, mixed H\"older gives
	\begin{align*}
		\|u_1u_2...u_k\|_{L^p}\ls \prod_{j=1}^{k-4}\|u_j\|_{L^\infty}\|u_{k-3}\|_{L_s^rL_z^\infty}\|u_{k-2}\|_{L_s^2L_z^\infty}\|u_{k-1}\|_{L_s^{p'}L_z^{q}}\|u_k\|_{L_s^\infty L_z^2}.
	\end{align*}
	Hence, for \(1<p\le \frac 43\),
	\[
	\|u_1u_2...u_k\|_{L^p}
	\lesssim
	\prod_{j=1}^{k-4}\lambda_j\cdot 
	\lambda_{k-3}^{\frac52-\frac2p}
	\lambda_{k-2}^{\frac12}\sqrt{\log\lambda_{k-2}}
	\lambda_{k-1}^{1-\frac1p}.
	\]
	At \(p=1\), we have $r=2$, so one obtains the same bound with the extra factor
	\(\sqrt{\log\lambda_{k-3}}\).  
	
	For \(\frac 43<p\le 2\), mixed H\"older gives
	\begin{align*}
		\|u_1u_2...u_k\|_{L^p}\ls \prod_{j=1}^{k-3}\|u_j\|_{L^\infty}\|u_{k-2}\|_{L_s^{q/2}L_z^\infty}\|u_{k-1}\|_{L_s^{p'}L_z^{q}}\|u_k\|_{L_s^\infty L_z^2}.
	\end{align*}
	Thus, for $\frac 43<p<2$, 
	\[
	\|u_1u_2...u_k\|_{L^p}\ls
	\prod_{j=1}^{k-3}\lambda_j\cdot 
	\lambda_{k-2}^{2-\frac 2p}
	\lambda_{k-1}^{1-\frac 1p}.
	\]
	At \(p=2\), we have $(p',q)=(2,\infty)$, so one obtains the same bound with the extra factor $\sqrt{\log\lambda_{k-1}}$.

	\subsection{Multilinear estimates for $p>2$}\label{subsec:large-p-upper}
	Let $k\ge2$ and $p>2$.   We shall prove the estimates in Theorem \ref{thm3} by applying Proposition \ref{kt} and mixed H\"older.
	
	Let $p_c=\frac{2(d+2)}d$. For $p\ge p_c$, H\"older gives
	\begin{align*}
		\|u_1u_2...u_k\|_{L^p}&\ls \prod_{j=1}^{k-1}\|u_j\|_{L^\infty}\|u_k\|_{L^p}\ls \prod_{j=1}^{k-1}\la_j^{\frac d2}\cdot 
		\lambda_k^{\frac d2-\frac {d+1}{p}}.
	\end{align*}
	For $2< p< p_c$ when $d\ge2$, and $3\le p<6$ when $d=1$, let $a=\frac{4p}{4-d(p-2)}$ and $b=\frac{4p}{d(p-2)}$. We have $$\frac1a+\frac1b=\frac1p.$$
	Then mixed H\"older gives 
	\begin{align}\nonumber
		\|u_1u_2...u_k\|_{L^p}&\ls \prod_{j=1}^{k-2}\|u_j\|_{L^\infty}\|u_{k-1}\|_{L_s^aL_z^\infty}\|u_k\|_{L_s^bL_z^p}\\ 
		&\ls \prod_{j=1}^{k-2}\la_j^{\frac d2}\cdot \lambda_{k-1}^{\frac{3d}{4}-\frac{d+2}{2p}}
		\lambda_k^{\frac d4-\frac d{2p}}.\label{gb}
	\end{align}
	For  $2< p<3$ when $d=1$, let $c=\frac{2p}{3-p}$. We have
	\[\frac1c+\frac14+\frac1b=\frac1p.\]
	Then mixed H\"older gives
	\begin{align*}
		\|u_1u_2...u_k\|_{L^p}
		&\lesssim
		\prod_{j=1}^{k-3}\|u_j\|_{L^\infty}\|u_{k-2}\|_{L_s^{c}L_z^\infty}
		\|u_{k-1}\|_{L_s^4L_z^\infty}
		\|u_k\|_{L_s^{b}L_z^p}  \\
		&\lesssim
		\prod_{j=1}^{k-3}\la_j^{\frac d2}\cdot \lambda_{k-2}^{1-\frac{3}{2p}}
		\lambda_{k-1}^{\frac14}
		\lambda_k^{\frac14-\frac1{2p}}.
	\end{align*}
	This bound is smaller than \eqref{gb} when $\la_{k-1}\ll\la_k$. This completes the proof.

	\section{Model examples on the sphere}\label{sec11}
	In this section we construct the model profiles which will be used in the proof of the sharpness results in Section~\ref{sec12}. We work on the round sphere \(S^n\) and write \(d=n-1\) for the dimension transverse to a geodesic.  The degree-\(\nu\) spherical harmonics are
	\[
	\HH_\nu(S^n)=\{Y\in C^\infty(S^n):-\Delta_{S^n}Y=\nu(\nu+n-1)Y\}.
	\]
	All lower examples below are exact spherical harmonics.  Replacing \(\nu\) by the true spectral parameter \(\sqrt{\nu(\nu+n-1)}\) changes only constants and hence does not change any power of the frequency.
	
	\subsection{Basic beams and one-packet models}
	Write
	\[
	x=(x_1,x_2,x'')\in\mathbb R\times\mathbb R\times\mathbb R^d,
	\qquad x_0=(1,0,0).
	\]
	For \(|\eta|<1\), set
	\[
	\omega(\eta)=(\sqrt{1-|\eta|^2},\eta)\in S^d,
	\qquad a_\eta=(1,i\omega(\eta))\in\mathbb C^{n+1},
	\]
	and
	\[
	q_{\nu,\eta}(x)=\kappa_{\nu,n}(a_\eta\cdot x)^\nu,
	\qquad
	\kappa_{\nu,n}=
	\left(
	\frac{|S^n|\Gamma((n+1)/2)\nu!}{\Gamma(\nu+(n+1)/2)}
	\right)^{-1/2}.
	\]
	
	\begin{lemma}[Gaussian beam]\label{lem:exact-beams}
		For every \(\nu\ge1\) and every \(|\eta|,|\zeta|<1\),
		\[
		q_{\nu,\eta}\in\HH_\nu(S^n),\qquad
		\|q_{\nu,\eta}\|_2=1,
		\qquad
		\kappa_{\nu,n}\approx \nu^{d/4},
		\]
		and
		\[
		\langle q_{\nu,\eta},q_{\nu,\zeta}\rangle_{L^2(S^n)}
		=\left(\frac{1+\omega(\eta)\cdot\omega(\zeta)}2\right)^\nu .
		\]
		Consequently, after restricting \(|\eta|,|\zeta|\le\delta_0\),
		\[
		\left|\langle q_{\nu,\eta},q_{\nu,\zeta}\rangle\right|
		\lesssim e^{-c\nu|\eta-\zeta|^2},
		\qquad
		\left|\langle q_{\nu,\eta},q_{\nu,\zeta}\rangle\right|
		\gtrsim e^{-C\nu|\eta-\zeta|^2}.
		\]
	\end{lemma}
	
	\begin{proof}
		We first check that the beams are spherical harmonics.  Since
		\[
		a_\eta\cdot a_\eta
		=1+(i\sqrt{1-|\eta|^2})^2+\sum_{j=1}^d(i\eta_j)^2=0,
		\]
		one has
		\[
		\Delta_{\mathbb R^{n+1}}(a_\eta\cdot x)^\nu
		=\nu(\nu-1)(a_\eta\cdot a_\eta)(a_\eta\cdot x)^{\nu-2}=0.
		\]
		Thus \((a_\eta\cdot x)^\nu\) is a homogeneous harmonic polynomial of degree
		\(\nu\).  Its restriction to \(S^n\) is therefore an eigenfunction of
		\(-\Delta_{S^n}\) with eigenvalue \(\nu(\nu+n-1)\), hence belongs to
		\(\HH_\nu(S^n)\).
		
		It remains to compute the normalization and the overlaps.  For
		\(c\in\mathbb C^{n+1}\) we use the standard entire expansion
		\[
		\int_{S^n}e^{c\cdot x}\,dS(x)
		=|S^n|\Gamma\!\left(\frac{n+1}{2}\right)
		\sum_{m\ge0}\frac{(c\cdot c/4)^m}{m!\Gamma(m+(n+1)/2)} .
		\]
		Indeed, this is the usual Bessel-function formula for real \(c\), continued
		analytically to complex \(c\).  Apply it with
		\(c=s a_\eta+t\overline{a_\zeta}\).  Since
		\(a_\eta\cdot a_\eta=\overline{a_\zeta}\cdot\overline{a_\zeta}=0\),
		\[
		c\cdot c=2st\,a_\eta\cdot\overline{a_\zeta}
		=2st\bigl(1+\omega(\eta)\cdot\omega(\zeta)\bigr).
		\]
		On the other hand,
		\[
		\int_{S^n}e^{s a_\eta\cdot x+t\overline{a_\zeta}\cdot x}\,dS(x)
		=\sum_{j,k\ge0}\frac{s^j t^k}{j!k!}
		\int_{S^n}(a_\eta\cdot x)^j
		(\overline{a_\zeta}\cdot x)^k\,dS(x).
		\]
		Comparing the coefficient of \(s^\nu t^\nu\) gives
		\[
		\int_{S^n}(a_\eta\cdot x)^\nu(\overline{a_\zeta}\cdot x)^\nu\,dS(x)
		=|S^n|\Gamma\!\left(\frac{n+1}{2}\right)
		\frac{\nu!}{\Gamma(\nu+(n+1)/2)}
		\left(\frac{1+\omega(\eta)\cdot\omega(\zeta)}2\right)^\nu.
		\]
		Taking \(\zeta=\eta\) shows that the constant \(\kappa_{\nu,n}\) makes
		\(\|q_{\nu,\eta}\|_2=1\), and the same identity gives the stated overlap
		formula.  Stirling's formula yields
		\[
		\frac{\Gamma(\nu+(n+1)/2)}{\nu!}\approx \nu^{(n-1)/2}=\nu^{d/2},
		\]
		so \(\kappa_{\nu,n}\approx \nu^{d/4}\).
		
		Finally,
		\[
		\frac{1+\omega(\eta)\cdot\omega(\zeta)}2
		=1-\frac{|\omega(\eta)-\omega(\zeta)|^2}{4}.
		\]
		After restricting to a sufficiently small ball \(|\eta|,|\zeta|\le\delta_0\),
		the map \(\eta\mapsto\omega(\eta)\) is bi-Lipschitz and
		\(|\omega(\eta)-\omega(\zeta)|^2/4\le1/2\).  Therefore
		\[
		1-C|\eta-\zeta|^2
		\le \frac{1+\omega(\eta)\cdot\omega(\zeta)}2
		\le 1-c|\eta-\zeta|^2 .
		\]
		Raising to the \(\nu\)-th power and using
		\((1-u)^\nu\le e^{-\nu u}\) and \((1-u)^\nu\ge e^{-2\nu u}\) for
		\(0\le u\le1/2\) gives the two Gaussian overlap bounds.
	\end{proof}

	\begin{lemma}[Envelope and zonal]\label{lem:coherent-short-packets}
		Fix \(K\ge1\).  After choosing \(r_*=r_*(K)>0\) sufficiently small, the
		following holds for every \(\nu^{-1}\le r\le r_*\).  There is a normalized
		\(u_{\nu,r}\in\HH_\nu(S^n)\) such that, in the Fermi chart
		\[
		x(t,z)=\left(\cos|z|\cos t,\ \cos|z|\sin t,\ \frac{\sin|z|}{|z|}z\right),
		\qquad z\in\mathbb R^d,
		\]
		we have the lower bound
		\begin{equation}\label{eq:short-packet-lower}
			|u_{\nu,r}(x(t,z))|
			\gtrsim \nu^{d/4}r^{-d/4}
		\end{equation}
		whenever \(|t|\le cr\) and \(|z|\le c(r/\nu)^{1/2}\), and the tail bound
		\begin{equation}\label{eq:short-packet-upper}
			|u_{\nu,r}(x(t,z))|
			\le C_K\nu^{d/4}r^{-d/4}
			\left(1+\frac{|t|}{r}\right)^{-d/2}
		\end{equation}
		whenever \(|t|\le r_*\) and \(|z|\le K(r/\nu)^{1/2}\).
		
		In particular, the case \(r=\nu^{-1}\), after rotation, gives a normalized
		zonal function \(p_{\nu,y}\in\HH_\nu(S^n)\) satisfying
		\[
		|p_{\nu,y}(x)|\gtrsim \nu^{d/2}
		\qquad\text{if } \operatorname{dist}_{S^n}(x,y)\le c\nu^{-1}.
		\]
		Consequently, if \(E\subset B(y,c\mu^{-1})\) and \(\nu\le\mu\), then
		\(|p_{\nu,y}|\gtrsim \nu^{d/2}\) on \(E\).
	\end{lemma}
	
	\begin{proof}
		Let \(\delta=(\nu r)^{-1/2}\), and choose a fixed nonnegative
		\(\psi\in C_0^\infty(\{|\theta|\le 2\varepsilon\})\), equal to one for
		\(|\theta|\le\varepsilon\).  Define
		\[
		U_{\nu,r}(x)=\int \psi(\eta/\delta)q_{\nu,\eta}(x)\,d\eta,
		\qquad u_{\nu,r}=U_{\nu,r}/\|U_{\nu,r}\|_2 .
		\]
		By the Gaussian overlap estimate in Lemma~\ref{lem:exact-beams},
		\[
		\|U_{\nu,r}\|_2^2\approx \delta^d\nu^{-d/2},
		\qquad
		\|U_{\nu,r}\|_2^{-1}\approx \nu^{d/4}\delta^{-d/2}.
		\]
		On the box \(|t|\le cr\), \(|z|\le c(r/\nu)^{1/2}\), the factors
		\((a_\eta\cdot x(t,z))^\nu\), \(|\eta|\le2\varepsilon\delta\), have modulus
		\(\gtrsim1\) and phases differing by at most a sufficiently small absolute constant, provided \(c\) and
		\(\varepsilon\) are small.  Hence
		\[
		|U_{\nu,r}(x(t,z))|\gtrsim \nu^{d/4}\delta^d,
		\]
		and normalization gives \eqref{eq:short-packet-lower}.
		
		For the upper bound, put \(\eta=\delta\theta\).  Then
		\[
		u_{\nu,r}(x(t,z))
		=B_{\nu,r}\int \psi(\theta)\exp\{\nu\log w_{\delta\theta}(t,z)\}\,d\theta,
		\qquad B_{\nu,r}\approx \nu^{d/4}r^{-d/4},
		\]
		where
		\[
		w_\eta(t,z)=a_\eta\cdot x(t,z)
		=\cos\rho\cos t
		+i\bigl(\cos\rho\sin t\,(1-|\eta|^2)^{1/2}+\beta_\rho z\cdot\eta\bigr),
		\quad \beta_\rho=\frac{\sin\rho}{\rho},\quad \rho=|z|.
		\]
		On the region \(|z|\le K(r/\nu)^{1/2}\), the logarithm is taken on a fixed
		branch and \(|w_{\delta\theta}|\approx1\).  Write
		\[
		\Phi(\theta)=\operatorname{Im}(\nu\log w_{\delta\theta}(t,z)),
		\qquad
		A(\theta)=\psi(\theta)\exp\{\operatorname{Re}(\nu\log w_{\delta\theta}(t,z))\},
		\qquad
		\Lambda=1+\frac{|t|}{r}.
		\]
		A Taylor expansion of the explicit formula for \(w_\eta\), using
		\(\nu\delta^2=1/r\), gives uniformly on the support of \(\psi\)
		\[
		\partial_{\theta_i}\partial_{\theta_j}\Phi(\theta)
		=-\frac{\sin t\cos t}{r}\delta_{ij}
		+O_K\!\left(1+\varepsilon\frac{|t|}{r}\right),
		\]
		and the scaled symbol bounds
		\[
		|\partial_\theta^\alpha \Phi(\theta)|\le C_{\alpha,K}\Lambda\quad(|\alpha|\ge2),
		\qquad
		|\partial_\theta^\alpha A(\theta)|\le C_{\alpha,K}\Lambda^{|\alpha|/2}.
		\]
		If \(|t|\lesssim r\), the integral is \(O(1)\), which is the desired bound.  If
		\(|t|\ge C_0r\), then \(|\sin t\cos t|\approx |t|\) on the fixed small coordinate
		interval.  Taking \(C_0\) large and \(\varepsilon\) small, the phase has Hessian
		of size \(\Lambda\).  The stationary phase estimate with the above scaled
		symbol bounds gives
		\[
		\left|\int A(\theta)e^{i\Phi(\theta)}\,d\theta\right|
		\le C_K\Lambda^{-d/2}.
		\]
		Multiplying by \(B_{\nu,r}\) proves \eqref{eq:short-packet-upper}.
		
		Finally, when \(r=\nu^{-1}\), the lower-bound box contains a ball of radius
		\(c\nu^{-1}\) after reducing \(c\).  Rotating the construction gives the stated zonal function.
	\end{proof}

	\subsection{The packet profile}
	For \(0<L\lesssim1\), \(R\ge1\), write \(E_{L,R}\) for a rectangular packet box in one of the local charts used above,
	\[
	E_{L,R}=\{|s|\lesssim L,\ |y|\lesssim R^{-1/2}\},
	\qquad |E_{L,R}|\approx LR^{-d/2}.
	\]
	The center of the box may be translated along the reference geodesic.  The following lemmas give the individual constructions that make up the profile.
	
	\begin{lemma}[Beam block: \(1\le R\le\nu\)]\label{lem:wide-transverse-block}
		Assume \(1\le R\le\nu\).  One can find a normalized \(u\in\HH_\nu(S^n)\) and a set \(E\) contained in a fixed coordinate patch, with
		\[
		|E|\approx LR^{-d/2},
		\qquad
		|u(x)|\gtrsim R^{d/4}\quad (x\in E).
		\]
	\end{lemma}
	
	\begin{proof}
		Use the chart
		\[
		X(s,y)=(\cos s,\sin s\,\omega(y)),
		\qquad  s\in[1/4,1/2],\quad |y|\ll1,
		\]
		where \(\omega(y)=(\sqrt{1-|y|^2},y)\in S^d\).  In this chart the beams from Lemma~\ref{lem:exact-beams} satisfy, after shrinking the chart,
		\begin{equation}\label{eq:beam-belt-decay}
			|q_{\nu,\eta}(X(s,y))|
			\le C\nu^{d/4}e^{-c\nu|y-\eta|^2},
			\qquad
			|q_{\nu,\eta}(X(s,y))|
			\ge c\nu^{d/4}
		\end{equation}
		whenever \(|y-\eta|\le c\nu^{-1/2}\).  This follows directly from
		\[
		q_{\nu,\eta}(X(s,y))
		=\kappa_{\nu,n}\bigl(\cos s+i\sin s\,\omega(\eta)\cdot\omega(y)\bigr)^\nu
		\]
		and the bi-Lipschitz property of \(\omega\) near the origin.
		
		Choose a subinterval \(I_L\subset[1/4,1/2]\) with \(|I_L|\approx L\).  Let \(D\gg1\) be fixed, and choose a maximal \(D\nu^{-1/2}\)-separated set
		\[
		\mathcal A\subset\{\eta: |\eta|\le c_0R^{-1/2}\}.
		\]
		Then
		\[
		N:=|\mathcal A|\approx (\nu/R)^{d/2}.
		\]
		For \(\eta\in\mathcal A\), set
		\[
		T_\eta=\{X(s,y):s\in I_L,\ |y-\eta|\le c\nu^{-1/2}\},
		\qquad E=\bigcup_{\eta\in\mathcal A}T_\eta .
		\]
		The tubes \(T_\eta\) are disjoint if \(c\ll D\), and therefore
		\[
		|E|\approx N L\nu^{-d/2}\approx LR^{-d/2}.
		\]
		\begin{figure}[htbp]
	\centering
	\includegraphics[width=0.8\textwidth]{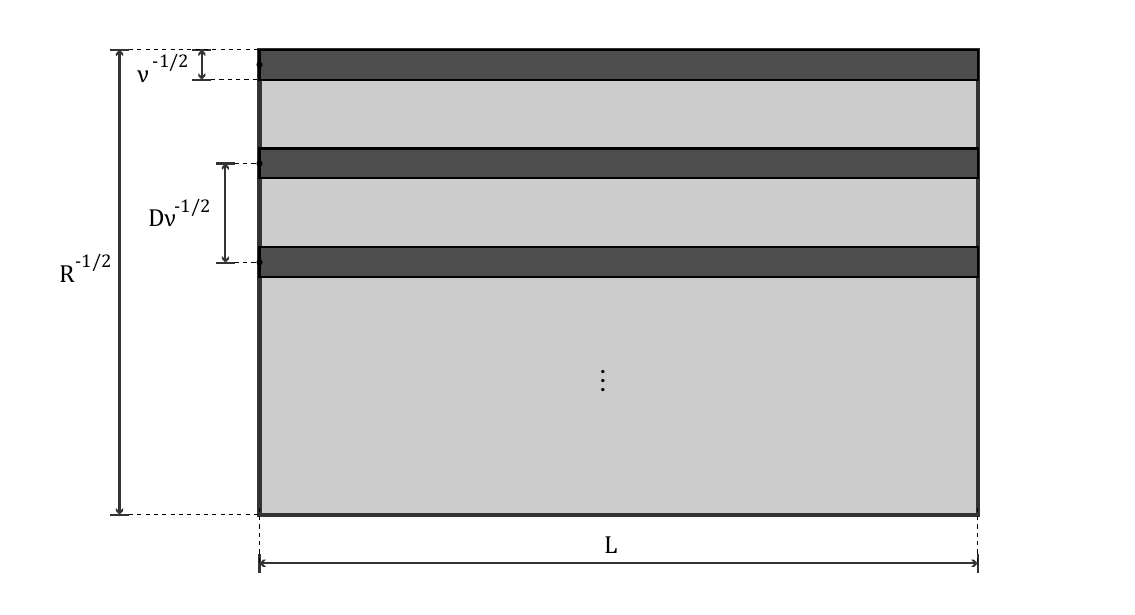}
	\caption{Beam block: $1\leq R \leq \nu$.}
	\label{figbb}
	\end{figure}

		Set
		\[
		Q=\sum_{\eta\in\mathcal A}q_{\nu,\eta}.
		\]
		By the Gaussian overlap bound in Lemma~\ref{lem:exact-beams},
		\[
		\|Q\|_2^2
		\le \sum_{\eta,\zeta\in\mathcal A}e^{-c\nu|\eta-\zeta|^2}
		\le C_D N.
		\]
		If \(x\in T_\eta\), then the \(\eta\)-beam is \(\gtrsim\nu^{d/4}\) by \eqref{eq:beam-belt-decay}, while the remaining beams contribute at most
		\[
		C\nu^{d/4}\sum_{\zeta\in\mathcal A\setminus\{\eta\}}
		e^{-c\nu|\eta-\zeta|^2}.
		\]
		Taking \(D\) sufficiently large makes this tail smaller than half of the main term.  Hence \(|Q|\gtrsim\nu^{d/4}\) on \(E\).  With \(u=Q/\|Q\|_2\),
		\[
		|u(x)|\gtrsim \nu^{d/4}N^{-1/2}\approx R^{d/4},
		\qquad x\in E.
		\]
	\end{proof}
	
	\begin{lemma}[Train construction]\label{lem:separated-short-packet-trains}
		Assume \(d\ge3\), \(\nu^{-1}\le r\le r_*\), and let \(P_a\) denote a rotation of the packet in Lemma~\ref{lem:coherent-short-packets} whose core is centered at the time \(a\) on the reference geodesic.  If \(a_j\) are contained in a fixed coordinate interval and
		\[
		|a_j-a_\ell|\ge D r\qquad (j\ne \ell),
		\]
		then, for \(D\) sufficiently large but fixed,
		\begin{equation}\label{eq:short-train-norm}
			\Big\|\sum_j P_{a_j}\Big\|_2\lesssim N^{1/2},
			\qquad N=\#\{j\},
		\end{equation}
		and on the union of the cores
		\[
		F_j=\{x(t,z): |t-a_j|\le cr,\ |z|\le c(r/\nu)^{1/2}\}
		\]
		one has
		\begin{equation}\label{eq:short-train-pointwise}
			\Big|\sum_j P_{a_j}(x)\Big|
			\gtrsim \nu^{d/4}r^{-d/4}.
		\end{equation}
	\end{lemma}

	\begin{proof}
		Let \(R_a\) be the rotation along the reference geodesic, so that
		\(R_a x(t,z)=x(t+a,z)\), and write \(P_a=R_a u_{\nu,r}\).  The lower bound on a
		single core and the tail estimate away from that core are exactly
		\eqref{eq:short-packet-lower} and \eqref{eq:short-packet-upper}, after replacing
		\(t\) by \(t-a\).  It remains to prove the almost-orthogonality estimate.
		
		By rotation invariance it suffices to estimate \(\langle P_0,P_\tau\rangle\),
		where \(\tau=b-a\).  Put \(\delta=(\nu r)^{-1/2}\) and
		\(\alpha_\eta=(1-|\eta|^2)^{1/2}\).  Lemma~\ref{lem:exact-beams}, applied after
		rotating one beam by \(\tau\), gives
		\[
		\langle q_{\nu,\eta},R_\tau q_{\nu,\zeta}\rangle
		=\Gamma_\tau(\eta,\zeta)^\nu,
		\]
		where
		\[
		\Gamma_\tau(\eta,\zeta)
		=\frac12\Bigl(
		\cos\tau\,(1+\alpha_\eta\alpha_\zeta)+\eta\cdot\zeta
		+i\sin\tau\,(\alpha_\eta+\alpha_\zeta)
		\Bigr).
		\]
		Consequently,
		\[
		\langle P_0,P_\tau\rangle
		=c_{\nu,r}\iint
		\psi(\theta)\overline{\psi(\vartheta)}
		\Gamma_\tau(\delta\theta,\delta\vartheta)^\nu
		\,d\theta\,d\vartheta,
		\qquad c_{\nu,r}\approx r^{-d/2}.
		\]
		For \(|\tau|\lesssim r\), Cauchy's inequality gives
		\(|\langle P_0,P_\tau\rangle|\le1\).  Assume \(|\tau|\ge C_0r\).  Taylor expansion of
		the explicit kernel gives
		\[
		\operatorname{Re}\bigl(\nu\log\Gamma_\tau(\delta\theta,\delta\vartheta)\bigr)
		\le -c\frac{|\theta-\vartheta|^2}{r}+C.
		\]
		With
		\[
		u=\frac{\theta+\vartheta}{2},
		\qquad v=\frac{\theta-\vartheta}{\sqrt r},
		\]
		the Jacobian is \(d\theta\,d\vartheta=r^{d/2}\,du\,dv\), which cancels
		\(c_{\nu,r}\).  The same expansion gives, in the \(u\)-variables,
		\[
		\partial_{u_i}\partial_{u_j}
		\operatorname{Im}\bigl(\nu\log\Gamma_\tau(\delta\theta,\delta\vartheta)\bigr)
		=-\frac{\sin\tau}{r}\delta_{ij}
		+O\!\left(1+\varepsilon\frac{|\tau|}{r}\right).
		\]
		On the fixed coordinate interval, \(|\sin\tau|\approx |\tau|\).  Therefore the
		inner \(u\)-integral has a nondegenerate Hessian of size \(|\tau|/r\), while the
		\(v\)-dependence is Gaussian.  The same scaled stationary-phase estimate used in
		Lemma~\ref{lem:coherent-short-packets} yields
		\begin{equation}\label{eq:short-packet-overlap}
			|\langle P_0,P_\tau\rangle|
			\le C\left(1+\frac{|\tau|}{r}\right)^{-d/2}.
		\end{equation}
		Rotating gives the same bound for \(P_a\) and \(P_b\).
		
		Now
		\[
		\Big\|\sum_j P_{a_j}\Big\|_2^2
		\le \sum_{j,\ell}|\langle P_{a_j},P_{a_\ell}\rangle|
		\le C\sum_j\sum_{m\in\mathbb Z}(1+D|m|)^{-d/2}
		\lesssim N,
		\]
		because the points \(a_j\) are \(Dr\)-separated on a line and \(d/2>1\).  This
		proves \eqref{eq:short-train-norm}.
		
		Finally, fix \(x\in F_j\).  The main packet satisfies
		\(|P_{a_j}(x)|\gtrsim\nu^{d/4}r^{-d/4}\) by
		\eqref{eq:short-packet-lower}.  For \(\ell\ne j\), the upper bound
		\eqref{eq:short-packet-upper} gives
		\[
		|P_{a_\ell}(x)|
		\le C\nu^{d/4}r^{-d/4}
		\left(1+\frac{|a_\ell-a_j|}{r}\right)^{-d/2},
		\]
		after decreasing the core constant \(c\).  Hence
		\[
		\sum_{\ell\ne j}|P_{a_\ell}(x)|
		\le C\nu^{d/4}r^{-d/4}
		\sum_{m\ge1}(1+Dm)^{-d/2}.
		\]
		Choosing the fixed separation constant \(D\) large makes this tail smaller than
		the main term.  This proves \eqref{eq:short-train-pointwise}.
	\end{proof}

	\begin{lemma}[Envelope train: \(\nu\le R\le\nu^2\)]\label{lem:intermediate-packet-profile}
		Assume \(\nu\le R\le\nu^2\) and put \(r=\nu/R\).  If either \(L\le r\), or \(d\ge3\), one can find a normalized \(u\in\HH_\nu(S^n)\) and a set \(E\) with
		\[
		|E|\approx LR^{-d/2}
		\]
		such that
		\begin{equation}\label{eq:intermediate-packet-profile}
			|u(x)|\gtrsim
			R^{d/4}\min\{1,(L/r)^{-1/2}\}
			=R^{d/4}\min\{1,(LR/\nu)^{-1/2}\},
			\qquad x\in E.
		\end{equation}
	\end{lemma}
	
	\begin{proof}
		Since \(\nu\le R\le\nu^2\), the number \(r=\nu/R\) satisfies \(\nu^{-1}\le r\le1\), and the transverse radius in Lemma~\ref{lem:coherent-short-packets} is
		\[
		(r/\nu)^{1/2}=R^{-1/2}.
		\]
		If \(L\le r\), take one envelope \(u_{\nu,r}\) from Lemma~\ref{lem:coherent-short-packets} and restrict its core to \(|t|\lesssim L\).  On this set
		\[
		|u_{\nu,r}|
		\gtrsim \nu^{d/4}r^{-d/4}=R^{d/4},
		\]
		and the measure is \(\approx LR^{-d/2}\).  This is the first alternative in \eqref{eq:intermediate-packet-profile}.
		
		Assume now that \(L>r\) and \(d\ge3\).  Choose times \(a_j\) in an interval of length \(\approx L\), separated by \(Dr\), with
		\[
		N\approx L/r=LR/\nu.
		\]

	\begin{figure}[htbp]
	\centering
	\includegraphics[width=0.8\textwidth]{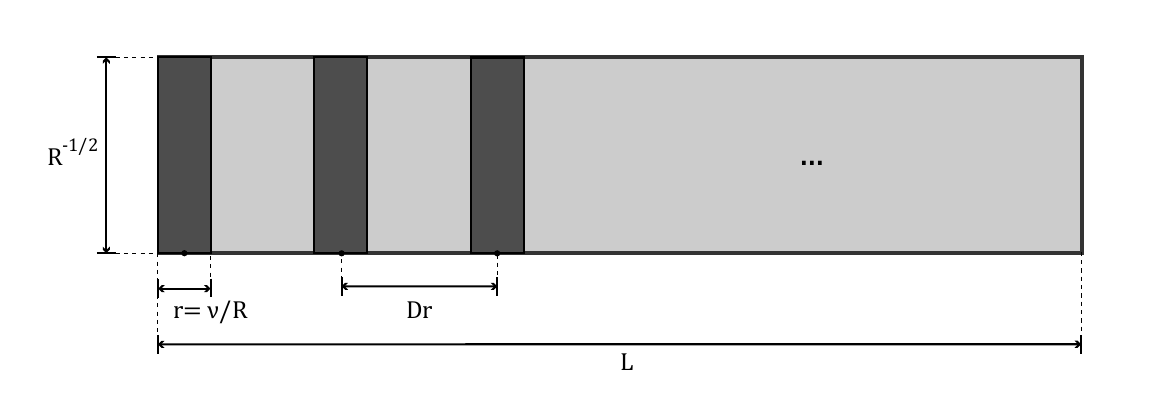}
	\caption{Envelope train: $\nu \leq R \leq \nu^2$.}
	\label{figet}
	\end{figure}

		Let \(P_{a_j}\) be the corresponding rotated envelopes, set
		\[
		Q=\sum_{j=1}^N P_{a_j},
		\qquad u=Q/\|Q\|_2,
		\]
		and let
		\[
		E=\bigcup_{j=1}^N
		\{x(t,z): |t-a_j|\le cr,
		\ |z|\le cR^{-1/2}\}.
		\]
		The pieces are disjoint, so
		\[
		|E|\approx N r R^{-d/2}\approx LR^{-d/2}.
		\]
		By Lemma~\ref{lem:separated-short-packet-trains}, \(\|Q\|_2\lesssim N^{1/2}\) and \(|Q|\gtrsim R^{d/4}\) on \(E\).  Therefore
		\[
		|u(x)|\gtrsim R^{d/4}N^{-1/2}
		\approx R^{d/4}(L/r)^{-1/2},
		\qquad x\in E.
		\]
	\end{proof}

	\begin{lemma}[Zonal train: \(R\ge\nu^2\)]\label{lem:zonal-train-profile}
		Assume \(R\ge\nu^2\).  If either \(L\nu\lesssim1\), or \(d\ge3\), one can find a normalized \(u\in\HH_\nu(S^n)\) and a set \(E\) with
		\[
		|E|\approx LR^{-d/2}
		\]
		such that
		\begin{equation}\label{eq:zonal-train-profile}
			|u(x)|\gtrsim
			\nu^{d/2}\min\{1,(L\nu)^{-1/2}\},
			\qquad x\in E.
		\end{equation}
	\end{lemma}
	
	\begin{proof}
		This is the endpoint case \(r=\nu^{-1}\) of the envelope train.
		Since \((r/\nu)^{1/2}=\nu^{-1}\) and \(R^{-1/2}\le\nu^{-1}\), a single zonal function from Lemma~\ref{lem:coherent-short-packets} gives the result when
		\(L\nu\lesssim1\), after restricting its core to
		\[
		E=\{x(t,z): |t|\le cL,\ |z|\le cR^{-1/2}\}.
		\]
		Indeed \(|u|\gtrsim\nu^{d/2}\) there and \(|E|\approx LR^{-d/2}\).
		
		Assume now that \(L\nu\gg1\) and \(d\ge3\).  Choose times \(a_j\) in an interval
		of length \(\approx L\), separated by \(D\nu^{-1}\), with
		\[
		N\approx L\nu.
		\]

	\begin{figure}[htbp]
	\centering
	\includegraphics[width=0.8\textwidth]{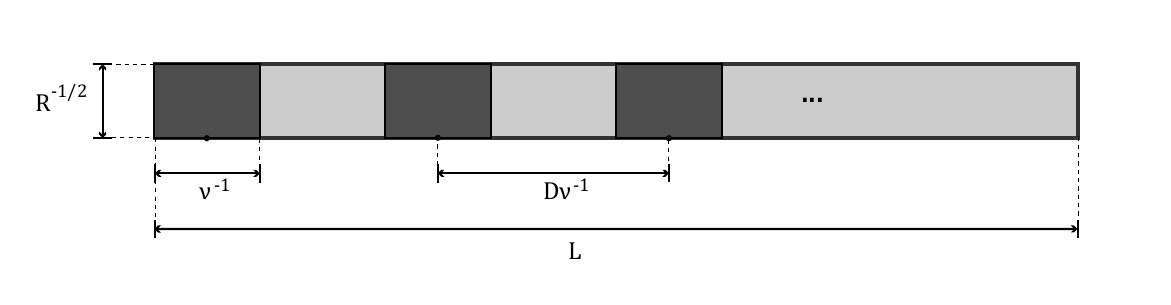}
	\caption{Zonal train: $R \geq \nu^2$.}
	\label{figzt}
	\end{figure}
		Let \(P_{a_j}\) be the corresponding rotations of \(u_{\nu,\nu^{-1}}\), and set
		\[
		Q=\sum_{j=1}^N P_{a_j},
		\qquad u=Q/\|Q\|_2.
		\]
		Let
		\[
		E=\bigcup_{j=1}^N
		\{x(t,z): |t-a_j|\le c\nu^{-1},\ |z|\le cR^{-1/2}\}.
		\]
		The pieces are disjoint and
		\[
		|E|\approx N\nu^{-1}R^{-d/2}\approx LR^{-d/2}.
		\]
		Lemma~\ref{lem:separated-short-packet-trains}, with \(r=\nu^{-1}\), gives
		\(\|Q\|_2\lesssim N^{1/2}\) and \(|Q|\gtrsim\nu^{d/2}\) on \(E\).  Therefore
		\[
		|u(x)|\gtrsim \nu^{d/2}N^{-1/2}
		\approx \nu^{d/2}(L\nu)^{-1/2},
		\qquad x\in E.
		\]
	\end{proof}

	For a fixed packet box with longitudinal length \(L\) and transverse radius
	\(R^{-1/2}\), set
	\begin{equation}\label{eq:packet-profile-general}
		A_\nu(L,R)=
		\begin{cases}
			R^{d/4}\min\{1,(LR/\nu)^{-1/2}\}, & 1\le R\le \nu^2,\\[1mm]
			\nu^{d/2}\min\{1,(L\nu)^{-1/2}\}, & R\ge \nu^2.
		\end{cases}
	\end{equation}
	
	\noindent\textbf{Definition of admissible triples.}
	Let \(0<L\lesssim1\), \(R\ge1\), and \(\nu\gg1\).  We call
	\((\nu,L,R)\) admissible if one of the following alternatives holds.
	\begin{enumerate}
		\item[\((\mathrm A_1)\)] \(1\le R\le\nu\).  This is the range of
		Lemma~\ref{lem:wide-transverse-block}, which gives the lower bound
		\(R^{d/4}\gtrsim A_\nu(L,R)\).
		
		\item[\((\mathrm A_2)\)] \(\nu\le R\le\nu^2\).  Put \(r=\nu/R\).  Assume
		either \(L\le r\), or \(d\ge3\).  This is exactly the hypothesis of
		Lemma~\ref{lem:intermediate-packet-profile}, which gives
		\[
		R^{d/4}\min\{1,(L/r)^{-1/2}\}
		=A_\nu(L,R).
		\]
		In the subcase \(L>r\), the proof of
		Lemma~\ref{lem:intermediate-packet-profile} uses
		Lemma~\ref{lem:separated-short-packet-trains}, and this is why \(d\ge3\) is
		required.
		
		\item[\((\mathrm A_3)\)] \(R\ge\nu^2\).  Assume either \(L\nu\lesssim1\),
		or \(d\ge3\).  This is exactly the hypothesis of
		Lemma~\ref{lem:zonal-train-profile}, which gives
		\[
		\nu^{d/2}\min\{1,(L\nu)^{-1/2}\}=A_\nu(L,R).
		\]
	\end{enumerate}
	For \(d=1,2\), alternatives \((\mathrm A_2)\) and \((\mathrm A_3)\) reduce to
	\(L\le\nu/R\) and \(L\nu\lesssim1\), respectively.
	
	\begin{lemma}[Common active set for admissible triples]\label{lem:common-refinement}
		Fix \(0<L\lesssim1\), \(R\ge1\), and frequencies
		\(\nu_1,\ldots,\nu_m\).  Assume that each \((\nu_i,L,R)\) is admissible.
		Then, after one rotation of \(S^n\), there is one coordinate box
		\[
		B_{L,R}=\{x(t,z):t\in I,\ z\in Q_R\},
		\qquad |I|=L,
		\qquad |Q_R|\approx R^{-d/2},
		\]
		normalized harmonics \(u_i\in\HH_{\nu_i}(S^n)\), and sets
		\(T\subset I\), \(Z\subset Q_R\), such that
		\[
		|T|\approx L,
		\qquad |Z|\approx R^{-d/2}.
		\]
		With \(E=x(T,Z)\), one has
		\[
		|E|\approx LR^{-d/2}
		\]
		and
		\begin{equation}\label{eq:common-refinement-lower}
			|u_i(x)|\gtrsim A_{\nu_i}(L,R),
			\qquad x\in E,
			\qquad i=1,\ldots,m .
		\end{equation}
		The constants may depend on \(m,d\) and on the fixed separation constants in
		Lemmas~\ref{lem:wide-transverse-block}--\ref{lem:zonal-train-profile}, but not
		on \(L,R\), or the frequencies.
	\end{lemma}

	\begin{proof}
		Work in one box \(B_{L,R}=\{x(t,z):t\in I,\ z\in Q_R\}\). A shift means the origin of a packet grid.  For a time scale \(\ell\le L\) and shift \(\tau\in[0,D\ell)\), set
		\[
		S_{\ell,\tau}=I\cap\bigcup_{j\in\mathbb Z}
		\{t:|t-(\tau+jD\ell)|\le c\ell\}.
		\]
		For a transverse scale \(\rho\le R^{-1/2}\) and shift \(\sigma\in[0,D\rho)^d\), set
		\[
		G_{\rho,\sigma}=Q_R\cap\bigcup_{k\in\mathbb Z^d}
		\bigl(\sigma+D\rho k+[-c\rho,c\rho]^d\bigr).
		\]
		Thus \(\tau\) translates a time train and \(\sigma\) translates a transverse lattice.  The required shifts are chosen by averaging.  For each fixed \(t\in I\) and \(z\in Q_R\),
		\[
		\mathbb{E}_\tau {\bf 1}_{S_{\ell,\tau}}(t):={1\over D\ell}\int_0^{D\ell}{\bf 1}_{S_{\ell,\tau}}(t)\,d\tau={2c\over D},
		\]
		\[ \mathbb{E}_\sigma {\bf 1}_{G_{\rho,\sigma}}(z):=
		{1\over (D\rho)^d}\int_{[0,D\rho)^d}{\bf 1}_{G_{\rho,\sigma}}(z)\,d\sigma
		=\left({2c\over D}\right)^d .\]
		Then for any fixed $A,B\le m$
		\[\mathbb{E}_{\tau_1,...,\tau_A}\left|I\cap\bigcap_{a=1}^A S_{\ell_a,\tau_a}\right|=\int_{I}\prod_{a=1}^A\mathbb{E}_{\tau_a}{\bf 1}_{S_{\ell_a,\tau_a}}(t)dt=\Big(\frac{2c}D\Big)^A|I|\approx L,\]
		\[\mathbb{E}_{\sigma_1,...,\sigma_B}\left|Q_R\cap\bigcap_{b=1}^B G_{\rho_b,\sigma_b}\right|=\int_{Q_R}\prod_{b=1}^B\mathbb{E}_{\sigma_b}{\bf 1}_{G_{\rho_b,\sigma_b}}(z)dz=\Big(\frac{2c}D\Big)^{dB}|Q_R|\approx R^{-d/2},\]
		Thus there are shifts $\tau_1,...,\tau_A$ and $\sigma_1,...,\sigma_B$ so that
		\begin{equation}\label{eq:grid-shift-choice}
			\left|I\cap\bigcap_{a=1}^A S_{\ell_a,\tau_a}\right|\gtrsim L,
			\qquad
			\left|Q_R\cap\bigcap_{b=1}^B G_{\rho_b,\sigma_b}\right|\gtrsim R^{-d/2}.
		\end{equation}
		
		For each frequency choose a good product set \(T_i\times Z_i\subset I\times Q_R\).  Full factors \(I\) and \(Q_R\) impose no restriction.  If \(1\le R\le\nu_i\), take \(\rho_i=\nu_i^{-1/2}\) and, by Lemma~\ref{lem:wide-transverse-block},
		\[
		T_i=I,
		\qquad Z_i=G_{\rho_i,\sigma_i},
		\qquad |u_i|\gtrsim R^{d/4}\gtrsim A_{\nu_i}(L,R)
		\quad \hbox{on } T_i\times Z_i .
		\]
		If \(\nu_i\le R\le\nu_i^2\), set \(r_i=\nu_i/R\).  For \(L\le r_i\), Lemma~\ref{lem:intermediate-packet-profile} gives the bound on all of \(I\times Q_R\).  For \(L>r_i\), admissibility gives \(d\ge3\), and the train with centers \(\tau_i+jDr_i\) gives
		\[
		T_i=S_{r_i,\tau_i},
		\qquad Z_i=Q_R,
		\qquad |u_i|\gtrsim R^{d/4}(L/r_i)^{-1/2}=A_{\nu_i}(L,R)
		\quad \hbox{on } T_i\times Z_i .
		\]
		If \(R\ge\nu_i^2\), Lemma~\ref{lem:zonal-train-profile} gives the bound on all of \(I\times Q_R\) when \(L\nu_i\lesssim1\).  For \(L\nu_i\gg1\), use the endpoint train with centers \(\tau_i+jD\nu_i^{-1}\), obtaining
		\[
		T_i=S_{\nu_i^{-1},\tau_i},
		\qquad Z_i=Q_R,
		\qquad |u_i|\gtrsim \nu_i^{d/2}(L\nu_i)^{-1/2}=A_{\nu_i}(L,R)
		\quad \hbox{on } T_i\times Z_i .
		\]
		
		Choose all nontrivial shifts by \eqref{eq:grid-shift-choice} and set
		\[
		T=I\cap\bigcap_i T_i,
		\qquad
		Z=Q_R\cap\bigcap_i Z_i,
		\]
		omitting full factors.  Then \(|T|\gtrsim L\) and \(|Z|\gtrsim R^{-d/2}\); after discarding subsets we may take comparable upper bounds as well.  With \(E=x(T,Z)\), the coordinate density gives \(|E|\approx LR^{-d/2}\).  Since \(T\times Z\subset T_i\times Z_i\) for every \(i\), the pointwise bounds above imply \eqref{eq:common-refinement-lower}.
	\end{proof}
	
	The following profile is the form of the preceding lemmas used in the sharpness proofs.
	
	\begin{corollary}[Model packet profile]\label{cor:packet-profile}
		Let \(1\ll\nu\), \(0<L\lesssim1\), and \(R\ge1\).  If \((\nu,L,R)\) is admissible, then one can find an \(L^2\)-normalized \(u\in\HH_\nu(S^n)\) and a set \(E\subset S^n\) such that
		\[
		|E|\approx LR^{-d/2},
		\qquad
		|u(x)|\gtrsim A_\nu(L,R)\quad (x\in E).
		\]
		More generally, if \((\nu_i,L,R)\), \(i=1,\ldots,m\), are admissible for one fixed pair \((L,R)\), then the \(u_i\in\HH_{\nu_i}(S^n)\) can be chosen with one common set \(E\) satisfying
		\[
		|E|\approx LR^{-d/2},
		\qquad
		|u_i(x)|\gtrsim A_{\nu_i}(L,R)
		\quad (x\in E,\ 1\le i\le m).
		\]
		For \(d\ge3\), every \((\nu,L,R)\) with \(0<L\lesssim1\) and \(R\ge1\) is admissible.  For \(d=1,2\), only the subcases allowed in \((\mathrm A_1)\)--\((\mathrm A_3)\) are used below.
	\end{corollary}
	
	\begin{proof}
		The one-function statement is Lemma~\ref{lem:wide-transverse-block}, Lemma~\ref{lem:intermediate-packet-profile}, or Lemma~\ref{lem:zonal-train-profile}, according as \((\mathrm A_1)\), \((\mathrm A_2)\), or \((\mathrm A_3)\) holds.  The common-set statement is Lemma~\ref{lem:common-refinement}.
	\end{proof}
	
	\begin{lemma}\label{lem:fixed-box-lower}
		Let \(1\le p\le\infty\).  Suppose \(E\subset S^n\) and \(u_i\in\HH_{\nu_i}(S^n)\), \(i=1,\ldots,m\), are \(L^2\)-normalized and satisfy
		\[
		|u_i(x)|\ge M_i,
		\qquad x\in E.
		\]
		Then
		\begin{equation}\label{eq:common-set-lower}
			\left\|\prod_{i=1}^m u_i\right\|_{L^{p}(S^n)}
			\ge |E|^{\frac1p}\prod_{i=1}^m M_i.
		\end{equation}
	\end{lemma}
	
	\begin{proof}
		If \(M=\prod_iM_i\), then \(\prod_i|u_i|\ge M\) on \(E\), and hence
		\[
		\left\|\prod_i u_i\right\|_{L^{p}}
		\ge \|M\|_{L^p(E)}=|E|^{\frac1p}M.
		\]
	\end{proof}
	
	\section{Proof of the sharpness}\label{sec12}
	In this section we prove the sharpness assertion in Theorem~\ref{thm4} on the round sphere.  By Lemmas \ref{soggelem}--\ref{ellip}, Sogge's spectral cluster parametrix is a Carleson--Sj\"olin operator, so together with Section \ref{sec2} this proves the sharpness assertions in Theorems \ref{thm1}--\ref{thm3}.

	For \(1\ll\lambda_1\le\cdots\le\lambda_k\), define
	\begin{equation}\label{eq:def-L-functional}
		\LL_{p,n}^{(k)}(\lambda_1,\ldots,\lambda_k)
		:=
		\sup_{\substack{u_j\in\HH_{\lambda_j}(S^n)\\ \|u_j\|_{L^2(S^n)}=1}}
		\left\|\prod_{j=1}^k u_j\right\|_{L^p(S^n)} .
	\end{equation}
	Empty products are interpreted as \(1\).  When a compact formula contains an index \(j\le0\), that dummy factor is omitted, and equivalently one may set \(\lambda_j=2\) for \(j\le0\), changing only constants.

	The proofs use the packet profiles from Section~\ref{sec11} and apply Corollary \ref{cor:packet-profile} and Lemma \ref{lem:fixed-box-lower}.  Let \(0<L\lesssim1\), \(R\ge1\), and \(\nu\gg1\). A packet box has longitudinal length
	\(L\) and transverse radius \(R^{-1/2}\):
	\[
	E_{L,R}=\{|s|\lesssim L,\ |y|\lesssim R^{-1/2}\},\qquad |E_{L,R}|\approx L R^{-d/2}.
	\]
	For a degree \(\nu\) spherical harmonic, the available lower-bound size on such a box is
	\[
	A_\nu(L,R)=
	\begin{cases}
		R^{d/4}\min\{1,(LR/\nu)^{-1/2}\}, & 1\le R\le \nu^2,\\[2mm]
		\nu^{d/2}\min\{1,(L\nu)^{-1/2}\}, & R\ge \nu^2.
	\end{cases}
	\]
	We use the following names for the model packet profiles in the different parameter regimes:
	\[
	\begin{array}{ccl}
		\text{beam}&:& R=\nu,\quad A_\nu(L,R)=\nu^{d/4},\\
		\text{beam block} &:& 1\le R\le \nu,\quad A_\nu(L,R)=R^{d/4},\\
		\text{envelope} &:& \nu\le R\le \nu^2,\ L\le \nu/R,\quad A_\nu(L,R)=R^{d/4},\\
		\text{envelope train} &:& \nu\le R\le \nu^2,\ L>\nu/R,\ d\ge 3,\quad
		A_\nu(L,R)=R^{d/4}(LR/\nu)^{-1/2},\\
		\text{zonal} &:& R\ge \nu^2,\ L\nu\lesssim 1,\quad A_\nu(L,R)=\nu^{d/2},\\
		\text{zonal train} &:& R\ge \nu^2,\ L\nu\gg 1,\ d\ge 3,\quad
		A_\nu(L,R)=\nu^{d/2}(L\nu)^{-1/2}.
	\end{array}
	\]
	For \(d\ge3\), all triples \((\nu,L,R)\) used below are admissible.  For \(d=1,2\), we only use the non-train subcases:
	\(L\le \nu/R\) in the envelope range and \(L\nu\lesssim1\) in the zonal range.

	\subsection{Sharpness for $1\le p\le 2$}

	\begin{proposition}[Bilinear sharpness]\label{prop:sharp-k2-bilinear}
	Let $p\le 2$.	The bilinear estimates in Theorem~\ref{thm4} are sharp.  
	\end{proposition}

	\begin{proof}
		Let $1\ll\la\le \mu$ and $d=n-1$. 
		We first use the box \(L\approx1\), \(R=\lambda\).  The $\la$-factor is a  beam that contributes 
		$A_\la(1,\la)= \la^{d/4}$,
		while  the $\mu$-factor is a beam block that contributes
			$A_\mu(1,\la)= \la^{d/4}.$
		 Since $|E_{L,R}|\approx \la^{-d/2}$, Lemma~\ref{lem:fixed-box-lower} gives
		\[
		\|uv\|_p\gtrsim \lambda^{d/2}\lambda^{-d/(2p)}
		=\lambda^{\frac d2(1-1/p)}.
		\]
		This agrees with the powers in Theorem~\ref{thm1} for \(d=1,2\), and the first range for \(d\ge3\).
		
		Assume \(d\ge3\).    If \(\mu\le\lambda^2\), use first \(L\approx1\), \(R=\mu\).  The \(\lambda\)-factor is an envelope train that contributes
		\[
		A_\lambda(1,\mu)=\mu^{d/4}(\mu/\lambda)^{-1/2}
		=\lambda^{1/2}\mu^{d/4-1/2},
		\]
		while the \(\mu\)-factor is a  beam that contributes \(A_\mu(1,\mu)=\mu^{d/4}\).  Since \(|E_{L,R}|\approx\mu^{-d/2}\),
		\[
		\|uv\|_p\gtrsim
		\lambda^{1/2}\mu^{\frac{d-1}{2}-\frac d{2p}}.
		\]
		In the same frequency regime, use the box
		\(L=\mu/\lambda^2\), \(R=\lambda^2\).  The 
		$\la$-factor is a zonal train that contributes
		\[A_\lambda(\mu/\la^2,\la^2)=\la^{d/2}(\mu/\la)^{-1/2}=\lambda^{(d+1)/2}\mu^{-1/2},\]
		while the $\mu$ factor is an envelope train that contributes
		\[
		A_\mu(\mu/\la^2,\la^2)=\lambda^{d/2}.
		\]
		Since $
		|E_{L,R}|\approx \mu\lambda^{-(d+2)}$, 
		\[
		\|uv\|_p\gtrsim
		\lambda^{\frac{2d+1}{2}-\frac{d+2}{p}}\mu^{\frac1p-\frac12}.
		\]
		If \(\mu\ge\lambda^2\), use \(L\approx1\), \(R=\lambda^2\).  Then the $\la$-factor is a zonal train that contributes $	A_\lambda(1,\lambda^2)=\lambda^{(d-1)/2},$
		while the $\mu$-factor is a beam block that contributes 
		$A_\mu(1,\lambda^2)=\lambda^{d/2}$.
	Since
		$|E_{L,R}|\approx\lambda^{-d}$,
		\[
		\|uv\|_p\gtrsim \lambda^{d-\frac12-\frac d p}.
		\]
		These are the two minimum branches in the remaining ranges. 
	\end{proof}

	\begin{proposition}[Multilinear sharpness]\label{prop:sharpness-main}
		Let $k\ge3$ and $p\le 2$. The $k$-linear estimates in Theorem~\ref{thm4} are sharp.
	\end{proposition}
	
	\begin{proof}
		For \(k\ge3\), use Propositions~\ref{prop:sharp-dge4}, \ref{prop:sharp-d3}, \ref{prop:sharp-d2}, and~\ref{prop:sharp-d1} below, according as \(n\ge5\), \(n=4\), \(n=3\), and \(n=2\).
	\end{proof}

	\begin{proposition}\label{prop:sharp-dge4}
		Let $n\ge5$, $k\ge3$, and $p\le 2$.  The $k$-linear estimates in Theorem~\ref{thm4} are sharp.
	\end{proposition}
	\begin{proof}
		Let $d=n-1.$
		Then \(d\ge 4\).  The lower-frequency factors that are not named explicitly below are chosen to be zonal functions from Lemma \ref{lem:coherent-short-packets}, with scale \(r=\lambda_j^{-1}\), all centered at the same point as the packet box. The boxes used below are contained in the region where these zonal functions have size \(\gtrsim \lambda_j^{d/2}\). Hence such filler factors contribute the product of the corresponding powers \(\lambda_j^{d/2}\).
		
		We first consider
		\[
		1\le p\le p_2=\frac{d+2}{d+1}.
		\]
		Set
		\[
		a=\lambda_{k-3},\qquad b=\lambda_{k-2},\qquad
		c=\lambda_{k-1},\qquad e=\lambda_k .
		\]
		The filler factors contribute \(\prod_{j=1}^{k-4}\lambda_j^{d/2}\).
		
		Assume first that $ac\le b^2.$
		Choose
		$
		L=c/b^2,\  R=b^2.$
		Then \(L\lesssim 1\), and
		\[
		|E_{L,R}|\approx \frac{c}{b^2}(b^2)^{-d/2}
		= c b^{-(d+2)}.
		\]
		We now list the sizes of the four active factors on \(E_{L,R}\).
		
		The \(a\)-factor is zonal and $A_a(L,R)=a^{d/2}.$ 
		
		The \(b\)-factor is a zonal train and $A_b(L,R)=b^{d/2}\left(\frac{c}{b}\right)^{-1/2}
		=b^{(d+1)/2}c^{-1/2}.$
		
		The \(c\)-factor is either a beam block or an envelope, and in both cases $	A_c(L,R)=R^{d/4}=b^{d/2}.$
		
		The \(e\)-factor is again a beam block or an envelope, and $A_e(L,R)=R^{d/4}=b^{d/2}.$
		
		Therefore, Lemma~\ref{lem:fixed-box-lower} gives
		\[
		\begin{aligned}
			L^{(k)}_{p,n}(\lambda_1,\ldots,\lambda_k)
			&\gtrsim
			\left(\prod_{j=1}^{k-4}\lambda_j^{d/2}\right)
			a^{d/2} b^{(3d+1)/2} c^{-1/2}
			\left(c b^{-(d+2)}\right)^{1/p}  \\
			&=
			\left(\prod_{j=1}^{k-4}\lambda_j^{d/2}\right)
			a^{d/2}
			b^{\frac{3d+1}{2}-\frac{d+2}{p}}
			c^{\frac1p-\frac12}.
		\end{aligned}
		\]
		This is the first branch in the first range of Theorem \ref{thm2}.
		
		Assume next that $ac\ge b^2.$
		Choose
		$
		L=a^{-1},\ R=ac .
		$
		Then
		\[
		|E_{L,R}|\approx a^{-1}(ac)^{-d/2}
		=a^{-1-d/2}c^{-d/2}.
		\]
		
		The \(a\)-factor is zonal and $
		A_a(L,R)=a^{d/2}.$
		
		The \(b\)-factor is a zonal train and 
		$
		A_b(L,R)=b^{d/2}\left(\frac{b}{a}\right)^{-1/2}
		=a^{1/2}b^{(d-1)/2}.
		$
		
		The \(c\)-factor is an envelope and 
		$
		A_c(L,R)=R^{d/4}=(ac)^{d/4}.
		$
		
		The \(e\)-factor is a beam block or an envelope, and 
		$
		A_e(L,R)=R^{d/4}=(ac)^{d/4}.
		$
		
		Therefore,
		Lemma~\ref{lem:fixed-box-lower} gives
		\[
		\begin{aligned}
			L^{(k)}_{p,n}(\lambda_1,\ldots,\lambda_k)
			&\gtrsim
			\left(\prod_{j=1}^{k-4}\lambda_j^{d/2}\right)
			a^{(2d+1)/2}b^{(d-1)/2}c^{d/2}
			\left(a^{-1-d/2}c^{-d/2}\right)^{1/p} \\
			&=
			\left(\prod_{j=1}^{k-4}\lambda_j^{d/2}\right)
			a^{\frac{2d+1}{2}-\frac{d+2}{2p}}
			b^{\frac{d-1}{2}}
			c^{\frac d2-\frac d{2p}}.
		\end{aligned}
		\]
		This is the second branch in the first range of Theorem \ref{thm2}. Moreover, the ratio of the first branch to the second branch is
		\[
		\left(\frac{b^2}{ac}\right)^{{(d+1)}/{2}-{(d+2)}/{(2p)}}.
		\]
		Since \(1\le p\le (d+2)/(d+1)\), the exponent is nonpositive, so the two subcases above select the smaller of the two branches.
		
		We now turn to the remaining ranges. Put
		\[
		a=\lambda_{k-2},\qquad b=\lambda_{k-1},\qquad c=\lambda_k .
		\]
		The filler factors contribute
		$
		\prod_{j=1}^{k-3}\lambda_j^{d/2}.
		$
		
		First suppose
		\[
		p_2\le p\le p_1=\frac{d}{d-1}.
		\]
		Choose
		$
		L=a^{-1},\  R=ab .
		$
		Then
		\[
		|E_{L,R}|\approx a^{-1}(ab)^{-d/2}
		=a^{-1-d/2}b^{-d/2}.
		\]
		
		The \(a\)-factor  is zonal and $A_a(L,R)=a^{d/2}.$
		
		The \(b\)-factor is an envelope and 
		$
		A_b(L,R)=R^{d/4}=(ab)^{d/4}.
		$
		
		The \(c\)-factor is a beam block or an envelope, and in either case
		$A_c(L,R)=R^{d/4}=(ab)^{d/4}.$

		Hence, 
		Lemma~\ref{lem:fixed-box-lower} yields
		\[
		\begin{aligned}
			L^{(k)}_{p,n}(\lambda_1,\ldots,\lambda_k)
			&\gtrsim
			\left(\prod_{j=1}^{k-3}\lambda_j^{d/2}\right)
			a^d b^{d/2}
			\left(a^{-1-d/2}b^{-d/2}\right)^{1/p} \\
			&=
			\left(\prod_{j=1}^{k-3}\lambda_j^{d/2}\right)
			a^{d-\frac{d+2}{2p}}
			b^{\frac d2-\frac d{2p}}.
		\end{aligned}
		\]
		This is precisely the branch in Theorem \ref{thm2} for \(p_2\le p\le p_1\).
		
		Next assume
		\[
		p_1\le p\le p_0=\frac{d+2}{d}.
		\]
		There are two frequency regimes.
		
		If $ac\le b^2,$
		choose
		$
		L=a^{-1},\  R=ac .$
		Then
		\[
		|E_{L,R}|\approx a^{-1}(ac)^{-d/2}
		=a^{-1-d/2}c^{-d/2}.
		\]
		The \(a\)-factor is zonal and  $A_a(L,R)=a^{d/2}.$
		The \(b\)-factor  is an envelope train and contributes
		\[
		A_b(L,R)=R^{d/4}\left(\frac{LR}{b}\right)^{-1/2}
		=(ac)^{d/4}\left(\frac{c}{b}\right)^{-1/2}
		=a^{d/4}b^{1/2}c^{d/4-1/2}.
		\]
		The \(c\)-factor  is an envelope and contributes
	$	A_c(L,R)=R^{d/4}=(ac)^{d/4}.$
	
		Consequently, Lemma \ref{lem:fixed-box-lower} gives
		\[
		\begin{aligned}
			L^{(k)}_{p,n}(\lambda_1,\ldots,\lambda_k)
			&\gtrsim
			\left(\prod_{j=1}^{k-3}\lambda_j^{d/2}\right)
			a^d b^{1/2}c^{(d-1)/2}
			\left(a^{-1-d/2}c^{-d/2}\right)^{1/p} \\
			&=
			\left(\prod_{j=1}^{k-3}\lambda_j^{d/2}\right)
			a^{d-\frac{d+2}{2p}}
			b^{1/2}
			c^{\frac{d-1}{2}-\frac d{2p}}.
		\end{aligned}
		\]
		This is the first branch in the range \(p_1\le p\le p_0\).
		
		If instead $ac\ge b^2,$
		choose
		$
		L=a^{-1},\  R=b^2.$
		Then
		\[
		|E_{L,R}|\approx a^{-1}b^{-d}.
		\]
		The \(a\)-factor is zonal and  $A_a(L,R)=a^{d/2}.$
		The \(b\)-factor  is a zonal train and contributes
		\[
		A_b(L,R)=b^{d/2}\left(\frac{b}{a}\right)^{-1/2}
		=a^{1/2}b^{(d-1)/2}.
		\]
		The \(c\)-factor is a beam block or an envelope, and $A_c(L,R)=R^{d/4}=b^{d/2}.$
		Thus
		 Lemma \ref{lem:fixed-box-lower} yields
		\[
		\begin{aligned}
			L^{(k)}_{p,n}(\lambda_1,\ldots,\lambda_k)
			&\gtrsim
			\left(\prod_{j=1}^{k-3}\lambda_j^{d/2}\right)
			a^{(d+1)/2}b^{d-1/2}
			\left(a^{-1}b^{-d}\right)^{1/p} \\
			&=
			\left(\prod_{j=1}^{k-3}\lambda_j^{d/2}\right)
			a^{\frac{d+1}{2}-\frac1p}
			b^{d-\frac12-\frac d p}.
		\end{aligned}
		\]
		This is the second branch in the range \(p_1\le p\le p_0\). The ratio of the first branch to the second branch is
		\[
		\left(\frac{ac}{b^2}\right)^{{(d-1)}/{2}-{d}/{(2p)}}.
		\]
		Since \(p\ge p_1=d/(d-1)\), the exponent is nonnegative, and the two subcases again select the smaller branch.
		
		It remains to prove the lower bounds for
		\[
		p_0\le p\le 2.
		\]
		We keep the notation \(a=\lambda_{k-2}\), \(b=\lambda_{k-1}\), and \(c=\lambda_k\), and the same filler contribution \(\prod_{j=1}^{k-3}\lambda_j^{d/2}\).
		
		If $ac\le b^2,$
		choose
		$
		L=c/b^2,\ R=b^2 .
		$
		Then
		\[
		|E_{L,R}|\approx c b^{-(d+2)}.
		\]
		The \(a\)-factor is zonal and
	$A_a(L,R)=a^{d/2}.$
		The \(b\)-factor is a zonal train,  and
		\[
		A_b(L,R)=b^{d/2}\left(\frac{c}{b}\right)^{-1/2}
		=b^{(d+1)/2}c^{-1/2}.
		\]
		The \(c\)-factor is a beam block or an envelope, and contributes $A_c(L,R)=R^{d/4}=b^{d/2}.$
		
		Hence,
		Lemma \ref{lem:fixed-box-lower} gives
		\[
		\begin{aligned}
			L^{(k)}_{p,n}(\lambda_1,\ldots,\lambda_k)
			&\gtrsim
			\left(\prod_{j=1}^{k-3}\lambda_j^{d/2}\right)
			a^{d/2} b^{(2d+1)/2} c^{-1/2}
			\left(c b^{-(d+2)}\right)^{1/p} \\
			&=
			\left(\prod_{j=1}^{k-3}\lambda_j^{d/2}\right)
			a^{d/2}
			b^{\frac{2d+1}{2}-\frac{d+2}{p}}
			c^{\frac1p-\frac12}.
		\end{aligned}
		\]
		This is the first branch in the range \(p_0\le p\le 2\).
		
		If $ac\ge b^2,$
		we use the same box as in the second subcase of the preceding range, namely
		$
		L=a^{-1},\  R=b^2 .
		$
		Then 
\[|E_{L,R}|\approx a^{-1}b^{-d}.\]
	The $a$-factor is zonal and $A_a(L,R)=a^{d/2}$. The $b$-factor is a zonal train and $$A_b(L,R)=b^{d/2}(b/a)^{-1/2}=a^{1/2}b^{(d-1)/2}.$$ 
	The $c$-factor is a beam block or an envelope, and 
		$A_c(L,R)=b^{d/2}$.
	Therefore Lemma \ref{lem:fixed-box-lower} gives
		\[
		\begin{aligned}
			L^{(k)}_{p,n}(\lambda_1,\ldots,\lambda_k)
			&\gtrsim
			\left(\prod_{j=1}^{k-3}\lambda_j^{d/2}\right)
			a^{(d+1)/2}b^{d-1/2}
			\left(a^{-1}b^{-d}\right)^{1/p} \\
			&=
			\left(\prod_{j=1}^{k-3}\lambda_j^{d/2}\right)
			a^{\frac{d+1}{2}-\frac1p}
			b^{d-\frac12-\frac d p}.
		\end{aligned}
		\]
		This is the second branch in the range \(p_0\le p\le 2\). In this range, the ratio of the first branch to the second branch is
		\[
		\left(\frac{ac}{b^2}\right)^{{1}/{p}-{1}/{2}},
		\]
		and the exponent is nonnegative because \(p\le 2\). Hence the two subcases select the smaller branch.
		
		Combining the four ranges, and recalling that \(d=n-1\), we obtain exactly the powers in Theorem \ref{thm2} for \(n\ge 5\). This proves the proposition.
	\end{proof}

	\begin{proposition}\label{prop:sharp-d3}
		Let $n=4$, $k\ge3$, and $p\le 2$. The $k$-linear estimates in Theorem~\ref{thm4} are sharp.
	\end{proposition}
	\begin{proof}
		Here \(d=3\).  As in the preceding proof, unnamed lower-frequency factors are chosen to be zonal functions from Lemma~\ref{lem:coherent-short-packets}, centered at the same packet box.  
		
		First assume
		\[
		1\le p\le \frac{10}{9}.
		\]
		Set
		\[
		a=\lambda_{k-3},\qquad b=\lambda_{k-2},\qquad c=\lambda_{k-1},\qquad e=\lambda_k .
		\]
		The filler factors contribute \(\prod_{j=1}^{k-4}\lambda_j^{3/2}\).
		
		Suppose first that \(ac\le b^2\).  Choose
		$
		L=a^{-1},\ R=ac .
		$
		Then
		\[
		|E_{L,R}|\approx a^{-1}(ac)^{-3/2}=a^{-5/2}c^{-3/2}.
		\]
		The \(a\)-factor is zonal and contributes \(A_a(L,R)=a^{3/2}\).  The \(b\)-factor is an envelope train and contributes
		\[
		A_b(L,R)=(ac)^{3/4}\left(\frac cb\right)^{-1/2}
		=a^{3/4}b^{1/2}c^{1/4}.
		\]
		The \(c\)- and \(e\)-factors are beam blocks or envelopes, and each contributes 
		\[A_c(L,R)=A_e(L,R)=(ac)^{3/4}\]  Hence
		\[
		\begin{aligned}
			\LL_{p,4}^{(k)}(\lambda_1,\ldots,\lambda_k)
			&\gtrsim
			\left(\prod_{j=1}^{k-4}\lambda_j^{3/2}\right)
			a^{3/2}\big(a^{3/4}b^{1/2}c^{1/4}\big)(ac)^{3/4}(ac)^{3/4}
			\left(a^{-5/2}c^{-3/2}\right)^{1/p} \\
			&=
			\left(\prod_{j=1}^{k-4}\lambda_j^{3/2}\right)
			a^{\frac{15}{4}-\frac{5}{2p}}
			b^{1/2}
			c^{\frac74-\frac{3}{2p}} .
		\end{aligned}
		\]
		This is the first branch in the first range of Theorem~\ref{thm2}.
		
		Suppose next that \(ac\ge b^2\).  Use the same box \(L=a^{-1}\), \(R=ac\).  The \(a\)-, \(c\)-, and \(e\)-factors are as above.  Now the \(b\)-factor is a zonal train and contributes
		\[
		A_b(L,R)=b^{3/2}\left(\frac ba\right)^{-1/2}=a^{1/2}b .
		\]
		Thus
		\[
		\begin{aligned}
			\LL_{p,4}^{(k)}(\lambda_1,\ldots,\lambda_k)
			&\gtrsim
			\left(\prod_{j=1}^{k-4}\lambda_j^{3/2}\right)
			a^{3/2}(a^{1/2}b)(ac)^{3/4}(ac)^{3/4}
			\left(a^{-5/2}c^{-3/2}\right)^{1/p} \\
			&=
			\left(\prod_{j=1}^{k-4}\lambda_j^{3/2}\right)
			a^{\frac72-\frac{5}{2p}}
			b
			c^{\frac32-\frac{3}{2p}} .
		\end{aligned}
		\]
		This is the second branch in the first range.  The ratio of the first branch to the second branch is
		\[
		\left(\frac{ac}{b^2}\right)^{1/4},
		\]
		so the two subcases select the smaller branch.
		
		We next assume
		\[
		\frac{10}{9}\le p\le \frac54,
		\]
		and keep the same notation \(a,b,c,e\).  If \(ac\le b^2\), choose
		$
		L=c/b^2,\ R=b^2 .
		$
		Then \(L\lesssim1\) and
		\[
		|E_{L,R}|\approx \frac{c}{b^2}(b^2)^{-3/2}=cb^{-5}.
		\]
		The \(a\)-factor is zonal and contributes \(A_a(L,R)=a^{3/2}\).  The \(b\)-factor is a zonal train and contributes
		\[
		A_b(L,R)=b^{3/2}\left(\frac cb\right)^{-1/2}=b^2c^{-1/2}.
		\]
		The \(c\)- and \(e\)-factors are beam blocks or envelopes, and each contributes 
		\[A_c(L,R)=A_e(L,R)=b^{3/2}.\] Therefore
		\[
		\begin{aligned}
			\LL_{p,4}^{(k)}(\lambda_1,\ldots,\lambda_k)
			&\gtrsim
			\left(\prod_{j=1}^{k-4}\lambda_j^{3/2}\right)
			a^{3/2}(b^2c^{-1/2})b^{3/2}b^{3/2}
			\left(cb^{-5}\right)^{1/p} \\
			&=
			\left(\prod_{j=1}^{k-4}\lambda_j^{3/2}\right)
			a^{3/2}
			b^{5-\frac5p}
			c^{\frac1p-\frac12} .
		\end{aligned}
		\]
		This is the first branch in the second range.
		
		If \(ac\ge b^2\), use again \(L=a^{-1}\), \(R=ac\). Then 
		\[|E_{L,R}|\approx a^{-5/2}c^{-3/2}. \]
		The $a$-factor is zonal and $A_a(L,R)=a^{3/2}$. The $b$-factor is a zonal train and $A_b(L,R)=a^{1/2}b$. The $c$-factor is an envelope and $A_c(L,R)=(ac)^{3/4}$. The $e$-factor is a beam block or an envelope, and $A_e(L,R)=(ac)^{3/4}$.
	
		 Hence
		\[
		\begin{aligned}
			\LL_{p,4}^{(k)}(\lambda_1,\ldots,\lambda_k)
			&\gtrsim
			\left(\prod_{j=1}^{k-4}\lambda_j^{3/2}\right)
			a^{\frac72}bc^{3/2}
			\left(a^{-5/2}c^{-3/2}\right)^{1/p} \\
			&=
			\left(\prod_{j=1}^{k-4}\lambda_j^{3/2}\right)
			a^{\frac72-\frac{5}{2p}}
			b
			c^{\frac32-\frac{3}{2p}} .
		\end{aligned}
		\]
		This is the second branch in the second range.  The ratio of the first branch to the second branch is
		\[
		\left(\frac{ac}{b^2}\right)^{\frac{5}{2p}-2}.
		\]
		Since \(10/9\le p\le5/4\), the exponent \(5/(2p)-2\) is nonnegative.  Thus the two subcases again select the smaller branch.
		
		For the remaining three ranges, set
		\[
		a=\lambda_{k-2},\qquad b=\lambda_{k-1},\qquad c=\lambda_k,
		\]
		and let the lower-frequency factors contribute \(\prod_{j=1}^{k-3}\lambda_j^{3/2}\).
		
		Assume first that
		\[
		\frac54\le p\le \frac32.
		\]
		Choose $L=a^{-1},\  R=ab .$
		Then
		\[
		|E_{L,R}|\approx a^{-1}(ab)^{-3/2}=a^{-5/2}b^{-3/2}.
		\]
		The \(a\)-factor is zonal and \(A_a(L,R)=a^{3/2}\).  The \(b\)- and \(c\)-factors are beam blocks or envelopes, and each contributes 
		\[A_b(L,R)=A_c(L,R)=(ab)^{3/4}.\]  Therefore
		\[
		\begin{aligned}
			\LL_{p,4}^{(k)}(\lambda_1,\ldots,\lambda_k)
			&\gtrsim
			\left(\prod_{j=1}^{k-3}\lambda_j^{3/2}\right)
			a^{3/2}(ab)^{3/4}(ab)^{3/4}
			\left(a^{-5/2}b^{-3/2}\right)^{1/p} \\
			&=
			\left(\prod_{j=1}^{k-3}\lambda_j^{3/2}\right)
			a^{3-\frac{5}{2p}}
			b^{\frac32-\frac{3}{2p}} .
		\end{aligned}
		\]
		This is the branch in the range \(5/4\le p\le3/2\).
		
		Next suppose that
		\[
		\frac32\le p\le \frac53.
		\]
		If \(ac\le b^2\), choose \(L=a^{-1}\), \(R=ac\).  Then $$|E_{L,R}|\approx a^{-5/2}c^{-3/2}.$$  The \(a\)-factor is zonal and contributes \(A_a(L,R)=a^{3/2}\). The \(b\)-factor is an envelope train and contributes
		\[
		A_b(L,R)=(ac)^{3/4}\left(\frac cb\right)^{-1/2}
		=a^{3/4}b^{1/2}c^{1/4}.
		\]
		The \(c\)-factor is an envelope and contributes \(A_c(L,R)=(ac)^{3/4}\).  Hence
		\[
		\begin{aligned}
			\LL_{p,4}^{(k)}(\lambda_1,\ldots,\lambda_k)
			&\gtrsim
			\left(\prod_{j=1}^{k-3}\lambda_j^{3/2}\right)
			a^{3/2}\big(a^{3/4}b^{1/2}c^{1/4}\big)(ac)^{3/4}
			\left(a^{-5/2}c^{-3/2}\right)^{1/p} \\
			&=
			\left(\prod_{j=1}^{k-3}\lambda_j^{3/2}\right)
			a^{3-\frac{5}{2p}}
			b^{1/2}
			c^{1-\frac{3}{2p}} .
		\end{aligned}
		\]
		This is the first branch in the range \(3/2\le p\le5/3\).
		
		If \(ac\ge b^2\), choose \(L=a^{-1}\), \(R=b^2\).  Then
		\[
		|E_{L,R}|\approx a^{-1}b^{-3}.
		\]
		The \(a\)-factor is zonal and contributes \(A_a(L,R)=a^{3/2}\). The \(b\)-factor is a zonal train and contributes
		\[
		A_b(L,R)=b^{3/2}\left(\frac ba\right)^{-1/2}=a^{1/2}b.
		\]
		The \(c\)-factor is a beam block or an envelope, and contributes \(A_c(L,R)=b^{3/2}\).  Therefore
		\[
		\begin{aligned}
			\LL_{p,4}^{(k)}(\lambda_1,\ldots,\lambda_k)
			&\gtrsim
			\left(\prod_{j=1}^{k-3}\lambda_j^{3/2}\right)
			a^{3/2}(a^{1/2}b)b^{3/2}
			\left(a^{-1}b^{-3}\right)^{1/p} \\
			&=
			\left(\prod_{j=1}^{k-3}\lambda_j^{3/2}\right)
			a^{2-\frac1p}
			b^{\frac52-\frac3p} .
		\end{aligned}
		\]
		This is the second branch in the range \(3/2\le p\le5/3\).  The ratio of the first branch to the second branch is
		\[
		\left(\frac{ac}{b^2}\right)^{1-\frac{3}{2p}}.
		\]
		Since \(p\ge3/2\), the exponent is nonnegative, and the two subcases select the smaller branch.
		
		It remains to consider
		\[
		\frac53\le p\le 2.
		\]
		If \(ac\le b^2\), choose
		$
		L=c/b^2,\  R=b^2 .
		$
		Then \(|E_{L,R}|\approx cb^{-5}\).  The \(a\)-factor is zonal and contributes \(A_a(L,R)=a^{3/2}\).  The \(b\)-factor is a zonal train and contributes
		\[
		A_b(L,R)=b^{3/2}\left(\frac cb\right)^{-1/2}=b^2c^{-1/2}.
		\]
		The \(c\)-factor is a beam block or an envelope, and contributes \(A_c(L,R)=b^{3/2}\).  Thus
		\[
		\begin{aligned}
			\LL_{p,4}^{(k)}(\lambda_1,\ldots,\lambda_k)
			&\gtrsim
			\left(\prod_{j=1}^{k-3}\lambda_j^{3/2}\right)
			a^{3/2}(b^2c^{-1/2})b^{3/2}
			\left(cb^{-5}\right)^{1/p} \\
			&=
			\left(\prod_{j=1}^{k-3}\lambda_j^{3/2}\right)
			a^{3/2}
			b^{\frac72-\frac5p}
			c^{\frac1p-\frac12} .
		\end{aligned}
		\]
		This is the first branch in the range \(5/3\le p\le2\).
		
		If \(ac\ge b^2\), use the same box as in the second subcase of the preceding range, namely \(L=a^{-1}\), \(R=b^2\).  The \(a\)-factor is zonal and contributes \(A_a(L,R)=a^{3/2}\).  The \(b\)-factor is a zonal train and contributes
		\[
		A_b(L,R)=b^{3/2}(b/a)^{-1/2}=a^{1/2}b.
		\]
		The \(c\)-factor is a beam block or an envelope, and contributes \(A_c(L,R)=b^{3/2}\).

		Thus, 
		\[
		\begin{aligned}
			\LL_{p,4}^{(k)}(\lambda_1,\ldots,\lambda_k)
			&\gtrsim
			\left(\prod_{j=1}^{k-3}\lambda_j^{3/2}\right)
			a^{2}b^{5/2}
			\left(a^{-1}b^{-3}\right)^{1/p} \\
			&=
			\left(\prod_{j=1}^{k-3}\lambda_j^{3/2}\right)
			a^{2-\frac1p}
			b^{\frac52-\frac3p} .
		\end{aligned}
		\]
		This is the second branch in the range \(5/3\le p\le2\).  The ratio of the first branch to the second branch is
		\[
		\left(\frac{ac}{b^2}\right)^{\frac1p-\frac12}.
		\]
		Since \(p\le2\), the exponent is nonnegative, so the two subcases select the smaller branch.
		
		Combining the five ranges gives exactly the \(n=4\) powers in Theorem~\ref{thm2}.  This proves the proposition.
	\end{proof}

	\begin{proposition}\label{prop:sharp-d2}
		Let $n=3$, $k\ge3$, and $p\le 2$. The $k$-linear estimates in Theorem~\ref{thm4} are sharp.
	\end{proposition}
	\begin{proof}
		Here \(d=2\).  We use only the admissible non-train subcases from the packet profile: in the envelope range \(\nu\le R\le \nu^2\) we always have \(L\le \nu/R\), and in the zonal range \(R\ge \nu^2\) we always have \(L\nu\lesssim1\).  Thus the geometric objects which occur below are zonals, beam blocks, and envelopes, but not trains.  The unnamed lower-frequency factors are chosen to be zonal functions from Lemma~\ref{lem:coherent-short-packets}, centered at the same point as the packet boxes.  
		
		First suppose that
		\[
		1\le p\le \frac43 .
		\]
		Set
		\[
		a=\lambda_{k-3},\qquad b=\lambda_{k-2},\qquad c=\lambda_{k-1},\qquad e=\lambda_k .
		\]
		We use $L=a^{-1},\ R=ac$. 
		Then
		\[
		|E_{L,R}|\approx a^{-2}c^{-1}.
		\]
		The filler factors contribute \(\prod_{j=1}^{k-4}\lambda_j\).
		The \(a\)-factor is zonal and $A_a(L,R)=a$.
		The \(b\)-factor is an envelope at the coarser transverse scale \(R'=ab\le R\).  Indeed \(b\le ab\le b^2\) and
		\[
		L=a^{-1}=\frac{b}{ab}=b/R'.
		\]
		Thus this is the endpoint envelope case, and
		\[
		A_b(L,R')=(ab)^{1/2}.
		\]
		Since \((ac)^{-1/2}\le (ab)^{-1/2}\), the same lower bound holds after restricting to the smaller set \(E_{L,R}\).
		The \(c\)-factor is an envelope 
		and $A_c(L,R)=(ac)^{1/2}.$
		Finally, the \(e\)-factor is a beam block or an envelope,  and 	$A_e(L,R)=(ac)^{1/2}$.
		Hence, 
		Lemma~\ref{lem:fixed-box-lower} gives
		\[
		\begin{aligned}
			\LL_{p,3}^{(k)}(\lambda_1,\ldots,\lambda_k)
			&\gtrsim
			\left(\prod_{j=1}^{k-4}\lambda_j\right)
			a^{5/2}b^{1/2}c\,
			\left(a^{-2}c^{-1}\right)^{1/p} \\
			&=
			\left(\prod_{j=1}^{k-4}\lambda_j\right)
			a^{\frac52-\frac2p}
			b^{1/2}
			c^{1-\frac1p}.
		\end{aligned}
		\]
		Substituting back \(a=\lambda_{k-3}\), \(b=\lambda_{k-2}\), and \(c=\lambda_{k-1}\), this is exactly the power \(\E_0\) in Theorem~\ref{thm2} for \(1\le p\le4/3\).
		
		We now suppose that
		\[
		\frac43\le p\le 2 .
		\]
		Set
		\[
		a=\lambda_{k-2},\qquad b=\lambda_{k-1},\qquad c=\lambda_k .
		\]
		We use
		$
		L=a^{-1},\ R=ab.
		$
		Thus
		\[
		|E_{L,R}|\approx a^{-1}(ab)^{-1}=a^{-2}b^{-1}.
		\]
	The filler factors contribute \(\prod_{j=1}^{k-3}\lambda_j\).
		The \(a\)-factor is zonal and
	$A_a(L,R)=a.$
		The \(b\)-factor is an envelope and 
		$
		A_b(L,R)=(ab)^{1/2}.$
		The \(c\)-factor is a beam block or an envelope, and 
	$A_c(L,R)=(ab)^{1/2}.$
		Consequently, 
		Lemma~\ref{lem:fixed-box-lower} gives
		\[
		\begin{aligned}
			\LL_{p,3}^{(k)}(\lambda_1,\ldots,\lambda_k)
			&\gtrsim
			\left(\prod_{j=1}^{k-3}\lambda_j\right)
			a^2b\,
			\left(a^{-2}b^{-1}\right)^{1/p} \\
			&=
			\left(\prod_{j=1}^{k-3}\lambda_j\right)
			a^{2-\frac2p}
			b^{1-\frac1p}.
		\end{aligned}
		\]
		Since \(a=\lambda_{k-2}\) and \(b=\lambda_{k-1}\), this is exactly the power \(\E_0\) in Theorem~\ref{thm2} for \(4/3\le p\le2\).
		
		Combining the two ranges proves the sharp lower bounds for the powers in the case \(n=3\).  This proves the proposition.
	\end{proof}

	\begin{proposition}\label{prop:sharp-d1}
		Let $n=2$, $k\ge3$, and $p\le 2$. The $k$-linear estimates in Theorem~\ref{thm4} are sharp.
	\end{proposition}
	\begin{proof}
		Here \(d=1\).  We use only the admissible non-train cases from the packet profile.  Thus the geometric objects which occur below are zonals, envelopes, and beam blocks; no envelope trains or zonal trains are used.  All packets are centered on the same geodesic, and the unnamed lower-frequency factors are chosen to be zonal functions from Lemma~\ref{lem:coherent-short-packets}.  
		
		Choose \(m\in\{4,5,6\}\) by
		\[
		m=6\quad \left(1\le p\le \frac65\right),
		\qquad
		m=5\quad \left(\frac65\le p\le \frac32\right),
		\qquad
		m=4\quad \left(\frac32\le p\le2\right).
		\]
		Set
		\[
		a=k-m+1,
		\qquad
		\alpha=\lambda_a,
		\qquad
		\omega=\lambda_{k-1}.
		\]
		We use 	$L=\alpha^{-1},
		R=\alpha\omega.$
		Thus, 
		\[
		|E_{L,R}|\approx \alpha^{-1}(\alpha\omega)^{-1/2}
		=\alpha^{-3/2}\omega^{-1/2}.
		\]
			The filler factors are zonal and contribute \(\prod_{j=1}^{a-1}\lambda_j\).
		The \(\alpha\)-factor is also zonal and 
		$A_\alpha(L,R)=\alpha^{1/2}.$ 
		
		For the intermediate factors \(a+1\le j\le k-2\), we use the larger tube with the same longitudinal length and transverse parameter
		\[
		R_j=\alpha\lambda_j.
		\]
		Since \(\alpha\le\lambda_j\), we have
		\[
		\lambda_j\le R_j=\alpha\lambda_j\le\lambda_j^2,
		\qquad
		L=\alpha^{-1}=\frac{\lambda_j}{\alpha\lambda_j}=\frac{\lambda_j}{R_j}.
		\]
		Hence each such factor is an endpoint envelope, and
		\[
		A_{\lambda_j}(L,R_j)
		=(\alpha\lambda_j)^{1/4},
		\qquad a+1\le j\le k-2.
		\]
		Because \(\lambda_j\le\omega\), the tube defining \(E_{L,R}\) is contained in this larger envelope tube. Therefore the same lower bound holds on \(E_{L,R}\).
		
		 The \(\omega\)-factor is an endpoint envelope, and $	A_\omega(L,R)=(\alpha\omega)^{1/4}.$ The $\la_k$-factor is a beam block or an envelope, and $A_{\lambda_k}(\alpha^{-1},\alpha\omega)=(\alpha\omega)^{1/4}.$

	Thus, Lemma~\ref{lem:fixed-box-lower} gives
		\[
		\begin{aligned}
			\LL_{p,2}^{(k)}(\lambda_1,\ldots,\lambda_k)
			&\gtrsim
		\left(\prod_{j=1}^{a-1}\lambda_j^{1/2}\right)
		\alpha^{1/2}
		\left(\prod_{j=a+1}^{k-2}(\alpha\lambda_j)^{1/4}\right)
		(\alpha\omega)^{1/4}(\alpha\omega)^{1/4}
			\left(\alpha^{-3/2}\omega^{-1/2}\right)^{1/p} \\
			&=
			\left(\prod_{j=1}^{a-1}\lambda_j^{1/2}\right)
			\alpha^{\frac{m+1}{4}-\frac{3}{2p}}
			\left(\prod_{j=a+1}^{k-2}\lambda_j^{1/4}\right)
			\omega^{\frac12-\frac1{2p}}.
		\end{aligned}
		\]
		
		Taking \(m=6\), so that \(a=k-5\) and \(\alpha=\lambda_{k-5}\), gives for \(1\le p\le6/5\)
		\[
		\LL_{p,2}^{(k)}
		\gtrsim
		\left(\prod_{j=1}^{k-6}\lambda_j^{1/2}\right)
		\lambda_{k-5}^{\frac74-\frac{3}{2p}}
		\lambda_{k-4}^{1/4}
		\lambda_{k-3}^{1/4}
		\lambda_{k-2}^{1/4}
		\lambda_{k-1}^{\frac12-\frac1{2p}}.
		\]
		Taking \(m=5\), so that \(a=k-4\) and \(\alpha=\lambda_{k-4}\), gives for \(6/5\le p\le3/2\)
		\[
		\LL_{p,2}^{(k)}
		\gtrsim
		\left(\prod_{j=1}^{k-5}\lambda_j^{1/2}\right)
		\lambda_{k-4}^{\frac32-\frac{3}{2p}}
		\lambda_{k-3}^{1/4}
		\lambda_{k-2}^{1/4}
		\lambda_{k-1}^{\frac12-\frac1{2p}}.
		\]
		Taking \(m=4\), so that \(a=k-3\) and \(\alpha=\lambda_{k-3}\), gives for \(3/2\le p\le2\)
		\[
		\LL_{p,2}^{(k)}
		\gtrsim
		\left(\prod_{j=1}^{k-4}\lambda_j^{1/2}\right)
		\lambda_{k-3}^{\frac54-\frac{3}{2p}}
		\lambda_{k-2}^{1/4}
		\lambda_{k-1}^{\frac12-\frac1{2p}}.
		\]
		These are exactly the three powers in Theorem~\ref{thm2} for \(n=2\).  This proves  the proposition.
	\end{proof}
		\subsection{Sharpness for $p>2$}\label{subsec:large-p-sharpness}
	
	\begin{proposition}\label{prop:sharp-large-p}
		Let $n\ge2$, $k\ge2$, and $p>2$. The $k$-linear estimates in Theorem~\ref{thm4} are sharp.
	\end{proposition}
	
	\begin{proof}
		The proof of sharpness for $p>2$ only needs the envelopes in Lemma~\ref{lem:coherent-short-packets}. For \(\nu^{-1}\le r\ll1\), there is an \(L^2\)-normalized spherical harmonic satisfying
		\[
		|u_{\nu,r}|\gtrsim (\nu/r)^{d/4}
		\]
		on a tube of length \(r\), transverse radius \((r/\nu)^{1/2}\), and measure \(\approx r(r/\nu)^{d/2}\).  Put all envelopes on the same geodesic.  If \(r_1\ge\cdots\ge r_k\) and \(r_j\ge\lambda_j^{-1}\), then the smallest tube gives
		\begin{equation}\label{eq:large-p-master-lower}
			\LL_{p,n}^{(k)}(\lambda_1,\ldots,\lambda_k)
			\gtrsim
			\prod_{j=1}^k\left(\frac{\lambda_j}{r_j}\right)^{d/4}
			\left[r_k\left(\frac{r_k}{\lambda_k}\right)^{d/2}\right]^{1/p}.
		\end{equation}
		
		For \(n=2\) and \(2\le p\le3\), take
		\[
		r_j=\lambda_j^{-1}\ (j\le k-2),
		\qquad r_{k-1}=r_k=\lambda_{k-2}^{-1}.
		\]
		Then \eqref{eq:large-p-master-lower} gives
		\[
		\LL_{p,2}^{(k)}
		\gtrsim
		\Big(\prod_{j=1}^{k-3}\lambda_j^{1/2}\Big)
		\lambda_{k-2}^{1-\frac{3}{2p}}
		\lambda_{k-1}^{1/4}
		\lambda_k^{1/4-\frac1{2p}}.
		\]
		
		Next take
		\[
		r_j=\lambda_j^{-1}\ (j\le k-1),
		\qquad r_k=\lambda_{k-1}^{-1}.
		\]
		For every \(d\ge1\), this yields
		\[
		\LL_{p,n}^{(k)}
		\gtrsim
		\Big(\prod_{j=1}^{k-2}\lambda_j^{d/2}\Big)
		\lambda_{k-1}^{\frac{3d}{4}-\frac{d+2}{2p}}
		\lambda_k^{\frac d4-\frac d{2p}}.
		\]
		Finally, taking \(r_j=\lambda_j^{-1}\) for all \(j\) gives
		\[
		\LL_{p,n}^{(k)}
		\gtrsim
		\Big(\prod_{j=1}^{k-1}\lambda_j^{d/2}\Big)
		\lambda_k^{d/2-\frac{d+1}{p}}.
		\]
		
	\end{proof}

	\bibliographystyle{plain}
	
\end{document}